\documentclass[10pt,a4paper]{amsart}

\numberwithin{equation}{section}
\allowdisplaybreaks % let equations break across pages (doesn't work with groups like aligned)

\usepackage{mathtools,amssymb,eucal,mathrsfs,setspace,graphicx} % mathtools needed for cases* and brackets; loads amsmath
\usepackage[noadjust]{cite} % stops spaces being inserted before citations
\usepackage[margin=22mm,foot=10mm]{geometry}

\usepackage{array}
\newcolumntype{C}{>{$}c<{$}} %Defines math mode in tabular

\usepackage[largesc,theoremfont]{newtxtext}      % mathptmx is now considered obsolete...
\usepackage[libertine,cmbraces,varbb]{newtxmath} % but the replacement greek math looks crap (hence libertine)...
\begingroup
  \makeatletter
  \@for\theoremstyle:=definition,remark,plain\do{%
    \expandafter\g@addto@macro\csname th@\theoremstyle\endcsname{%
      \addtolength\thm@preskip{.5\baselineskip plus .2\baselineskip minus .2\baselineskip}
      \addtolength\thm@postskip{.5\baselineskip plus .2\baselineskip minus .2\baselineskip}
    }%
  }
\endgroup
\usepackage{enumitem}
\setitemize{leftmargin=*}   % Removes margin from itemized lists (enumitem package)
\setenumerate{leftmargin=*, % Removes margin for enumerate lists (enumitem package)
  label=\textup{(\alph*)},  % Makes labels upright lowercase letters
	ref=\textup{\alph*}}      % Makes refs upright lowercase letters without parentheses

\usepackage{tikz}
\usetikzlibrary{arrows}

\usepackage[colorlinks=true,citecolor=red,linkcolor=blue]{hyperref} % load _before_ cleveref!

\usepackage[capitalise,noabbrev]{cleveref} % needed for \cref and \Cref (almost all refs except \eqref)

\newcommand{\eps}{\varepsilon}

\newcommand{\pd}{\partial}     % holomorphic partial d

\renewcommand{\ge}{\geqslant} % never use \geq or \geqslant, use \ge which is globally defined to be one or the other
\renewcommand{\le}{\leqslant} % same for \leq and \leqslant (and we should use \ne instead of \neq for consistency I guess)

\renewcommand{\cong}{\simeq} % I like this symbol better...

\DeclareMathOperator{\tr}{tr} % why does latex define \det but not \tr ?

\newcommand{\cc}{\mathsf{c}}   % central charge
\newcommand{\ee}{\mathsf{e}}   % ln e = 1
\newcommand{\ii}{\mathsf{i}}   % imaginary unit
\newcommand{\kk}{\mathsf{k}}   % level
\newcommand{\uu}{\mathsf{u}}   % minmod parameter
\providecommand{\vv}{\mathsf{v}}\renewcommand{\vv}{\mathsf{v}}   % minmod parameter [DR both needed as miktex thinks \vv is predefined]
\newcommand{\wun}{\vvmathbb{1}}  % here used to indicate characteristic functions (requires newtxmath)

\DeclarePairedDelimiter{\brac}{\lparen}{\rparen}   % use \brac for (...) and \brac* to automatically scale the ( and )
\DeclarePairedDelimiter{\sqbrac}{\lbrack}{\rbrack} % use \sqbrac[\big] for \bigl(...\bigr) etc...
\DeclarePairedDelimiter{\set}{\lbrace}{\rbrace}
\DeclarePairedDelimiter{\abs}{\lvert}{\rvert}

\newcommand{\no}[1]{\mathopen{:} #1 \mathclose{:}} % normal ordering (prevent := or =:)

\DeclarePairedDelimiterX{\comm}[2]{\lbrack}{\rbrack}{#1 , #2}  % commutators
\DeclarePairedDelimiterX{\acomm}[2]{\lbrace}{\rbrace}{#1 , #2} % anticommutators
\DeclarePairedDelimiterX{\inner}[2]{\langle}{\rangle}{#1 , #2} % scalar products
\DeclarePairedDelimiterX{\super}[2]{\lparen}{\rparen}{#1 \delimsize\vert \mathopen{} #2} % for super args (m|n)

\DeclarePairedDelimiter{\ket}{\lvert}{\rangle}
\newcommand{\ghvac}{\ket{0}}                    % ghost vacuum

\newcommand{\lra}{\longrightarrow}
\newcommand{\ira}{\hookrightarrow}    % for injections
\newcommand{\lma}{\overset{\sfsymb^{1/2}}{\longmapsto}} % spectral flow orbit

\newcommand{\dses}[5]{0 \lra #1 \overset{#2}{\lra} #3 \overset{#4}{\lra} #5 \lra 0} % displayed ses

\newcommand{\blank}{{-}}

\newcommand{\fld}[1]{\mathbb{#1}}    % for fields and related things
\newcommand{\alg}[1]{\mathfrak{#1}}  % for Lie algebras
\newcommand{\grp}[1]{\mathsf{#1}}    % for groups
\newcommand{\aalg}[1]{\mathsf{#1}}   % for associative algebras
\newcommand{\Mod}[1]{\mathcal{#1}}   % modules
\newcommand{\VOA}[1]{\mathsf{#1}}    % VOAs
\newcommand{\categ}[1]{\mathscr{#1}} % categories
\newcommand{\ZZ}{\fld{Z}}
\newcommand{\NN}{\ZZ_{\ge 0}} %\fld{N}} %D modified to eliminate confusion among uneducated people

\newcommand{\CC}{\fld{C}}

\newcommand{\affine}[1]{\widehat{#1}}

\newcommand{\SLA}[2]{\alg{#1}_{#2}}                      % Lie algebras like sl(2)
\newcommand{\SLSA}[3]{\alg{#1} \super{#2}{#3}}           % Lie superalgebras like gl(1|1)
\newcommand{\AKMA}[2]{\affine{\alg{#1}}_{#2}}            % Kac-Moody algebras
\newcommand{\sltwo}{\SLA{sl}{2}}
\newcommand{\slthree}{\SLA{sl}{3}}

\newcommand{\aslthree}{\AKMA{sl}{3}}

\newcommand{\pwlat}[1]{\grp{P}^{#1}_{\ge}}                           % dominant integral weights
\newcommand{\proots}{\Delta_+}                                       % positive roots
\newcommand{\sroot}[1]{\alpha_{#1}}                                  % simple root
\newcommand{\hroot}{\theta}                                          % highest root
\newcommand{\imroot}{\delta}                                         % imaginary root
\newcommand{\rvec}[1]{e^{#1}}                                        % root vector
\newcommand{\cvec}[1]{h^{#1}}                                        % Cartan basis vector
\newcommand{\nrvec}[1]{f^{#1}}                                       % negative root vector
\newcommand{\fwt}[1]{\omega_{#1}}                                    % fundamental weights
\newcommand{\wgrp}{\grp{W}}                                          % Weyl group
\newcommand{\wref}[1]{w_{#1}}                                        % Weyl reflection

\newcommand{\surv}[1]{\Sigma_{#1}}                                   % level k surviving weights
\newcommand{\survk}{\surv{\uu,\vv}}
\newcommand{\infwts}[1]{\Gamma_{#1}}                                 % level k surviving weights giving infdim top
\newcommand{\infwtsk}{\infwts{\uu,\vv}}

\newcommand{\bpsymb}{\VOA{BP}}

\newcommand{\uaffvoa}[2]{\VOA{V}^{#1}\brac{#2}}                   % universal affine VOA #1 = level, #2 = g
\newcommand{\saffvoa}[2]{\VOA{L}_{#1}\brac{#2}}                   % simple affine VOA #1 = level, #2 = g
\newcommand{\uslvoa}[1]{\uaffvoa{#1}{\slthree}}                   % universal sl_3 VOA
\newcommand{\ubpvoa}[1]{\bpsymb^{#1}}                             % universal BP VOA
\newcommand{\sslvoa}[1]{\saffvoa{#1}{\slthree}}                   % simple sl_3 VOA
\newcommand{\sbpvoa}[1]{\bpsymb_{#1}}                             % simple BP VOA
\newcommand{\slminmod}[2]{\VOA{A}_2\brac{#1,#2}}                  % simple fractional-level sl_3 VOA
\newcommand{\sltwominmod}[2]{\VOA{A}_1\brac{#1,#2}}               % simple fractional-level sl_2 VOA
\newcommand{\bpminmod}[2]{\bpsymb\brac{#1,#2}}                    % simple fractional-level BP VOA

\newcommand{\uslvoak}{\uslvoa{\kk}}                               % universal sl_3 VOA
\newcommand{\ubpvoak}{\ubpvoa{\kk}}                               % universal BP VOA
\newcommand{\sslvoak}{\sslvoa{\kk}}                               % simple sl_3 VOA
\newcommand{\sbpvoak}{\sbpvoa{\kk}}                               % simple BP VOA
\newcommand{\slminmoduv}{\slminmod{\uu}{\vv}}                     % general simple fractional-level sl_3 VOA
\newcommand{\bpminmoduv}{\VOA{BP}\brac{\uu,\vv}}                  % general simple fractional-level BP VOA

\newcommand{\bpmpi}[1]{\VOA{J}^{#1}}                              % maximal ideal of universal BP
\newcommand{\slmpi}[1]{\VOA{I}^{#1}}                              % maximal ideal of universal sl_3
\newcommand{\bpmpik}{\bpmpi{\kk}}
\newcommand{\slmpik}{\slmpi{\kk}}

\newcommand{\fgvoa}[1]{\VOA{F}^{#1}}                              % fermionic ghost system
\newcommand{\bgvoa}{\VOA{B}}                                      % bosonic ghost system
\newcommand{\ghvoa}{\VOA{G}}                                      % product of ghost systems

\newcommand{\twist}{\textup{tw}}                          % symbol for twisted things
\newcommand{\semisimp}{\textup{ss}}                       % symbol for semisimplified things
\newcommand{\semi}[1]{#1^{\semisimp}}

\newcommand{\modealg}{\aalg{U}}                           % mode algebra #1 = k
\newcommand{\twmodealg}{\modealg^{\twist}}

\newcommand{\envalg}[1]{\aalg{U}\brac{#1}}                     % Universal enveloping (super)algebras
\newcommand{\zhuq}[1]{\sqbrac[\big]{#1}}                       % quotient in definition of Zhu algebra

\newcommand{\zhu}[1]{\mathsf{Zhu}\sqbrac*{#1}}                 % Zhu's algebra/functor
\newcommand{\twzhu}[1]{\mathsf{Zhu}^{\twist}\sqbrac[\big]{#1}} % twisted Zhu algebra
\newcommand{\ind}[1]{\mathsf{Ind}\sqbrac*{#1}}                 % Inducing a Zhu module back to a VOA module
\newcommand{\twind}[1]{\mathsf{Ind}^{\twist}\sqbrac[\big]{#1}} % Inducing a twisted Zhu module back to a twisted VOA module

\newcommand{\twaalg}[1]{\aalg{Z}_{#1}}                    % "Smith"-like associative algebra covering the twisted Zhu algebras
\newcommand{\zk}{\twaalg{\kk}}
\newcommand{\ck}{\aalg{C}_{\kk}}                          % centraliser of the Cartan subalgebra in \zk
\newcommand{\ak}{\aalg{A}_{\kk}}                          % intersection algebra appearing in relaxed classification proof

\newcommand{\zversion}[1]{\overline{#1}}                  % version for the Zhu algebra

\newcommand{\conjsymb}{\gamma}                            % conjugation
\newcommand{\zconjsymb}{\zversion{\gamma}}                % "conjugation" for twisted Zhu algebra
\newcommand{\sfsymb}{\sigma}                              % spectral flow

\newcommand{\conjmod}[1]{\conjsymb(#1)}                   % conjugate module
\newcommand{\zconjmod}[1]{\zconjsymb(#1)}                 % conjugate Zhu module
\newcommand{\sfmod}[2]{\sfsymb^{#1}(#2)}                  % spectrally flowed module

\renewcommand{\top}[1]{#1^{\textup{top}}}                 % top space of a module

\newcommand{\irrsymb}{\Mod{H}}
\newcommand{\versymb}{\Mod{V}}
\newcommand{\relsymb}{\Mod{R}}
\newcommand{\rvsymb}{\Mod{W}}
\newcommand{\maxsymb}{\Mod{M}}
\newcommand{\naxsymb}{\Mod{N}}
\newcommand{\cohsymb}{\Mod{C}}
\newcommand{\slsymb}{\Mod{L}}
\newcommand{\slvsymb}{\Mod{K}}

\newcommand{\ihw}[1]{\irrsymb_{#1}}                            % irreducible hwmod for bp
\newcommand{\vhw}[1]{\versymb_{#1}}                            % Verma hwmod
\newcommand{\twihw}[1]{\irrsymb^{\twist}_{#1}}                 % twisted irreducible hwmod
\newcommand{\twvhw}[1]{\versymb^{\twist}_{#1}}                 % twisted Verma hwmod
\newcommand{\twrhw}[1]{\relsymb^{\twist}_{#1}}                 % twisted relaxed hwmod

\newcommand{\twprhw}[1]{\relsymb^{\twist,+}_{#1}}              % nonsemisimple twisted relaxed hwmod
\newcommand{\twmrhw}[1]{\relsymb^{\twist,-}_{#1}}
\newcommand{\twpmrhw}[1]{\relsymb^{\twist,\pm}_{#1}}
\newcommand{\twprv}[1]{\rvsymb^{\twist,+}_{#1}}                % twisted generalised Verma mod

\newcommand{\twpmrv}[1]{\rvsymb^{\twist,\pm}_{#1}}
\newcommand{\twpmax}[1]{\maxsymb^{\twist,+}_{#1}}              % maximal submodule

\newcommand{\twpmmax}[1]{\maxsymb^{\twist,\pm}_{#1}}
\newcommand{\twpnax}[1]{\naxsymb^{\twist,+}_{#1}}              % near-maximal submodule

\newcommand{\twpmnax}[1]{\naxsymb^{\twist,\pm}_{#1}}

\newcommand{\zihw}[1]{\zversion{\irrsymb}_{#1}}                % irreducible hw Z_k-mod
\newcommand{\zvhw}[1]{\zversion{\versymb}_{#1}}                % Verma Z_k-mod
\newcommand{\zrhw}[1]{\zversion{\relsymb}_{#1}}                % relaxed hw Z_k-mod
\newcommand{\zssrhw}[1]{\semi{\zrhw{#1}}}                      % semisimplified rhw Z_k-mod
\newcommand{\zprhw}[1]{\zrhw{#1}^{\,+}}                        % G^+ bijective rhw Z_k-mod
\newcommand{\zmrhw}[1]{\zrhw{#1}^{\,-}}                        % G^- bijective rhw Z_k-mod
\newcommand{\zpmrhw}[1]{\zrhw{#1}^{\,\pm}}
\newcommand{\zcfam}[1]{\zversion{\cohsymb}_{#1}}               % generic coherent families
\newcommand{\zsscfam}[1]{\semi{\zcfam{#1}}}                    % semisimple coherent families
\newcommand{\zpcfam}[1]{\zcfam{#1}^{\,+}}                      % G^+ bijective coherent families
\newcommand{\zmcfam}[1]{\zcfam{#1}^{\,-}}                      % G^m bijective coherent families
\newcommand{\zpmcfam}[1]{\zcfam{#1}^{\,\pm}}

\newcommand{\slihw}[1]{\slsymb_{#1}}                       % irreducible hwmod for L_k(sl_3)
\newcommand{\slvhw}[1]{\slvsymb_{#1}}                      % Verma for L_k(sl_3)

\newcommand{\mult}[2]{\bigl[ #1 : #2 \bigr]}               % multiplicity of composition factor #2 in #1

\newcommand{\wcat}[1]{\categ{W}_{#1}}             % category of weight modules (with fdim weight spaces)
\newcommand{\twcat}[1]{\categ{W}^{\twist}_{#1}}   % twisted version

\newcommand{\slocat}[1]{\affine{\categ{O}}_{#1}}  % category O for affine sl_3 at level k
\newcommand{\bpocat}[1]{\categ{O}_{#1}}           % category O for BP at level k

\newcommand{\qhrfun}[1]{H^0(#1)}                  % reduction functor

\newcommand{\cft}{conformal field theory}

\newcommand{\uea}{universal enveloping algebra}

\newcommand{\lw}{lowest-weight}
\newcommand{\lwv}{\lw\ vector}
\newcommand{\lwvs}{\lwv s}
\newcommand{\lwm}{\lw\ module}
\newcommand{\lwms}{\lwm s}
\newcommand{\hw}{highest-weight}
\newcommand{\hwv}{\hw\ vector}
\newcommand{\hwvs}{\hwv s}
\newcommand{\hwm}{\hw\ module}
\newcommand{\hwms}{\hwm s}
\newcommand{\rhw}{relaxed highest-weight}
\newcommand{\rhwv}{\rhw\ vector}

\newcommand{\rhwm}{\rhw\ module}
\newcommand{\rhwms}{\rhwm s}
\newcommand{\sv}{singular vector}
\newcommand{\svs}{\sv s}
\newcommand{\vo}{vertex operator}
\newcommand{\voa}{\vo\ algebra}
\newcommand{\voas}{\voa s}

\newcommand{\svoa}{\vo\ superalgebra}
\newcommand{\svoas}{\svoa s}
\newcommand{\ope}{operator product expansion}
\newcommand{\opes}{\ope s}
\newcommand{\qhr}{quantum hamiltonian reduction} %D yes the _h_ is lowercase
\newcommand{\qhrs}{\qhr s}
\newcommand{\emt}{energy-momentum tensor}

\newcommand{\lhs}{left-hand side}

\newcommand{\rhs}{right-hand side}

\newcommand{\fdim}{finite-dimensional}
\newcommand{\infdim}{infinite-dimensional}

\newcommand{\bgg}{Bern\v{s}te\u{\i}n--Gel'fand--Gel'fand}
\newcommand{\bp}{Bershadsky--Polyakov}
\newcommand{\km}{Kac--Moody}

\newcommand{\pbw}{Poincar\'{e}--Birkhoff--Witt}

\theoremstyle{plain}
\newtheorem{theorem}{Theorem}[section]
\newtheorem{corollary}[theorem]{Corollary}
\newtheorem{lemma}[theorem]{Lemma}
\newtheorem{proposition}[theorem]{Proposition}

\newtheorem{mthm}{Main Theorem}

\newtheorem{definition}[theorem]{Definition}

\newtheorem{assumption}{Assumption}

\Crefname{assumption}{Assumption}{Assumptions}

\makeatletter
\renewcommand\author@andify{%
  \nxandlist {\unskip ,\penalty-1 \space\ignorespaces}%
    {\unskip {} \@@and~}%
    {\unskip \penalty-2 \space \@@and~}%
}
\makeatother

\begin{document}

\title[Relaxed highest-weight modules for the Bershadsky--Polyakov algebras]{Classifying relaxed highest-weight modules for \\ admissible-level Bershadsky--Polyakov algebras}

\author[Z~Fehily]{Zachary Fehily}
\address[Zachary Fehily]{
School of Mathematics and Statistics \\
University of Melbourne \\
Parkville, Australia, 3010.
}
\email{zfehily@student.unimelb.edu.au}

\author[]{Kazuya Kawasetsu}
\address[Kazuya Kawasetsu]{
Priority Organization for Innovation and Excellence \\
Kumamoto University \\
Kumamoto 860-8555, Japan.
}
\email{kawasetsu@kumamoto-u.ac.jp}

\author[D~Ridout]{David Ridout}
\address[David Ridout]{
School of Mathematics and Statistics \\
University of Melbourne \\
Parkville, Australia, 3010.
}
\email{david.ridout@unimelb.edu.au}

\begin{abstract}
	The Bershadsky--Polyakov algebras are the minimal quantum hamiltonian reductions of the affine vertex algebras associated to $\mathfrak{sl}_3$ and their simple quotients have a long history of applications in conformal field theory and string theory.  Their representation theories are therefore quite interesting.  Here, we classify the simple relaxed highest-weight modules, with \fdim\ weight spaces, for all admissible but nonintegral levels, significantly generalising the known highest-weight classifications \cite{AraRat13,AdaCla19}.  In particular, we prove that the simple Bershadsky--Polyakov algebras with admissible nonintegral $\mathsf{k}$ are always rational in category $\mathscr{O}$, whilst they always admit nonsemisimple relaxed highest-weight modules unless $\mathsf{k}+\frac{3}{2} \in \mathbb{Z}_{\ge0}$.
\end{abstract}

\maketitle

\markleft{Z~FEHILY, K~KAWASETSU AND D~RIDOUT} % don't know how to remove the oxford comma in the running head

\onehalfspacing

\section{Introduction} \label{sec:intro}

\subsection{Background} \label{sec:back}

The \bp\ algebras $\ubpvoak$, $\kk \in \CC$, are among the simplest and best-known nonregular W-algebras \cite{PolGau90,BerCon91}.  They may be characterised \cite{KacQua03} as the minimal (or subregular) \qhrs\ of the level-$\kk$ universal affine vertex algebras $\uslvoak$.  Here, we are interested in their representation theories and, in particular, those of their simple quotients $\sbpvoak$.

When $\kk+\frac{3}{2} \in \NN$, $\sbpvoak$ is known to be rational and $C_2$-cofinite \cite{AraAss15,AraRat13}, meaning that the representation theory is semisimple and that there are finitely many simple $\sbpvoak$-modules, up to isomorphism.  More recently, the representation theory of $\sbpvoak$ was explored for certain other levels in \cite{AdaCla19,AdaBer20}.  There, the \hwms\ were classified and some non\hwms\ were described.  These works both relied on explicit formulae for \svs\ in $\ubpvoak$.  Here, we shall extend these classifications to more general levels where the \sv\ method is unavailable.  Instead, we shall exploit the properties \cite{KacQua04,AraRep05} of minimal \qhr.

In particular, we are interested in the \rhw\ theory of the simple \bp\ algebras $\sbpvoak$.  Relaxed \hwms\ are a type of generalised \hwm\ \cite{AdaVer95,FeiEqu98,RidRel15} that have been shown to be essential to achieve consistent modular properties for many nonrational \voas, for example the admissible-level affine ones associated with $\sltwo$ \cite{AdaVer95,GabFus01,RidSL210,CreMod12,CreMod13,RidRel15,AugMod17,AdaRea17,KawRel18}, their affine cousins \cite{AraWei16,RidAdm17,AdaRea17,KawRel18,WooAdm18,CreCos18,KawRel19,KawRel20,FutSim20,FutPos20} and other close relatives \cite{AdaRea16,RidBos14}.  We therefore expect them to play a central role in \bp\ representation theory and, indeed, in the representation theory of most nonrational W-algebras.  This will be discussed further in \cite{FR20}.

Here, we classify the simple \rhw\ $\sbpvoak$-modules with \fdim\ weight spaces, in both the untwisted and twisted sectors, when $\kk$ is admissible and nonintegral.  The much more difficult nonadmissible and integral cases are left for future investigations.  This classification includes the \hw\ classification as a special case.  We also show that there are nonsemisimple \rhw\ $\sbpvoak$-modules when $\kk$ is admissible, nonintegral and $2\kk+3 \notin \NN$.  In a companion paper \cite{AdaRea20}, these relaxed modules are constructed from the \hwms\ of the Zamolodchikov algebra \cite{ZamInf85}, the regular W-algebra associated to $\slthree$, using the inverse \qhr\ procedure of \cite{SemInv94,AdaRea17}.  This results in beautiful character formulae for the relaxed $\sbpvoak$-modules, generalising those found in \cite{CreMod13,KawRel18} for $\saffvoa{\kk}{\sltwo}$ and $\saffvoa{\kk}{\SLSA{osp}{1}{2}}$.

\subsection{Results} \label{sec:results}

Our strategy in classifying \rhw\ $\sbpvoak$-modules starts from the \hw\ classification.  The idea for the latter is to use Arakawa's celebrated results on minimal \qhr\ \cite{AraRep05}.  However, we must first establish a subtle technical result concerning the surjectivity of the minimal reduction functor.  This is the content of our first main result.
\begin{mthm}[\cref{thm:surjection}] \label{mthm1}
	Let $\kk$ be an admissible nonintegral level.  Then, every simple (untwisted) \hw\ $\sbpvoak$-module may be realised as the minimal \qhr\ of a simple \hw\ $\sslvoak$-module.
\end{mthm}
\noindent In \cite{AraRat13}, an explicit \sv\ formula is used to prove this theorem when $2\kk+3\in\NN$.  Our general proof also uses the existence of a generating \sv, but is necessarily very different because an explicit formula is unavailable.

Given this result, it is straightforward to classify the simple (untwisted and twisted) \hw\ $\sbpvoak$-modules and determine how they are related to one another.  For this, write $\kk+3 = \frac{\uu}{\vv}$, where $\uu\ge3$ and $\vv\ge2$ are coprime, and introduce the set $\survk$ of $\aslthree$ weights $\lambda = \lambda^I - \frac{\uu}{\vv} \lambda^F$ satisfying $\lambda^I \in \pwlat{\uu-3}$, $\lambda^F \in \pwlat{\vv-1}$ and $\lambda^F_0 \ne 0$.  Here, $\pwlat{\ell}$ denotes the dominant integral weights of $\aslthree$ whose level is $\ell$.
\begin{mthm} \label{mthm2}
	Let $\kk$ be admissible and nonintegral.  Then:
	\begin{enumerate}
		\item \label{it1} [\cref{thm:hwclass}] The isomorphism classes of the simple untwisted and twisted \hw\ $\sbpvoak$-modules, $\ihw{\lambda}$ and $\twihw{\lambda}$, are each in bijection with $\survk$.  The connection between the $\aslthree$ weights $\lambda \in \survk$ and the native $\sbpvoak$ data is given explicitly in \cref{eq:bphwfromsl3,eq:bptwhwfromsl3}.
		\item \label{it2} [\cref{thm:hwrat}] Every (untwisted or twisted) \hw\ $\sbpvoak$-module is simple, so $\sbpvoak$ is rational in the \bgg\ category $\bpocat{\kk}$.
		\item \label{it3} [\cref{prop:hwconj}] The module conjugate to $\ihw{\lambda}$, $\lambda \in \survk$, is $\ihw{\mu}$, where $\mu = [\lambda_0,\lambda_2,\lambda_1] \in \survk$.  The module conjugate to $\twihw{\lambda}$ is \hw\ if and only if $\lambda^F_1 = 0$, in which case it is $\twihw{\nu}$, where $\nu = [\lambda_2 - \frac{\uu}{\vv},\lambda_1,\lambda_0+\frac{\uu}{\vv}] \in \survk$.
		\item \label{it4} [\cref{prop:sfhw}] The spectral flow of the untwisted (twisted) \hwm\ labelled by $\lambda \in \survk$ is \hw\ if and only if $\lambda^F_1 = 0$, in which case it is the untwisted (twisted) \hwm\ labelled by $[\lambda_2 - \frac{\uu}{\vv},\lambda_0+\frac{\uu}{\vv},\lambda_1] \in \survk$.
	\end{enumerate}
\end{mthm}
\noindent This then generalises the \hw\ classifications of \cite{AraRat13}, when $2\kk+3\in\NN$, and \cite{AdaCla19}, for $\kk=-\frac{5}{3}$ and $-\frac{9}{4}$.  We refer to \cref{sec:bpsf} for an introduction to the conjugation and spectral flow functors referred to above.

To extend the \hw\ classification to simple twisted \rhwms, with \fdim\ weight spaces, we adapt the methodology developed in \cite{KawRel19} for affine vertex algebras.  This uses Mathieu's notion of a coherent family \cite{MatCla00}, extending it from semisimple Lie algebras to the twisted Zhu algebra of $\ubpvoak$.  Let $\infwtsk$ consist of the $\aslthree$ weights $\lambda \in \survk$ satisfying $\lambda^F_1 \ne 0$.  Writing $\kk+3 = \frac{\uu}{\vv}$ as above, it follows that $\infwtsk$ is empty unless $\vv\ge3$.  Moreover, $\infwtsk$ admits a free $\ZZ_3$-action generated by $\lambda \mapsto [\lambda_2 - \frac{\uu}{\vv},\lambda_0,\lambda_1 + \frac{\uu}{\vv}]$ (\cref{lem:Z3action}).
\begin{mthm}[\cref{quotientsofcoherents}] \label{mthm3}
	Let $\kk$ be admissible and nonintegral.  Then:
	\begin{enumerate}
		\item The isomorphism classes of the simple twisted \rhw\ $\sbpvoak$-modules $\twrhw{[j],\lambda}$, each of which have \fdim\ weight spaces, form families that are in bijection with $\infwtsk / \ZZ_3$.  The connection between the $\aslthree$ weights $\lambda \in \infwtsk$ and the native $\sbpvoak$ data is given explicitly in \cref{eq:bptwhwfromsl3,eq:bprhwfromsl3}.
		\item The members of each of these families are indexed by all but three of the cosets $[j] \in \CC / \ZZ$, the exceptions being determined as the images of the $\ZZ_3$-orbit of $\lambda$ under \eqref{eq:bptwhwfromsl3}.
		\item The module conjugate to $\twrhw{[j],\lambda}$ is $\twrhw{[-j],\mu}$, where $\mu = [\lambda_2 - \frac{\uu}{\vv},\lambda_0+\frac{\uu}{\vv},\lambda_1] \in \infwtsk$.
	\end{enumerate}
\end{mthm}
\noindent Moreover, the spectral flow of each $\twrhw{[j],\lambda}$ is never a \rhwm.

Our final main result relates to the existence of nonsemisimple \rhw\ $\sbpvoak$-modules when $\vv\ge3$.  Roughly speaking, these ``fill in'' the three ``holes'' in the allowed values of $[j]$ in each family of simple relaxed modules above.  However, there are two ways of filling each hole, each way related to the other by taking contragredient duals.  This is very similar to the analogous nonsemisimple picture conjectured in \cite{CreMod13,RidRel15}, and proven in \cite{AdaRea17,KawRel18}, for $\saffvoa{\kk}{\sltwo}$.  In the case at hand, we establish this picture by combining a mix of the theory developed in \cite{KawRel18,KawRel19} with the rationality of $\sbpvoak$ in category $\bpocat{\kk}$ (\cref{thm:hwrat}).  This seems robust and we expect it to generalise to higher-rank cases.
\begin{mthm}[\cref{thm:nonssstructure}] \label{mthm4}
	Let $\kk$ be admissible and nonintegral.  Then:
	\begin{enumerate}
		\item Every $\lambda \in \infwtsk$ defines two indecomposable nonsemisimple \rhw\ $\sbpvoak$-modules $\twprhw{[j],\lambda}$ and $\twmrhw{[j],\lambda}$, where $j$ is determined from $\lambda$ by \eqref{eq:bptwhwfromsl3}.
		\item $\twprhw{[j],\lambda}$ has a submodule isomorphic to the conjugate of $\twihw{\mu}$, where $\mu = [\lambda_0,\lambda_2 - \frac{\uu}{\vv},\lambda_1 + \frac{\uu}{\vv}] \in \infwtsk$, and its quotient by this submodule is isomorphic to $\twihw{\lambda}$.  The structure of $\twmrhw{[j],\lambda}$ is similar, but with submodule and quotient exchanged.
	\end{enumerate}
\end{mthm}
\noindent Conjugation and spectral flow works as for the simple relaxed modules, except that the conjugate of a $+$-type module is of $-$-type (and vice versa).  These nonsemisimple modules prove that $\sbpvoak$ has a nonsemisimple module category, for $\kk$ admissible and nonintegral.  It follows that $\sbpvoak$ is neither rational nor $C_2$-cofinite for these levels.  Nevertheless, these nonsemisimple modules are, along with their spectral flows, the building blocks (the ``atypical standards'') for the resolutions that underpin the so-called \emph{standard module formalism} \cite{CreLog13,RidVer14} for modular transformations and Verlinde formulae for nonrational \voas.  We intend to return to this in a forthcoming paper \cite{FR20}.

\subsection{Outline} \label{sec:outline}

We start by defining the universal \bp\ \voas\ $\ubpvoak$ and their simple quotients $\sbpvoak$ in \cref{sec:bpvoa}.  It is worthwhile remarking that we choose the conformal structure so that the charged generators $G^{\pm}$ have equal conformal weight $\frac{3}{2}$.  Equivalently, the Heisenberg field is a Virasoro primary.  Accordingly, we study both untwisted and twisted $\ubpvoak$-modules.  \cref{sec:bpsf} then introduces the all-important conjugation and spectral flow automorphisms and explains how they lift to invertible functors of appropriate categories of $\ubpvoak$-modules.  Happily, the untwisted and twisted sectors of the categories of interest are related by spectral flow.

In \cref{sec:ubpmods}, we embark on the first part of the journey: to understand how to identify $\ubpvoak$-modules, untwisted and twisted, relaxed and \hw.  After defining these types of modules, we introduce Zhu algebras and determine that of $\ubpvoak$ in \cref{prop:untwzhu}.  This leads to an easy classification of untwisted \hw\ $\ubpvoak$-modules (\cref{thm:untwubpclass}).  The more-involved twisted classification (\cref{thm:classtwrhw}) is then detailed.  For this, we review the identification \cite{AraRat13} of the twisted Zhu algebra with a central extension of a Smith algebra \cite{SmiCla90} (\cref{prop:twzhu}) and classify the simple weight modules, with \fdim\ weight spaces, of this extension in \cref{thm:classzkmod}.  For later use, we also introduce coherent families of modules, following \cite{MatCla00}, for the twisted Zhu algebra.

The hard work then begins in \cref{sec:sbpmods} where we convert these classification results for the universal \bp\ algebras $\ubpvoak$ into the corresponding results for their simple quotients $\sbpvoak$.  \cref{sec:sl3} is devoted to \cref{mthm1}, first reviewing the \hw\ theory of the simple affine \voa\ $\sslvoak$ \cite{KacMod88,AraRat16} and some basic, though deep, results about minimal \qhr\ \cite{KacQua03,KacQua04,AraRep05}.  The actual proof of this crucial result is deferred to \cref{app:proof} so as not to disrupt the flow of the arguments too much.

From this, we immediately deduce the classification of \hw\ $\sbpvoak$-modules, as in \cref{mthm2}.  The remainder of \cref{sec:hwmod} then addresses how the \hwms\ are related by the conjugation and spectral flow functors.  This will be important for the standard module analysis in \cite{FR20}.  \cref{sec:rhwmod} then lifts this classification to simple \rhw\ $\sbpvoak$-modules, establishing \cref{mthm3}.  The existence of nonsemisimple \rhwms, hence \cref{mthm4}, is the focus of \cref{sec:indec}.

In \cref{sec:ex}, we conclude by illustrating our classification results with some examples.  The rational cases with $\vv=2$ were already investigated in \cite{AraRat13}, so here we content ourselves with a quick overview of the ``smallest'' nontrivial example $\sbpvoa{-1/2}$ and the slightly more involved example $\sbpvoa{3/2}$.  The latter is interesting because it has a simple current extension that may be regarded as a bosonic analogue of the $N=4$ superconformal algebra.  In particular, it has three fields of conformal weight $1$, generating a subalgebra isomorphic to $\saffvoa{1}{\SLA{sl}{2}}$, and four weight $\frac{3}{2}$ fields.

We also study three nonrational examples.  Two, namely $\sbpvoa{-9/4}$ and $\sbpvoa{-5/3}$, were already discussed in \cite{AdaCla19} and here we take the opportunity to explicitly extend their \hw\ classifications to the full relaxed classifications.  We finish by describing the example $\sbpvoa{-4/3}$ which we believe has not been analysed before.  After describing its \rhwms\ explicitly, we note an interesting fact: it seems to admit a simple current extension isomorphic to the minimal \qhr\ of $\saffvoa{-3/2}{\alg{g}_2}$.  It follows then that this $\alg{g}_2$ W-algebra should have a $\ZZ_3$-orbifold isomorphic to $\sbpvoa{-4/3}$, as well as a $\ZZ_2$-orbifold isomorphic to $\saffvoa{1/2}{\SLA{sl}{2}}$ \cite{CreMod13,AdaCon17}.

\subsection*{Acknowledgements}

We thank Dra\v{z}en Adamovi\'{c}, Tomoyuki Arakawa and Thomas Creutzig for interesting discussions related to the research reported here.

ZF's research is supported by an Australian Government Research Training Program (RTP) Scholarship.

KK's research is partially supported by MEXT Japan ``Leading Initiative for Excellent Young Researchers (LEADER)'', JSPS Kakenhi Grant numbers 19KK0065 and 19J01093 and Australian Research Council Discovery Project DP160101520.

DR's research is supported by the Australian Research Council Discovery Project DP160101520 and the Australian Research Council Centre of Excellence for Mathematical and Statistical Frontiers CE140100049.

\section{\bp\ algebras} \label{sec:bp}

\subsection{\bp\ \voas} \label{sec:bpvoa}

We begin by defining one of the families of \voas\ that we will study here.
\begin{definition}
	Given $\kk \in \CC$, $\kk \ne -3$, the \emph{level-$\kk$ universal \bp\ algebra} $\ubpvoak$ is the \voa\ with vacuum $\wun$ that is strongly and freely generated by fields $J(z)$, $G^+(z)$, $G^-(z)$ and $L(z)$ satisfying the following \opes{}:
	\begin{equation} \label{ope:bp}
		\begin{gathered}
	    L(z)L(w) \sim -\frac{(2\kk+3)(3\kk+1) \wun}{2 (\kk+3) (z-w)^4} + \frac{2L(w)}{(z-w)^2} + \frac{\partial L(w)}{(z-w)}, \\
	    L(z)J(w) \sim \frac{J(w)}{(z-w)^2} + \frac{\partial J(w)}{(z-w)}, \qquad
	    L(z)G^\pm(w) \sim \frac{\frac{3}{2}G^\pm(w)}{(z-w)^2} + \frac{\partial G^\pm(w)}{(z-w)},  \\
	    J(z)J(w) \sim  \frac{(2\kk+3) \wun}{3(z-w)^2}, \qquad
	    J(z)G^\pm(w) \sim  \frac{\pm G^\pm (w)}{(z-w)}, \qquad
	    G^\pm(z)G^\pm(w) \sim 0, \\
			G^+(z)G^-(w) \sim \frac{(\kk+1)(2\kk+3) \wun}{(z-w)^3} + \frac{3(\kk+1) J(w)}{(z-w)^2} + \frac{3 \no{JJ}(w) + \frac{3}{2} (\kk+1) \partial J(w) - (\kk+3) L(w)}{z-w}.
		\end{gathered}
	\end{equation}
\end{definition}
\noindent This family of \voas\ was first described in \cite{PolGau90,BerCon91} where it was constructed via a new type of \qhr\ from the corresponding family of universal affine \voas\ $\uslvoak$ associated to $\slthree$.  In the general framework of \qhrs\ \cite{KacQua03}, $\ubpvoak$ is the \emph{minimal} reduction corresponding to taking the nilpotent of $\slthree$ to be a root vector.

From \eqref{ope:bp}, we see that the conformal weights of the generating fields $J(z)$, $G^+(z)$, $G^-(z)$ and $L(z)$ are $1$, $\frac{3}{2}$, $\frac{3}{2}$ and $2$, respectively, whilst the central charge is
\begin{equation} \label{eq:bpc}
	\cc = -\frac{(2\kk+3)(3\kk+1)}{\kk+3}.
\end{equation}
We shall expand the homogeneous fields of $\ubpvoak$ in the form
\begin{equation}
	A(z) = \sum_{n \in \ZZ - \Delta_A + \eps_A} A_n z^{-n-\Delta_A},
\end{equation}
where $\Delta_A$ is the conformal weight of $A(z)$ and $\eps_A = \frac{1}{2}$, if $\Delta_A \in \ZZ+\frac{1}{2}$ and $A(z)$ is acting on a twisted $\ubpvoak$-module (see \cref{sec:ubpmods} below), and $\eps_A = 0$ otherwise.  Standard computations now give the mode relations.
\begin{proposition}
	The commutation relations of the modes of the generating fields of $\ubpvoak$ are
	\begin{equation} \label{cr:bp}
		\begin{gathered}
	    \comm{L_m}{L_n} = (m-n)L_{m+n} - \frac{(2\kk+3)(3\kk+1)}{\kk+3} \frac{m^3-m}{12} \delta_{m+n,0} \wun, \\
	    \comm{L_m}{J_n} = -nJ_{m+n}, \qquad
	    \comm{L_m}{G^\pm_s} = \brac*{\frac{m}{2}-s} G^\pm_{m+s}, \\
			\comm{J_m}{J_n} = \frac{2\kk+3}{3}m \delta_{m+n,0} \wun, \qquad
			\comm{J_m}{G^\pm_s} = \pm G^\pm_{m+s}, \qquad
			\comm{G_r^\pm}{G_s^\pm} = 0, \\
	    \comm{G_r^+}{G_s^-} = 3 \no{JJ}_{r+s} - (\kk+3)L_{r+s} + \frac{3}{2}(\kk+1)(r-s)J_{r+s} + (\kk+1)(2\kk+3) \frac{r^2-\frac{1}{4}}{2} \delta_{r+s,0} \wun.
		\end{gathered}
	\end{equation}
\end{proposition}
\noindent Here, the indices $m$ and $n$ always take values in $\ZZ$ while $r$ and $s$ take values in $\ZZ+\frac{1}{2}$, if acting on an untwisted module, and in $\ZZ$, if acting on a twisted module.  We call the (unital associative) algebra generated by the modes of the fields of $\ubpvoak$ the \emph{untwisted mode algebra} $\modealg$, in the first case, and the \emph{twisted mode algebra} $\twmodealg$, in the latter case.  Each is a completion of the corresponding algebra generated by the modes of the generating fields.

\begin{definition}
	\leavevmode
	\begin{itemize}
		\item A \emph{fractional level} $\kk \in \CC$ for the \bp\ algebras is one that is not critical, meaning that $\kk \ne -3$, and for which $\ubpvoak$ is not simple.
		\item The \emph{level-$\kk$ simple \bp\ \voa} $\sbpvoak$ is the unique simple quotient of $\ubpvoak$.
	\end{itemize}
\end{definition}
\noindent According to \cite[Thms.~0.2.1 and 9.1.2]{GorSim07}, the fractional levels are precisely the $\kk$ satisfying
\begin{equation} \label{eq:fraclev}
	\kk+3 = \frac{\uu}{\vv}, \quad \text{where}\ \uu \in \ZZ_{\ge2},\ \vv \in \ZZ_{\ge1}\ \text{and}\ \gcd\set{\uu,\vv}=1.
\end{equation}
If $\kk$ is fractional, then we shall refer to $\sbpvoak$ as a \emph{\bp\ minimal model} and favour the alternative notation $\bpminmoduv$.  We note that the central charge of the minimal model $\bpminmoduv$ takes the form
\begin{equation}
	\cc = -\frac{(3\uu-8\vv)(2\uu-3\vv)}{\uu\vv} = 1 - \frac{6(\uu-2\vv)^2}{\uu\vv}.
\end{equation}
Whilst the central charge is invariant under exchanging $\frac{\uu}{\vv}$ with $\frac{4\vv}{\uu}$, the corresponding simple \voas\ are not isomorphic.  We shall see this explicitly when we analyse their representation theories in \cref{sec:sbpmods}.

\subsection{Automorphisms} \label{sec:bpsf}

There are two types of automorphisms of $\ubpvoak$ that will prove useful for the classification results to follow: the \emph{conjugation} automorphism $\conjsymb$ and the \emph{spectral flow} automorphisms $\sfsymb^{\ell}$, $\ell \in \ZZ$.  It is easy to verify that their actions, given below on the generating fields, indeed preserve the \opes\ \eqref{ope:bp}.
\begin{proposition}
	There exist conjugation and spectral flow automorphisms $\conjsymb$ and $\sfsymb^{\ell}$, $\ell \in \ZZ$, of the vertex algebra underlying $\ubpvoak$.  They are uniquely determined by the following actions on the generating fields:
	\begin{equation}
		\begin{gathered}
			\conjsymb(J(z)) = -J(z), \qquad
			\conjsymb(G^+(z)) = +G^-(z), \qquad
			\conjsymb(G^-(z)) = -G^+(z), \qquad
			\conjsymb(L(z)) = L(z), \\
			\begin{aligned}
				\sfsymb^{\ell}(J(z)) &= J(z) - \frac{2\kk+3}{3} \ell z^{-1} \wun, &&&
				\sfsymb^{\ell}(G^+(z)) &= z^{-\ell} G^+(z), \\
				\sfsymb^{\ell}(L(z)) &= L(z) - \ell z^{-1} J(z) + \frac{2\kk+3}{6} \ell^2 z^{-2} \wun, &&&
				\sfsymb^{\ell}(G^-(z)) &= z^{+\ell} G^-(z)
			\end{aligned}
		\end{gathered}
	\end{equation}
\end{proposition}
\noindent The $\sfsymb^{\ell}$ with $\ell \ne 0$ are not \voa\ automorphisms because they do not preserve the conformal structure.  Note that conjugation has order $4$, whilst spectral flow has infinite order.  Together, they satisfy the dihedral group relation
\begin{equation} \label{eq:dihedral}
	\conjsymb \sfsymb^{\ell} = \sfsymb^{-\ell} \conjsymb,
\end{equation}
though we do not have $\conjsymb^2 = \wun$.

\begin{proposition}
	Conjugation and spectral flow act on the modes of the generating fields $J(z)$, $G^+(z)$, $G^-(z)$ and $L(z)$ of $\ubpvoak$ as follows:
	\begin{equation} \label{eq:auts}
		\begin{aligned}
			\conjsymb(J_n) &= -J_n, & \conjsymb(G^+_r) &= +G^-_r, & \conjsymb(G^-_r) &= -G^+_r, & \conjsymb(L_n) &= L_n, \\
			\sfsymb^\ell (J_n) &= J_n - \frac{2\kk+3}{3} \ell \delta_{n,0} \wun, & \sfsymb^\ell (G^+_r) &= G^+_{r-\ell}, &
			\sfsymb^\ell (G^-_r) &= G^-_{r+\ell}, & \sfsymb^\ell (L_n) &= L_n - \ell J_n + \frac{2\kk+3}{6} \ell^2 \delta_{n,0} \wun.
		\end{aligned}
	\end{equation}
\end{proposition}
\noindent An extremely useful observation is that if we extend the definition of $\sfsymb^{\ell}$ to allow $\ell \in \ZZ+\frac{1}{2}$, then we see that $\sfsymb^{1/2}$ exchanges the twisted and untwisted mode algebras $\modealg$ and $\twmodealg$ introduced above.

Our main application for these automorphisms is to construct new $\ubpvoak$-modules from old ones.  This amounts to applying the automorphism (or its inverse) before acting with the representation morphism.  As we prefer to keep representations implicit, we implement this twisting notationally through the language of modules as follows.  Given a $\ubpvoak$-automorphism $\omega$ and a $\ubpvoak$-module $\Mod{M}$, define $\omega^*(\Mod{M})$ to be the image of $\Mod{M}$ under an (arbitrarily chosen) isomorphism $\omega^*$ of vector spaces.  The action of $\ubpvoak$ on $\omega^*(\Mod{M})$ is then defined by
\begin{equation} \label{eq:invfuncts}
	A(z) \cdot \omega^*(v) = \omega^*(\omega^{-1}(A(z)) v), \quad A(z) \in \ubpvoak,\ v \in \Mod{M}.
\end{equation}
In other words, $\omega(A(z)) \cdot \omega^*(v) = \omega^*(A(z) v)$.  In view of this, we shall hereafter drop the star that distinguishes the automorphism $\omega$ from the corresponding vector space isomorphism $\omega^*$.

Each $\ubpvoak$-automorphism $\omega$ thus lifts to an autoequivalence of any category of $\ubpvoak$-modules that is closed under $\omega$-twists.  The examples we have in mind are the category $\wcat{\kk}$ of weight modules, with \fdim\ weight spaces (see \cref{def:weight} below), and the analogous category $\twcat{\kk}$ of twisted modules.  In particular, the conjugation and spectral flow automorphisms lift to invertible endofunctors that provide an action of the infinite dihedral group on $\wcat{\kk}$ and $\twcat{\kk}$.  Extending the above formulae for $\sfsymb^{\ell}$ to allow $\ell \in \ZZ+\frac{1}{2}$, we see that the lift of $\sfsymb^{1/2}$ moreover defines an equivalence between $\wcat{\kk}$ and $\twcat{\kk}$.  We remark that one of the important consistency requirements for building a \cft\ from a module category over a \voa\ is that it is closed under twisting by automorphisms, especially conjugation.

\section{Identifying \bp\ modules} \label{sec:ubpmods}

Our aim is to classify the simple \rhwms, untwisted and twisted, for the \bp\ minimal models $\bpminmoduv$.  In order to have well defined characters, necessary to construct partition functions in \cft, we shall also require that the weight spaces of these simple \rhwms\ are \fdim.  By \cite{ZhuMod96}, it therefore suffices to classify the simple weight modules, with \fdim\ weight spaces, of the untwisted and twisted Zhu algebras of $\bpminmoduv$.

A direct assault on this classification seems quite difficult.  Our alternative strategy is threefold:  First, we understand the classification for certain associative algebras which have the untwisted and twisted Zhu algebras of $\bpminmoduv$ as quotients.  (These algebras turn out to be the untwisted and twisted Zhu algebras of the universal \bp\ \voas\ $\ubpvoak$, but this is inessential to the argument.)  This allows us to identify $\bpminmoduv$-modules in terms of data for these more easily understood associative algebras.  Second, we use Arakawa's results \cite{AraRep05} on minimal \qhrs\ to directly obtain the \hw\ classification for the $\bpminmoduv$, at present only known for $\vv=2$ \cite{AraRat13}.  Third, we combine these results to arrive at the relaxed classification by further developing the methods developed in \cite{KawRel18,KawRel19}.

In this \lcnamecref{sec:ubpmods}, we complete the first step of this strategy.  As nothing we do in this step is special to the minimal models, we shall work in the setting of $\ubpvoak$-modules.  Of course, all $\bpminmoduv$-modules are \emph{a priori} $\ubpvoak$-modules.

\subsection{Relaxed \hw\ $\ubpvoak$-modules} \label{sec:rhwms}

In \cref{sec:bpvoa}, we introduced the untwisted mode algebra $\modealg$ of the universal \bp\ \voa\ $\ubpvoak$ and its twisted version $\twmodealg$.  Any $\ubpvoak$-module is obviously a $\modealg$-module and, similarly, any twisted $\ubpvoak$-module is a $\twmodealg$-module.  As these algebras are graded by conformal weight (eigenvalue of $\comm{L_0}{\blank}$), we have the following generalised triangular decompositions, as in \cite{KacQua04}:
\begin{equation} \label{eq:pbw}
	\modealg = \modealg_< \otimes \modealg_0 \otimes \modealg_> \qquad \text{and} \qquad
	\twmodealg = \twmodealg_< \otimes \twmodealg_0 \otimes \twmodealg_>.
\end{equation}
Here, $\modealg_<$, $\modealg_0$ and $\modealg_>$ are the unital subalgebras generated by the modes $A_n$, for all homogeneous $A(z) \in \ubpvoak$, with $n<0$, $n=0$ and $n>0$, respectively (and similarly for their twisted versions).

\begin{definition} \label{def:weight}
	\leavevmode
	\begin{itemize}
		\item A vector $v$ in a twisted or untwisted $\ubpvoak$-module $\Mod{M}$ is a \emph{weight vector} of \emph{weight} $(j,\Delta)$ if it is a simultaneous eigenvector of $J_0$ and $L_0$ with eigenvalues $j$ and $\Delta$ called the \emph{charge} and \emph{conformal weight} of $v$, respectively.  The nonzero simultaneous eigenspaces of $J_0$ and $L_0$ are called the \emph{weight spaces} of $\Mod{M}$.  If $\Mod{M}$ has a basis of weight vectors and each weight space is \fdim, then $\Mod{M}$ is a \emph{weight module}.
		\item A vector in an untwisted $\ubpvoak$-module is a \emph{\hwv} if it is a simultaneous eigenvector of $J_0$ and $L_0$ that is annihilated by the action of $\modealg_>$.  An untwisted $\ubpvoak$-module generated by a single \hwv\ is called an \emph{untwisted \hwm}.
		\item A vector in a twisted $\ubpvoak$-module is a \emph{\hwv} if it is a simultaneous eigenvector of $J_0$ and $L_0$ that is annihilated by $G^+_0$ and the action of $\twmodealg_>$.  A twisted $\ubpvoak$-module generated by a single \hwv\ is called a \emph{twisted \hwm}.
		\item A vector in a twisted or untwisted $\ubpvoak$-module is a \emph{\rhwv} if it is a simultaneous eigenvector of $J_0$ and $L_0$ that is annihilated by the action of $\twmodealg_>$ or $\modealg_>$, respectively.  A $\ubpvoak$-module generated by a single \rhwv\ is called a \emph{\rhwm}.
	\end{itemize}
\end{definition}
\noindent As every $\bpminmoduv$-module is also a $\ubpvoak$-module (with $\kk+3=\frac{\uu}{\vv}$), these definitions descend to $\bpminmoduv$-modules in the obvious way.

A simple consequence of these definitions is that an untwisted \rhwv\ of $\ubpvoak$ is automatically a \hwv.  We shall therefore be concerned with classifying untwisted \hwms\ and twisted \rhwms.  The name ``\rhwm'' was originally coined in \cite{FeiEqu98} for the simple affine \voa\ $\saffvoa{\kk}{\sltwo}$ and now seems to be quite widespread.  Such modules had, however, appeared in earlier works such as \cite{AdaVer95}.  Here, we follow the definition proposed for quite general \voas\ in \cite{RidRel15}.

From the actions of the conjugation and spectral flow automorphisms, given explicitly in \eqref{eq:auts} and \eqref{eq:invfuncts}, we deduce the following useful facts.
\begin{proposition} \label{prop:twvsuntw}
	\leavevmode
	\begin{itemize}
		\item If $\Mod{M}$ is a twisted or untwisted $\ubpvoak$-module and $v \in \Mod{M}$ is a weight vector of weight $(j,\Delta)$, then $\conjsymb(v)$ and $\sfsymb^{\ell}(v)$ are weight vectors in $\conjmod{\Mod{M}}$ and $\sfmod{\ell}{\Mod{M}}$ of weights $(-j,\Delta)$ and $(j + \frac{2\kk+3}{3} \ell, \Delta + j \ell + \frac{2\kk+3}{6} \ell^2)$, respectively.
		\item Let $\Mod{M}$ be an untwisted $\ubpvoak$-module.  Then, $v \in \Mod{M}$ is a \hwv\ of weight $(j,\Delta)$ if and only if $\sfsymb^{1/2}(v)$ is a \hwv\ in the twisted module $\sfmod{1/2}{\Mod{M}}$ of weight $(j + \frac{2\kk+3}{6}, \Delta + \frac{1}{2} j + \frac{2\kk+3}{24})$.
		\item $\Mod{M}$ is a simple untwisted \hw\ $\ubpvoak$-module if and only if $\sfmod{1/2}{\Mod{M}}$ is a simple twisted \hw\ $\ubpvoak$-module.
	\end{itemize}
\end{proposition}
\noindent In particular, to classify all simple \hw\ $\ubpvoak$-modules, it is enough to only classify the untwisted ones.

We remark that there are simple weight $\ubpvoak$-modules that are not \hw, nor even \rhw.  In particular, if $\Mod{M}$ is a simple \rhw\ $\ubpvoak$-module, then $\sfmod{\ell}{\Mod{M}}$ is simple and weight, but is usually only \rhw\ for a few choices of $\ell$.  We believe, however, that the simple objects of the categories $\wcat{\kk}$ and $\twcat{\kk}$ of untwisted and twisted weight $\ubpvoak$-modules are all spectral flows of simple \rhw\ $\ubpvoak$-modules.

\subsection{The untwisted Zhu algebra} \label{sec:zhu}

The main tools that we shall use to classify \bp\ modules are the functors induced between these modules and those of the corresponding (untwisted) Zhu algebra.  Although originally introduced by Zhu \cite{ZhuMod96}, the idea behind this unital associative algebra was already well known to physicists (see \cite{FeiAnn92} for example).  Here, we use a (slightly restricted) abstract definition that is based on the physicists' ``zero-modes acting on ground states'' approach to Zhu algebras.  We refer to \cite[App.~B]{RidRel15} for further details (and motivation).

Suppose that $\VOA{V}$ is a \voa\ with conformally graded mode algebra $\modealg = \modealg_< \otimes \modealg_0 \otimes \modealg_>$, as in \eqref{eq:pbw}.  Let $\modealg'_>$ denote the ideal of $\modealg_>$ generated by the modes $A_n$ (so that $\modealg_> = \CC \wun \oplus \modealg'_>$ as vector spaces).
\begin{definition} \label{def:zhu}
	The \emph{untwisted Zhu algebra} of $\VOA{V}$ is the vector space
	\begin{equation}
		\zhu{\VOA{V}} = \frac{\modealg_0}{\modealg_0 \cap (\modealg \modealg'_>)},
	\end{equation}
	equipped with the multiplication (defined for homogeneous $A$ of conformal weight $\Delta_A$ and extended linearly)
	\begin{equation} \label{eq:zhumult}
		\zhuq{A_0} \zhuq{B_0} = \zhuq{A_0 B_0} = \sum_{n=0}^{\infty} \binom{\Delta_A}{n} \zhuq{(A_{-\Delta_A+n} B)_0},
	\end{equation}
	where $\zhuq{U_0}$ is the image in $\zhu{\VOA{V}}$ of $U_0 \in \modealg_0$.
\end{definition}

Zhu defined two functors between the categories of $\VOA{V}$- and $\zhu{\VOA{V}}$-modules.  We shall refer to them as the Zhu functor and the Zhu induction functor.  The first is quite easy to define.
\begin{definition}
	The \emph{Zhu functor} assigns to any $\VOA{V}$-module $\Mod{M}$, the $\zhu{\VOA{V}}$-module $\zhu{\Mod{M}} = \Mod{M}^{\modealg'_>}$, the subspace of $\Mod{M}$ whose elements are annihilated by $\modealg'_>$.
\end{definition}
\noindent The second is not so easily defined, but morally amounts to inducing a $\zhu{\VOA{V}}$-module, treating it as a $\modealg_0$-module equipped with a trivial $\modealg'_>$-action, and taking a quotient that imposes, among other things, the generalised commutation relations (Borcherds relations) of $\VOA{V}$.  The details may be found in \cite{ZhuMod96,LiRep94}.
\begin{proposition}[\cite{ZhuMod96}] \label{prop:zhuind}
	There exists a functor, which we call the \emph{Zhu induction functor}, that assigns to any $\zhu{\VOA{V}}$-module $\Mod{N}$ a $\VOA{V}$-module $\ind{\Mod{N}}$ such that $\zhu{\ind{\Mod{N}}} \cong \Mod{N}$.
\end{proposition}

The Zhu functor is thus a left inverse of the Zhu induction functor, at the level of isomorphism classes of modules.  However, it is not a right inverse in general.  Nevertheless, it is if we restrict to a certain class of simple $\VOA{V}$-modules.
\begin{definition}
	A (twisted or untwisted) $\VOA{V}$-module $\Mod{M}$ is \emph{lower-bounded} if it decomposes into (generalised) eigenspaces for the Virasoro zero-mode $L_0$ and the corresponding eigenvalues are bounded below.  If $\Mod{M}$ is lower-bounded, then the (generalised) eigenspace of minimal $L_0$-eigenvalue is called the \emph{top space} of $\Mod{M}$ and will be denoted by $\top{\Mod{M}}$.
\end{definition}
\noindent If $\Mod{M}$ is a lower-bounded $\VOA{V}$-module, then $\top{\Mod{M}}$ is naturally a $\zhu{\VOA{V}}$-module.  In fact, it may be identified with $\zhu{\Mod{M}}$ if $\Mod{M}$ is also simple, though this will not be true in general.  Simple lower-bounded $\VOA{V}$-modules have the following property.
\begin{theorem}[\cite{ZhuMod96}] \label{thm:zhubij}
	$\zhu{\blank}$ and $\ind{\blank}$ induce a bijection between the sets of isomorphism classes of simple lower-bounded $\VOA{V}$-modules and simple $\zhu{\VOA{V}}$-modules.
\end{theorem}
\noindent To classify the simple lower-bounded $\VOA{V}$-modules, it is therefore sufficient to classify the simple $\zhu{V}$-modules and apply $\ind{\blank}$.  We remark that for $\VOA{V} = \ubpvoak$ or $\bpminmoduv$, the simple lower-bounded weight modules coincide precisely with the simple \rhwms.

The first order of business is therefore to get information about the untwisted Zhu algebra $\zhu{\ubpvoak}$.
\begin{proposition} \label{prop:untwzhu}
	$\zhu{\ubpvoak}$ is a quotient of $\CC[J,L]$.
\end{proposition}
\begin{proof}
	Since the fields $G^\pm(z)$ have half-integer conformal weights, they do not have zero modes when acting on untwisted modules.  More generally, only the (homogeneous) fields of integer conformal weight have zero modes.  Express the zero mode of such a field as a linear combination of monomials in the modes of the generating fields $J(z)$, $G^{\pm}(z)$ and $L(z)$.  Next, use the commutation relations to order the modes so that the mode index weakly increases from left to right --- it is easy to see that this is always possible despite the nonlinear nature of the commutation relations \eqref{cr:bp}.  Now remove any monomial which contains a positive mode.  The image of the zero mode in $\zhu{\ubpvoak}$ is thus a polynomial in $\zhuq{J_0}$ and $\zhuq{L_0}$.  Since $\zhuq{L_0}$ is central in $\zhu{\ubpvoak}$, the multiplication \eqref{eq:zhumult} of $\zhu{\ubpvoak}$ matches that of $\CC[J,L]$.  There is therefore a surjective homomorphism $\CC[J,L] \to \zhu{\ubpvoak}$ determined by $J \mapsto \zhuq{J_0}$ and $L \mapsto \zhuq{L_0}$.
\end{proof}
\noindent It is in fact easy to show that $\zhu{\ubpvoak} \cong \CC[J,L]$, though we will not need this result in what follows.

\subsection{Identifying simple untwisted \hw\ $\ubpvoak$-modules} \label{sec:untwmod}

Having identified $\zhu{\ubpvoak}$ as a quotient of the free abelian algebra $\CC[J,L]$, we may identify its \fdim\ simple modules as $\CC[J,L]$-modules.
\begin{definition}
	A $\CC[J,L]$-module is said to be \emph{weight} if $J$ and $L$ act semisimply and their simultaneous eigenspaces are all \fdim.
\end{definition}
\noindent The simple weight modules of $\CC[J,L]$ are therefore one-dimensional.  We shall denote them by $\CC v_{j,\Delta}$, where $\lambda$ and $\Delta$ are the eigenvalues of $J$ and $L$, respectively, on $v_{j,\Delta}$.  As every simple $\zhu{\ubpvoak}$-module must also be simple as a $\CC[J,L]$-module, we arrive at our first identification result.
\begin{proposition}
	Every simple weight $\zhu{\ubpvoak}$-module, and hence every simple weight $\zhu{\bpminmoduv}$-module, is isomorphic to some $\CC v_{j,\Delta}$, where $\lambda,\Delta \in \CC$.
\end{proposition}
\noindent \cref{prop:zhuind,thm:zhubij} then guarantee that if $\CC v_{j,\Delta}$ is a $\zhu{\ubpvoak}$-module, then there exists a simple untwisted $\ubpvoak$-module $\ihw{j,\Delta}$ which is uniquely determined (up to isomorphism) by the fact that its top space is isomorphic to $\CC v_{j,\Delta}$ (as a $\CC[J,L]$-module).  As this top space is one-dimensional, $\ihw{j,\Delta}$ is a \hwm.
\begin{theorem} \label{thm:untwubpclass}
	Every simple untwisted \rhw\ $\ubpvoak$-module, and hence every simple untwisted \rhw\ $\bpminmoduv$-module, is isomorphic to some $\ihw{j,\Delta}$, where $\lambda,\Delta \in \CC$.
\end{theorem}
\noindent Note that there will be other simple weight $\ubpvoak$- and $\bpminmoduv$-modules such as those obtained from the $\ihw{j,\Delta}$ by applying spectral flow.  Simple nonweight modules also exist in general \cite{AdaCla19}, but they will not concern us here.

\subsection{The twisted Zhu algebra} \label{sec:twzhu}

The theory that extends Zhu algebras and functors to twisted modules was developed independently, and in different levels of generality, by Kac and Wang \cite{KacVer94} and by Dong, Li and Mason \cite{DonTwi98}.  From the point of view of ``zero modes acting on ground states'' however, the twisted story is almost identical to the untwisted one.  This is discussed in detail in \cite[App.~A]{BloSVir16}.

Given a \voa\ $\VOA{V}$ with twisted mode algebra $\twmodealg = \twmodealg_< \otimes \twmodealg_0 \otimes \twmodealg_>$, let ${\twmodealg_>}'$ be the ideal of $\twmodealg_>$ generated by the modes $A_n$.  Then, the twisted Zhu algebra and twisted Zhu functor of $\VOA{V}$ may be characterised as follows.
\begin{definition}
	\leavevmode
	\begin{itemize}
		\item The \emph{twisted Zhu algebra} of $\VOA{V}$ is the vector space
		\begin{equation}
			\twzhu{\VOA{V}} = \frac{\twmodealg_0}{\twmodealg_0 \cap (\twmodealg {\twmodealg_>}')},
		\end{equation}
		equipped with the multiplication defined in \eqref{eq:zhumult}, but where $\zhuq{U_0}$ is now the image in $\twzhu{\VOA{V}}$ of $U_0 \in \twmodealg_0$.
		\item The \emph{twisted Zhu functor} assigns to any twisted $\VOA{V}$-module $\Mod{M}$ the $\twzhu{\VOA{V}}$-module $\twzhu{\Mod{M}} = \Mod{M}^{{\twmodealg_>}'}$ of elements of $\Mod{M}$ that are annihilated by ${\twmodealg_>}'$.
	\end{itemize}
\end{definition}
\noindent The obvious analogues of Zhu's theorems for the twisted setting then hold.
\begin{theorem}[\cite{DonTwi98}] \label{thm:twzhubij}
	\leavevmode
	\begin{itemize}
		\item There exists a \emph{twisted Zhu induction functor} that takes a $\twzhu{\VOA{V}}$-module $\Mod{N}$ to a $\VOA{V}$-module $\twind{\Mod{N}}$ satisfying $\twzhu{\twind{\Mod{N}}} \cong \Mod{N}$.
		\item $\twzhu{\blank}$ and $\twind{\blank}$ induce a bijection between the sets of isomorphism classes of simple lower-bounded twisted $\VOA{V}$-modules and simple $\twzhu{\VOA{V}}$-modules.
	\end{itemize}
\end{theorem}
\noindent Again, the simple lower-bounded twisted weight $\VOA{V}$-modules coincide with the simple twisted \rhwms\ when $\VOA{V} = \ubpvoak$ or $\bpminmoduv$.

Our aim is to show that $\twzhu{\ubpvoak}$ is a quotient of some reasonably accessible associative algebra.  In contrast to the untwisted case detailed in \cref{sec:zhu}, the fields $G^\pm(z)$ do have zero modes when acting on twisted modules.  We therefore expect that $\twzhu{\ubpvoak}$ will be more complicated than $\zhu{\ubpvoak}$ --- in particular, we expect it to be nonabelian --- and so its representation theory will be more interesting.

\begin{definition}
	Let $\zk$ denote the (complex) unital associative algebra generated by $J, G^+, G^-$ and $L$, subject to $L$ being central and
	\begin{equation} \label{cr:zk}
		\comm{J}{G^\pm} = \pm G^{\pm}, \qquad
		\comm{G^+}{G^-} = f_{\kk}(J,L), \quad \text{where}\ f_{\kk}(J,L) = 3J^2 - (\kk+3)L - \frac{1}{8}(\kk+1)(2\kk+3) \wun.
	\end{equation}
\end{definition}
\begin{proposition} \label{prop:twzhu}
	$\twzhu{\ubpvoak}$ is a quotient of $\zk$.
\end{proposition}
\begin{proof}
	Every homogeneous field of $\ubpvoak$ has a zero mode when acting on a twisted module.  As in the proof of \cref{prop:untwzhu}, it follows that the zero modes of the generating fields have images that generate $\twzhu{\ubpvoak}$.  The fact that the generator $\zhuq{L_0}$ is central is standard \cite{KacVer94,DonTwi98}, but is also easy to verify directly in this case.

	We therefore start by using \eqref{eq:zhumult} to compute the products of the images of $J_0$ and $G^{\pm}_0$ in $\twzhu{\ubpvoak}$:
	\begin{align}
		\zhuq{J_0} \zhuq{G^{\pm}_0} &= \sum_{n=0}^{\infty} \binom{1}{n} \zhuq{(J_{n-1} G^{\pm})_0}
		= \zhuq{(J_0 G^{\pm})_0} + \zhuq{(J_{-1} G^{\pm})_0}
		= \pm \zhuq{G^{\pm}_0} + \zhuq{\no{J G^{\pm}}_0}, \\
		\zhuq{G^{\pm}_0} \zhuq{J_0} &= \sum_{n=0}^{\infty} \binom{3/2}{n} \zhuq{(G^{\pm}_{n-3/2} J)_0}
		= \zhuq{(G^{\pm}_{-3/2} J)_0} + \frac{3}{2} \zhuq{(G^{\pm}_{-1/2} J)_0} \\
		&= \zhuq{(J_{-1} G^{\pm})_0} \pm \zhuq{(\pd G^{\pm})_0} \pm \zhuq{G^{\pm}_0}
		= \zhuq{\no{J G^{\pm}}_0}. \notag
	\end{align}
	Here, we have noted that $G^{\pm}_{-3/2} J = G^{\pm}_{-3/2} J_{-1} \wun = J_{-1} G^{\pm}_{-3/2} \wun \mp G^{\pm}_{-5/2} \wun = \no{J G^{\pm}} \mp \pd G^{\pm}$, that $G^{\pm}_{-1/2} J = \mp G^{\pm}$ (similarly) and that $(\pd G^{\pm})_0 = -\frac{3}{2} G^{\pm}_0$.  With the surjection induced by $A \mapsto \zhuq{A_0}$, $A = J, G^{\pm}, L$, this proves the first relation in \eqref{cr:zk}.  The same method works for the second relation; we omit the somewhat more tedious details.
\end{proof}

It turns out that $\zk$ is in fact isomorphic to $\twzhu{\ubpvoak}$, though again we do not need this for what follows.  One can establish this isomorphism by combining the fact that $\twzhu{\ubpvoak}$ is known \cite{DeSFin06} to be isomorphic to the finite W-algebra associated to $\slthree$ and the minimal nilpotent orbit, while an explicit presentation of this finite W-algebra is given in \cite{TjiFin92}.  Either way, $\zk$ is a central extension of a Smith algebra, these algebras being introduced and studied in \cite{SmiCla90} as examples of associative algebras generalising the \uea\ of $\sltwo$.  This is of course well known, see \cite{AraRat13,AdaCla19} for instance.  The representation theory of $\zk$ is therefore quite tractable, a fact that we shall exploit in the next \lcnamecref{sec:twmod}.

\subsection{Identifying simple twisted \rhw\ $\ubpvoak$-modules} \label{sec:twmod}

As in the untwisted case, we wish to identify simple $\twzhu{\ubpvoak}$-modules as $\zk$-modules.  For this, we need a classification of the simple $\zk$-modules.  As $\zk$ is ``$\sltwo$-like'', similar classification methods may be used.  We shall mostly follow the approach presented in \cite{MazLec10} for $\sltwo$.

To begin, a triangular decomposition for $\zk$ is given by
\begin{equation} \label{eq:zktridec}
	\zk = \CC[G^-] \otimes \CC[J,L] \otimes \CC[G^+].
\end{equation}
The existence of this decomposition is an easy extension of \cite[Cor.~1.3]{SmiCla90}, which guarantees a \pbw-style basis for $\zk$.  The analogue of the Cartan subalgebra of $\sltwo$ is then spanned by $J$ and $L$.
\begin{definition}
	\leavevmode
	\begin{itemize}
		\item A vector in a $\zk$-module is a \emph{weight vector} of \emph{weight} $(j,\Delta)$ if it is a simultaneous eigenvector of $J$ and $L$ with eigenvalues $j$ and $\Delta$, respectively.  The nonzero simultaneous eigenspaces of $J$ and $L$ are called the \emph{weight spaces}.  If the $\zk$-module has a basis of weight vectors and its weight spaces are all \fdim, then it is a \emph{weight module}.
		\item A vector in a $\zk$-module is a \emph{\hwv} (\emph{\lwv}) if it is a weight vector that is annihilated by $G^+$ (by $G^-$).  A \emph{\hwm} (\emph{\lwm}) is a $\zk$-module that is generated by a single \hwv\ (by a single \lwv).
		\item A weight $\zk$-module is \emph{dense} if its weights coincide with the set $[j] \times \set{\Delta}$, for some coset $[j] \in \CC/\ZZ$ and some $\Delta \in \CC$.
	\end{itemize}
\end{definition}

We note that $\zk$ possesses a ``conjugation'' automorphism $\zconjsymb$ defined by
\begin{equation} \label{eq:defzconj}
	\zconjsymb(J) = -J, \quad \zconjsymb(G^+) = +G^-, \quad \zconjsymb(G^-) = -G^+, \quad \zconjsymb(L) = L.
\end{equation}
Conjugating a \hw\ $\zk$-module of highest weight $(j,\Delta)$ then results in a \lwm\ of lowest weight $(-j,\Delta)$ and vice versa.  The structures of highest- and \lw\ $\zk$-modules are therefore equivalent.

To construct \hw\ $\zk$-modules, we realise them as quotients of Verma $\zk$-modules.  Let $\zk^{\ge}$ denote the (unital) subalgebra of $\zk$ generated by $J$, $L$ and $G^+$.  Let $\CC_{j,\Delta}$, with $j,\Delta \in \CC$, be the one-dimensional $\zk^{\ge}$-module, spanned by $v$, on which we have $Jv = jv$, $Lv = \Delta v$ and $G^+ v=0$.  The Verma $\zk$-module $\zvhw{j,\Delta}$ is then the induced module $\zk \otimes_{\zk^{\ge}} \CC_{j,\Delta}$, as usual.  It is easy to check that $\zvhw{j,\Delta}$ is a \hwm\ with \hwv\ $v = \wun \otimes v$ and one-dimensional weight spaces of weights $(j-n,\Delta)$, $n \in \NN$.  Let $\zihw{j,\Delta}$ denote the unique simple quotient of $\zvhw{j,\Delta}$.

For convenience, we define
\begin{equation} \label{eq:defh}
	h_{\kk}^n(J,L) = \sum_{m=0}^{n-1} f_{\kk}(J-m\wun,L) = n \brac*{n^2 \wun - \frac{3}{2} n (2J+\wun) + \frac{1}{2} (6J^2+6J+\wun) - (\kk+3) L - \frac{1}{8} (\kk+1) (2\kk+3) \wun},
\end{equation}
where the $f_{\kk}$ were defined in \eqref{cr:zk}.
\begin{proposition} \label{prop:hwzkmod}
	\leavevmode
	\begin{itemize}
		\item The Verma module $\zvhw{j,\Delta}$ is simple, so $\zihw{j,\Delta} = \zvhw{j,\Delta}$, unless $h_{\kk}^n(j,\Delta) = 0$ for some $n \in \ZZ_{\ge1}$.
		\item Verma $\zk$-modules may have at most three composition factors.  Exactly one of these is \infdim.
		\item If $h_{\kk}^n(j,\Delta) = 0$ for some $n \in \ZZ_{\ge1}$ and $N$ is the minimal such $n$, then $\zihw{j,\Delta} \cong \zvhw{j,\Delta} \big/ \zvhw{j-N,\Delta}$ and $\dim \zihw{j,\Delta} = N$.
	\end{itemize}
\end{proposition}
\begin{proof}
	The first statement follows easily by noting that every proper nonzero submodule of $\zvhw{j,\Delta}$ is generated by a \sv\ of the form $(G^-)^n v$, $n \in \ZZ_{\ge1}$.  The condition to be a \sv\ is
	\begin{align}
		0 &= G^+ (G^-)^n v = \sum_{m=0}^{n-1} (G^-)^{n-1-m} \comm{G^+}{G^-} (G^-)^m v = \sum_{m=0}^{n-1} (G^-)^{n-1-m} f_{\kk}(J,L) (G^-)^m v \\
		&= \sum_{m=0}^{n-1} (G^-)^{n-1} f_{\kk}(J-m\wun,L) v = (G^-)^{n-1} \sum_{m=0}^{n-1} f_{\kk}(j-m\wun,\Delta) v = h_{\kk}^n(j,\Delta) (G^-)^{n-1} v. \notag
	\end{align}
	Since $h_{\kk}^n$ is a cubic polynomial in $n$, there can be at most three roots in $\ZZ_{\ge1}$, hence at most three \hwvs.  The remaining statements are now clear.
\end{proof}
\noindent Unlike $\sltwo$, there exist nonsemisimple \fdim\ $\zk$-modules.  Examples include the \hwms\ obtained by quotienting a Verma module with three composition factors by its socle.

This \lcnamecref{prop:hwzkmod} completes the classification of \fdim\ $\zk$-modules and \hw\ $\zk$-modules.  To obtain the analogous classification of \lw\ $\zk$-modules, we apply the conjugation automorphism $\zconjsymb$.  The conjugate of a simple Verma module $\zvhw{j,\Delta}$ is the \lw\ Verma module of lowest weight $(-j,\Delta)$.  However, if $\zvhw{j,\Delta}$ is not simple and $N$ is the smallest positive integer such that $h_{\kk}^N(j,\Delta) = 0$, then the conjugate of $\zihw{j,\Delta}$ is isomorphic to $\zihw{N-j-1,\Delta}$.

It remains to determine the simple weight $\zk$-modules that are neither highest- nor \lw.  Such modules are necessarily dense.  As for $\sltwo$, the classification of simple dense $\zk$-modules is greatly simplified by identifying the centraliser $\ck$ of the subalgebra $\CC[J,L]$ in $\zk$.
\begin{lemma}
	The centraliser $\ck$ is the polynomial algebra $\CC[J,L,G^+G^-]$.
\end{lemma}
\begin{proof}
	Note first that $G^+ G^-$ obviously commutes with $J$, by \eqref{cr:zk}.  Consider a \pbw\ basis of $\zk$ given by elements of the form $J^a L^b (G^+)^c (G^-)^d$, for $a,b,c,d \in \NN$.  It is easy to check that such a basis element belongs to $\ck$ if and only if $c=d$.  To show that $J$, $L$ and $G^+G^-$ generate $\ck$, it therefore suffices to show that $(G^+)^c (G^-)^c$ may be written as a polynomial in $J$, $L$ and $G^+G^-$, for each $c \in \NN$.

	Proceeding by induction, this is clear for $c=0$.  So take $c \ge 1$ and assume that $(G^+)^{c-1} (G^-)^{c-1}$ is a polynomial in $J$, $L$ and $G^+G^-$.  Then, the commutation rules \eqref{cr:zk} give
	\begin{align}
		(G^+)^c (G^-)^c &= (G^+ G^-) (G^+)^{c-1} (G^-)^{c-1} + G^+ \comm{(G^+)^{c-1}}{G^-} (G^-)^{c-1} \notag \\
		&= (G^+ G^-) (G^+)^{c-1} (G^-)^{c-1} + \sum_{n=1}^{c-1} (G^+)^n f_{\kk}(J,L) (G^+)^{c-1-n} (G^-)^{c-1}.
	\end{align}
	The first term on the \rhs\ is a polynomial in $J$, $L$ and $G^+G^-$, by the inductive hypothesis.  For the remaining terms, note that as $L$ is central and $G^+ J = (J-\wun) G^+$, we have $(G^+)^n J = (J-n\wun) (G^+)^n$ and hence
	\begin{equation}
		\sum_{n=1}^{c-1} (G^+)^n f_{\kk}(J,L) (G^+)^{c-1-n} (G^-)^{c-1}
		= \sum_{n=1}^{c-1} f_{\kk}(J-n\wun,L) (G^+)^{c-1} (G^-)^{c-1},
	\end{equation}
	which is likewise a polynomial in $J$, $L$ and $G^+G^-$.
\end{proof}

Recall that the weight spaces of a simple weight $\zk$-module are simple $\ck$-modules (see \cite[Lem.~3.4.2]{MazLec10} for example).  The fact that $\ck$ is abelian now gives the following result.
\begin{proposition} \label{prop:zkwtsp}
	The weight spaces of a simple weight $\zk$-module are one-dimensional.
\end{proposition}
\noindent To understand these weight spaces, one therefore needs to know the eigenvalues of $J$, $L$ and $G^+G^-$ on a given simple weight $\zk$-module.  The latter will vary with the weight $(j,\Delta)$ in general, so it is convenient to note that we may replace $G^+G^-$ by a central element of $\zk$, something like a Casimir operator, whose eigenvalue is therefore constant.
\begin{lemma} \label{lem:idcent}
	The element
	\begin{equation} \label{eq:DefCas}
		\Omega = G^+ G^- + G^- G^+ + 2J^3 + J - 2J \brac*{(\kk+3)L + \frac{1}{8}(\kk+1)(2\kk+3)\wun}
	\end{equation}
	is central in $\zk$ and we have $\zconjsymb(\Omega) = -\Omega$ and $\ck = \CC[J,L,\Omega]$.
\end{lemma}
\begin{proof}
	We start by noting that
	\begin{equation}
		\comm{G^+G^-}{G^+} = -G^+ f_{\kk}(J,L) = -G^+ \brac*{3J^2 - (\kk+3)L - \tfrac{1}{8}(\kk+1)(2\kk+3)\wun}.
	\end{equation}
	Since $\comm{J^n}{G^+} = G^+ \brac*{(J+\wun)^n - J^n}$, we can cancel the terms appearing on the \rhs\ (starting with $3J^2$) by adding counterterms to $G^+G^-$.  In this way, we arrive at an element $\widetilde{\Omega} \in \zk$ that commutes with $J$, $G^+$ and $L$:
	\begin{equation}
		\widetilde{\Omega} = G^+G^- + J^3 - \frac{3}{2} J^2 + \frac{1}{2} J - J \brac*{(\kk+3)L + \frac{1}{8}(\kk+1)(2\kk+3)\wun}.
	\end{equation}
	By using $G^+G^- = G^-G^+ + f_{\kk}(J,L)$, we obtain a second expression for $\widetilde{\Omega}$.  Adding the two expressions, we see that
	\begin{equation}
		\Omega = 2 \widetilde{\Omega} + (\kk+3)L + \frac{1}{8}(\kk+1)(2\kk+3)\wun
	\end{equation}
	also commutes with $J$, $G^+$ and $L$.  But, the explicit form \eqref{eq:DefCas} shows that it also commutes with $G^-$ because the conjugation automorphism \eqref{eq:defzconj} gives $\zconjsymb(\Omega) = -\Omega$.
\end{proof}

By \eqref{eq:DefCas}, the eigenvalue of $\Omega$ on a \hwv\ ($+$) or \lwv\ ($-$) of weight $(j,\Delta)$ is given by
\begin{equation} \label{eq:defomegapm}
	\omega^{\pm}_{j,\Delta} = (2j\pm1) \brac*{j(j\pm1) - (\kk+3) \Delta - \frac{1}{8} (\kk+1) (2\kk+3)}.
\end{equation}
These eigenvalues satisfy the following relations:
\begin{equation} \label{eq:omegasymm}
	\omega^-_{-j,\Delta} = -\omega^+_{j,\Delta} = \omega^+_{-j-1,\Delta}.
\end{equation}
We note that the first equality is consistent with conjugation.

We now construct dense $\zk$-modules by induction.  Let $\CC_{j,\Delta,\omega}$ be a one-dimensional $\ck$-module, spanned by $v$, on which we have $Jv=jv$, $Lv=\Delta v$ and $\Omega v = \omega v$, for some $j,\Delta,\omega \in \CC$.  Define the induced module $\zrhw{j,\Delta,\omega} = \zk \otimes_{\ck} \CC_{j,\Delta,\omega}$ and note that a basis of $\zrhw{j,\Delta,\omega}$ is given by $v = \wun \otimes v$ and the $(G^{\pm})^n v$ with $n \in \ZZ_{\ge1}$.  The weights therefore coincide with $[j] \times \set{\Delta}$ and so $\zrhw{j,\Delta,\omega}$ is a dense $\zk$-module generated by $v$.

\begin{proposition} \label{prop:zkdense}
	\leavevmode
	\begin{itemize}
		\item For each $n \in \NN$, $(G^-)^{n+1} v$ is a \hwv\ of $\zrhw{j,\Delta,\omega}$ if and only if $\omega = \omega^+_{j-n-1,\Delta}$.
		\item For each $n \in \NN$, $(G^+)^{n+1} v$ is a \lwv\ of $\zrhw{j,\Delta,\omega}$ if and only if $\omega = \omega^-_{j+n+1,\Delta}$.
		\item The dense $\zk$-module $\zrhw{j,\Delta,\omega}$ is simple if and only if $\omega \ne \omega^+_{i,\Delta}$ (equivalently $\omega \ne \omega^-_{i,\Delta}$) for any $i \in [j]$.
		\item $\zrhw{j,\Delta,\omega}$ has at most four composition factors.  If it is not simple, then one composition factor is \infdim\ \hw\ and another is \infdim\ \lw; any other composition factors are \fdim.
	\end{itemize}
\end{proposition}
\begin{proof}
	The existence criteria for highest- and \lwvs\ is straightforward calculation using \eqref{eq:omegasymm}.  The simplicity of $\zrhw{j,\Delta,\omega}$ is equivalent to the absence of highest- and \lwvs.  However, $\omega \ne \omega^-_{j-n,\Delta}$ for all $n \in \NN$ implies that $\omega \ne \omega^+_{j-n-1,\Delta}$ for all $n \in \NN$, by \eqref{eq:omegasymm}.  Combining with $\omega \ne \omega^+_{j+n,\Delta}$ for all $n \in \NN$, we get the desired condition.  The statements about composition factors now follow from the fact that $\omega - \omega^{\pm}_{i,\Delta}$ is a cubic polynomial in $i$, so it can have at most three roots $i \in [j]$.
\end{proof}
\noindent It follows from this \lcnamecref{prop:zkdense} that we have isomorphisms $\zrhw{j,\Delta,\omega} \cong \zrhw{j+1,\Delta,\omega}$ when these modules are simple.  We shall therefore denote these simple dense $\zk$-modules by $\zrhw{[j],\Delta,\omega}$, where $[j] \in \CC/\ZZ$.

\begin{theorem} \label{thm:classzkmod}
	Every simple weight $\zk$-module is isomorphic to one of the modules in the following list of pairwise-inequivalent modules:
	\begin{itemize}
		\item The \fdim\ \hwms\ $\zihw{j,\Delta}$ with $j,\Delta \in \CC$ such that $h_{\kk}^n(j,\Delta) = 0$ for some $n \in \ZZ_{\ge1}$.
		\item The \infdim\ \hwms\ $\zihw{j,\Delta} = \zvhw{j,\Delta}$ with $j,\Delta \in \CC$ such that $h_{\kk}^n(j,\Delta) \ne 0$ for all $n \in \ZZ_{\ge1}$.
		\item The \infdim\ \lwms\ $\zconjmod{\zihw{j,\Delta}} = \zconjmod{\zvhw{j,\Delta}}$ with $j,\Delta \in \CC$ such that $h_{\kk}^n(j,\Delta) \ne 0$ for all $n \in \ZZ_{\ge1}$.
		\item The \infdim\ dense modules $\zrhw{[j],\Delta,\omega}$ with $[j] \in \CC/\ZZ$ and $\Delta,\omega \in \CC$ such that $\omega \ne \omega^+_{i,\Delta}$ for any $i \in [j]$.
	\end{itemize}
\end{theorem}
\begin{proof}
	The classification was already completed after \cref{prop:hwzkmod} for the first three cases, that is when the simple weight module has either a highest- or \lw\ (or both).  If the simple weight module has no highest- or \lw, choose an arbitrary weight space.  This is a simple $\ck$-module, hence it is one-dimensional (\cref{prop:zkwtsp}) and spanned by $v$ say.  As there are no highest- or \lwvs, $G^+$ and $G^-$ act freely on $v$ and so the simple weight module is dense and so isomorphic to one of the $\zrhw{[j],\Delta,\omega}$ in the list.
\end{proof}

As in the untwisted case, the fact that $\twzhu{\ubpvoak}$ is a quotient of $\zk$ means that every simple $\twzhu{\ubpvoak}$-module is also simple as a $\zk$-module.  \cref{thm:twzhubij} then guarantees that every simple weight $\twzhu{\ubpvoak}$-module $\zversion{\Mod{M}}$ corresponds to a simple twisted \rhw\ $\ubpvoak$-module $\Mod{M} = \twind{\zversion{\Mod{M}}}$ which is uniquely determined (up to isomorphism) by the fact that its top space is isomorphic to $\zversion{\Mod{M}}$ (as a $\zk$-module).

\begin{theorem} \label{thm:classtwrhw}
	Every simple twisted \rhw\ $\ubpvoak$-module, and hence every simple twisted \rhw\ $\bpminmoduv$-module, is isomorphic to one of the modules in the following list of pairwise-inequivalent modules:
	\begin{itemize}
		\item The \hwms\ $\twihw{j,\Delta}$ with $j,\Delta \in \CC$ such that $h_{\kk}^n(j,\Delta) = 0$ for some $n \in \ZZ_{\ge1}$.
		\item The \hwms\ $\twihw{j,\Delta} = \twvhw{j,\Delta}$ with $j,\Delta \in \CC$ such that $h_{\kk}^n(j,\Delta) \ne 0$ for all $n \in \ZZ_{\ge1}$.
		\item The conjugate \hwms\ $\conjmod{\twihw{j,\Delta}} = \conjmod{\twvhw{j,\Delta}}$ with $j,\Delta \in \CC$ such that $h_{\kk}^n(j,\Delta) \ne 0$ for all $n \in \ZZ_{\ge1}$.
		\item The \rhwms\ $\twrhw{[j],\Delta,\omega}$ with $[j] \in \CC/\ZZ$ and $\Delta,\omega \in \CC$ such that $\omega \ne \omega^+_{i,\Delta}$ for all $i \in [j]$.
	\end{itemize}
\end{theorem}
\noindent Again, we remark that spectral flow will allow us to construct simple twisted weight $\ubpvoak$-modules that are not \rhw, in general.

\subsection{Coherent families} \label{sec:cohfam}

A crucial observation of Mathieu \cite{MatCla00} concerning simple dense $\alg{g}$-modules, for $\alg{g}$ a simple Lie algebra, is that they may be naturally arranged into coherent families.  Here, we extend this observation to dense $\zk$-modules in preparation for showing that it also extends to $\twzhu{\bpminmoduv}$-modules.  While Mathieu's general results rely heavily on the properties of his twisted localisation functors, our discussion of this simple case will be quite elementary.

\begin{definition}
	A \emph{coherent family of $\zk$-modules} is a weight module $\zcfam{}$ for which:
	\begin{itemize}
		\item $L$ and $\Omega$ act as multiples, $\Delta$ and $\omega$ respectively, of the identity on $\zcfam{}$.
		\item There exists $d \in \NN$ such that for all $j \in \CC$, the dimension of the weight space $\zcfam{}(j,\Delta)$ of weight $(j,\Delta)$ is $d$.
		\item For each $U \in \ck$, the function taking $j \in \CC$ to $\tr_{\zcfam{}(j,\Delta)} U$ is polynomial in $j$.
	\end{itemize}
\end{definition}

Coherent families are highly decomposable.  Indeed, a coherent family of $\zk$-modules necessarily has the form
\begin{equation}
	\zcfam{} = \bigoplus_{[j] \in \CC/\ZZ} \zcfam{[j]}.
\end{equation}
If all of the $\zcfam{[j]}$ are semisimple as $\zk$-modules, then $\zcfam{}$ is said to be \emph{semisimple}.  If any of the $\zcfam{[j]}$ are simple as $\zk$-modules, then $\zcfam{}$ is said to be \emph{irreducible}.  It follows immediately from \cref{prop:zkwtsp} that the common dimension $d$ of the weight spaces of an irreducible coherent family of $\zk$-modules is $1$.

We would like to form a coherent family of $\zk$-modules by summing over some collection of dense modules $\zrhw{[j],\Delta,\omega}$, $[j] \in \CC/\ZZ$, whilst holding $\Delta$ and $\omega$ fixed.  However, this is mildly ambiguous because there will always be at least one $[j]$ (generically three) for which the corresponding element in the collection will not be simple and so we should then specify precisely which module we mean.  For such $j$, we shall specify this in three distinct ways (though there are others).
\begin{itemize}
	\item The first is to define $\zrhw{[j],\Delta,\omega}$ to be $\semi{\zrhw{j,\Delta,\omega}}$, where the \emph{semisimplification} $\semi{\Mod{M}}$ of a (finite-length) module $\Mod{M}$ is the direct sum of its composition factors.  This is well defined as $\semi{\zrhw{j,\Delta,\omega}} \cong \semi{\zrhw{j+1,\Delta,\omega}}$.
	\item An alternative is to define $\zrhw{[j],\Delta,\omega}$ to be $\zprhw{[j],\Delta,\omega} = \zrhw{j^+,\Delta,\omega}$, where we choose $j^+ \in [j]$ to have smaller real part than those of the solutions $i \in [j]$ of $\omega = \omega^+_{i,\Delta}$.  This ensures that $\zprhw{[j],\Delta,\omega}$ has no \hwvs.
	\item We may instead define $\zrhw{[j],\Delta,\omega}$ to be $\zmrhw{[j],\Delta,\omega} = \zrhw{j^-,\Delta,\omega}$, where we choose $j^- \in [j]$ to have larger real part than those of the solutions $i \in [j]$ of $\omega = \omega^-_{i,\Delta}$.  This ensures that $\zmrhw{[j],\Delta,\omega}$ has no \lwvs.
\end{itemize}

For each of the three choices above, we take the direct sum of the $\zrhw{[j],\Delta,\omega}$ over $[j] \in \CC/\ZZ$.  The result is easily verified to be an irreducible coherent family of $\zk$-modules.  It will be denoted by $\zsscfam{\Delta,\omega}$, $\zpcfam{\Delta,\omega}$ or $\zmcfam{\Delta,\omega}$, respectively.  The first is semisimple, whilst the second is nonsemisimple with $G^+$ acting injectively and the third is nonsemisimple with $G^-$ acting injectively.  It is easy to check that the conjugates of these irreducible coherent families are
\begin{equation} \label{eq:conjcfam}
	\zconjmod{\zsscfam{\Delta,\omega}} \cong \zsscfam{\Delta,-\omega}, \quad
	\zconjmod{\zpcfam{\Delta,\omega}} \cong \zmcfam{\Delta,-\omega} \quad \text{and} \quad
	\zconjmod{\zmcfam{\Delta,\omega}} \cong \zpcfam{\Delta,-\omega}.
\end{equation}

For classifying simple $\bpminmoduv$-modules, the semisimple coherent families $\zsscfam{\Delta,\omega}$ are most suitable.  Note that $\zsscfam{\Delta,\omega}$ is the unique irreducible semisimple coherent family of $\zk$-modules on which $L$ acts as multiplication by $\Delta$ and $\Omega$ acts as multiplication by $\omega$, up to isomorphism.  We shall return to $\zpcfam{\Delta,\omega}$ and $\zmcfam{\Delta,\omega}$ in \cref{sec:indec} when considering the existence of nonsemisimple $\bpminmoduv$-modules.

\begin{proposition} \label{prop:cfamuniv}
	\leavevmode
	\begin{itemize}
		\item Every simple weight $\zk$-module embeds into a unique irreducible semisimple coherent family.
		\item Every irreducible semisimple coherent family of $\zk$-modules contains an \infdim\ \hw\ submodule.
	\end{itemize}
\end{proposition}
\begin{proof}
	By \cref{thm:classzkmod}, a simple dense $\zk$-module $\Mod{M}$ is isomorphic to some $\zrhw{[j],\Delta,\omega}$, where $[j] \in \CC/\ZZ$ and $\Delta,\omega \in \CC$ satisfy $\omega \ne \omega^+_{i,\Delta}$ for any $i \in [j]$.  As $\semi{\zrhw{[j],\Delta,\omega}} = \zrhw{[j],\Delta,\omega}$, we have an embedding $\Mod{M} \ira \zsscfam{\Delta,\omega}$.  The target is obviously unique, up to isomorphism, since no other irreducible semisimple coherent family has the correct $L$- and $\Omega$-eigenvalues.

	A simple \hw\ $\zk$-module $\Mod{M}$ is isomorphic to $\zihw{j,\Delta}$, for some $j,\Delta \in \CC$.  Take $\omega = \omega^+_{j,\Delta}$, so that $\zrhw{j,\Delta,\omega}$ is not simple and there is a \hwv\ of weight $(j,\Delta)$ in $\semi{\zrhw{j,\Delta,\omega}}$, by \cref{prop:zkdense}.  This vector generates a copy of $\zihw{j,\Delta}$, so we again have an embedding $\Mod{M} \ira \zsscfam{\Delta,\omega}$ with unique target.

	Finally, if $\Mod{M}$ is a simple \lw\ $\zk$-module, then we have an embedding $\zconjmod{\Mod{M}} \ira \zsscfam{\Delta,\omega}$ for some unique $\Delta,\omega \in \CC$.  By \eqref{eq:conjcfam}, we have $\Mod{M} \ira \zsscfam{\Delta,-\omega}$.  This covers all possibilities, by \cref{thm:classzkmod}, so the first statement is established.

	For the second, a given irreducible semisimple coherent family $\zsscfam{\Delta,\omega}$ is uniquely specified by choosing $\Delta,\omega \in \CC$.  As $\omega - \omega^+_{i,\Delta}$ is a cubic polynomial in $i$, there is at least one solution in $\CC$, $i=j$ say.  Then, $\zrhw{j,\Delta,\omega}$ is not simple and has an \infdim\ \hw\ submodule, by \cref{prop:zkdense}, hence so does $\semi{\zrhw{j,\Delta,\omega}} \subset \zsscfam{\Delta,\omega}$.
\end{proof}

\section{Modules of the simple admissible-level \bp\ algebras} \label{sec:sbpmods}

Recall \cite{FreVer92} that if $\VOA{I}$ is an ideal of a \voa\ $\VOA{V}$, then $\zhu{\VOA{V}/\VOA{I}} \cong \zhu{\VOA{V}} / \zhu{\VOA{I}}$.  If $\bpmpik$ denotes the maximal ideal of $\ubpvoak$, then classifying the \rhwms\ of $\sbpvoak = \ubpvoak / \bpmpik$ is then just a matter of classifying those of $\ubpvoak$ and then testing which have Zhu-images annihilated by $\zhu{\bpmpik}$.  The twisted classification then follows, roughly speaking, from spectral flow.  Unfortunately, it is hard to compute $\zhu{\bpmpik}$ in general.

Instead, we shall combine Arakawa's celebrated classification \cite{AraRat16} of the \hwms\ of all simple admissible-level affine \voas\ $\saffvoa{\kk}{\alg{g}}$, specialised to $\alg{g} = \slthree$, with his results \cite{AraRep05} on minimal quantum hamiltonian reduction.  The result will be a classification of the \hwms\ for the \bp\ minimal models from which we will extract the full (twisted and untwisted) \rhw\ classification.

\subsection{Admissible-level $\slthree$ minimal models} \label{sec:sl3}

Recall from \eqref{eq:fraclev} the fractional levels of $\ubpvoak$ and their parametrisation in terms of $\uu$ and $\vv$.  These are also the fractional levels for the affine \voas\ associated to $\slthree$ --- $\uslvoak$ is not simple \cite[Thm.~0.2.1]{GorSim07} when $\kk$ is a fractional level.  For such $\kk$, the simple quotient will be denoted by $\sslvoak = \slminmoduv$.

\begin{definition}
	An \emph{admissible level} $\kk$ for the affine \voas\ associated to $\slthree$, and the \bp\ algebras, is a fractional level for which $\uu\ge3$.
\end{definition}

Every \hwm\ for the affine \km\ algebra $\aslthree$ is a $\uslvoak$-module \cite{FreVer92}.  Let $\slihw{\lambda}$ denote the simple \hw\ $\aslthree$-module of highest weight $\lambda = \lambda_0 \fwt{0} +\lambda_1 \fwt{1} + \lambda_2 \fwt{2}$, where the $\lambda_i$ are the Dynkin labels and the $\fwt{i}$ are the fundamental weights.  To be a level-$\kk$ module, we must have $\lambda_0 + \lambda_1 + \lambda_2 = \kk$.  Let $\pwlat{\ell}$ denote the set of dominant integral level-$\ell$ weights of $\aslthree$, that is the set of weights $\lambda$ satisfying $\lambda_i \in \NN$ and $\lambda_0 + \lambda_1 + \lambda_2 = \ell$.  This set is obviously empty unless $\ell \in \NN$.  Let $\wref{i}$, $i=0,1,2$, denote the Weyl reflection corresponding to the simple root $\sroot{i}$ of $\aslthree$.

The following definition specialises that of \cite{KacMod88} to $\aslthree$ (see also \cite[App.~18.B]{DiFCon97}).
\begin{definition} \label{def:admsl3wt}
	Let $\kk$ be an admissible level.  A level-$\kk$ \emph{admissible weight} $\lambda$ of $\aslthree$ is one of the form
	\begin{equation} \label{eq:admsl3wt}
		\lambda = w \cdot \brac*{\lambda^I - \frac{\uu}{\vv} \lambda^{F,w}},
	\end{equation}
	where $w \in \set{\wun,\wref{1}}$ is a Weyl transformation of $\slthree$, $\cdot$ is the shifted Weyl group action, $\lambda^I \in \pwlat{\uu-3}$, $\lambda^{F,w} \in \pwlat{\vv-1}$ and $\lambda_1^{F,\wref{1}} \ge 1$.  A weight of the form \eqref{eq:admsl3wt} will be called a $w=\wun$ or $w=\wref{1}$ admissible weight according as to which $w$ is used.
\end{definition}
\noindent We remark that one may allow $w$ to range over the full Weyl group, adding appropriate restrictions on the $\lambda^{F,w}$, but this gives no further admissible weights.  In fact, every set of $w=w'$ admissible weights is equal to either the $w=\wun$ or $w=\wref{1}$ sets and, moreover, these two sets are disjoint \cite[Prop.~2.1]{KacCla89}.

Arakawa's \hw\ classification for affine \voas\ now specialises as follows.
\begin{theorem}[\cite{AraRat16}] \label{thm:affhwclass}
	For $\kk$ admissible, the simple level-$\kk$ \hwm\ $\slihw{\lambda}$ is an $\slminmoduv$-module if and only if $\lambda$ is admissible.
\end{theorem}

Denote by $\qhrfun{\blank}$ the minimal \qhr\ functor \cite{KacQua03} taking $\uslvoak$-modules to $\ubpvoak$-modules, so that $\qhrfun{\uslvoak} = \ubpvoak$.  For definiteness, we take the nilpotent element of $\slthree$ defining this functor to be the negative highest-root vector $\nrvec{\hroot}$.  We assemble some useful results about this functor, specialised to our setting.
\begin{theorem} \label{thm:qhr}
	\leavevmode
	\begin{itemize}
		\item\cite[Thm.~6.3]{KacQua04} If $\slvhw{\lambda}$ denotes the Verma module of $\uslvoak$ with highest weight $\lambda$, then $\qhrfun{\slvhw{\lambda}}$ is isomorphic to the Verma module $\vhw{j,\Delta}$ of $\ubpvoak$ with
		\begin{equation} \label{eq:bphwfromsl3}
			j = \frac{\lambda_1 - \lambda_2}{3} \quad \text{and} \quad
			\Delta = \frac{(\lambda_1-\lambda_2)^2 - 3(\lambda_1+\lambda_2) \brac[\big]{2(\kk+1)-\lambda_1-\lambda_2}}{12(\kk+3)}.
		\end{equation}
		\item\cite[Thm.~6.7.4]{AraRep05} $\qhrfun{\slihw{\lambda}} = 0$ if and only if $\lambda_0 \in \NN$.  For $\lambda_0 \notin \NN$, we have instead $\qhrfun{\slihw{\lambda}} \cong \ihw{j,\Delta}$, where $j$ and $\Delta$ are given by \eqref{eq:bphwfromsl3}.
		\item\cite[Cor.~6.7.3]{AraRep05} The restriction of $\qhrfun{\blank}$ to the category $\slocat{\kk}$ of level-$\kk$ $\aslthree$-modules is exact.
		\item $\qhrfun{\blank}$ induces a surjection from the set of isomorphism classes of simple \hw\ $\uslvoak$-modules to the union of $\set{0}$ and the set of isomorphism classes of simple \hw\ $\ubpvoak$-modules.  Moreover, there are at most two inequivalent $\slihw{\lambda}$ mapping onto the same $\ihw{j,\Delta}$.
	\end{itemize}
\end{theorem}
\begin{proof}
	We only prove the last assertion.  It follows from the second assertion above and by inverting \eqref{eq:bphwfromsl3} to obtain two solutions $(\lambda_1,\lambda_2)$ for each $(j,\Delta)$.  We have to ensure that at least one solution gives $\lambda_0 \notin \NN$.  But, a simple calculation gives
	\begin{equation} \label{eq:solvebphw}
		\lambda_0 = \kk - \lambda_1 - \lambda_2 = -1 \pm \sqrt{4(\kk+3)\Delta + (\kk+1)^2 - 3j^2},
	\end{equation}
	so the zeroth Dynkin labels of the two solutions sum to $-2$.
\end{proof}

\begin{definition}
	For $\kk$ admissible, we shall call a level-$\kk$ weight $\lambda$ of $\aslthree$ \emph{surviving} if it is admissible and $\lambda_0 \notin \NN$.  \cref{thm:qhr} then ensures that $\qhrfun{\slihw{\lambda}}$ is nonzero (and is moreover a simple $\ubpvoak$-module).
\end{definition}
\begin{lemma} \label{lem:survivors}
	\leavevmode
	\begin{itemize}
    \item Every $w=\wref{1}$ admissible weight is surviving.
    \item A $w=\wun$ admissible weight $\lambda$ is surviving if and only if $\lambda^{F,\wun}_0 \ge 1$.
    \item $\wref{0} \cdot$ gives a $(j,\Delta)$-preserving bijection between the $w=\wun$ surviving weights and the $w=\wref{1}$ admissible weights.
    \item If $\lambda$ and $\mu$ are distinct $w=\wun$ surviving weights, then $\qhrfun{\slihw{\lambda}}$ and $\qhrfun{\slihw{\mu}}$ are not isomorphic.
	\end{itemize}
\end{lemma}
\begin{proof}
	The zeroth Dynkin label of a level-$\kk$ admissible $\aslthree$-weight $\lambda$ has one of the following two forms:
	\begin{equation}
		\lambda_0 = \lambda_0^I -\frac{\uu}{\vv} \lambda_0^{F,\wun} \quad \text{($w = \wun$)} \qquad \text{or} \qquad
		\lambda_0 = \lambda_0^I+\lambda_1^I - \frac{\uu}{\vv}\brac*{\lambda_0^{F,\wref{1}}+\lambda_1^{F,\wref{1}}} + 1 \quad \text{($w = \wref{1}$)}.
	\end{equation}
	Consider first a $w=\wun$ admissible weight $\lambda$.  Since $\lambda^{F,\wun} \in \pwlat{\vv-1}$, we clearly have $\lambda_0 \in \ZZ$ if and only if $\lambda^{F,\wun}_0 = 0$.  On the other hand, a $w=\wref{1}$ admissible weight $\lambda$ necessarily has $0 < \lambda_0^{F,\wref{1}} + \lambda_1^{F,\wref{1}} < \vv$, since $\lambda^{F,\wref{1}} \in \pwlat{\vv-1}$ and $\lambda_1^{F,\wref{1}} \ge 1$.  It follows that the Dynkin label $\lambda_0$ can never be an integer in this case.  The first two statements are thus established.

	For the third, let $\mu$ be a level-$\kk$ weight.  Explicit calculation shows that the Dynkin labels of $\wref{0} \cdot \wref{1} \cdot \mu$ are
	\begin{equation}
		\sqbrac[\Big]{\mu_2 - \frac{\uu}{\vv},\ \mu_0,\ \mu_1 + \frac{\uu}{\vv}}.
	\end{equation}
	Let $\lambda = \wref{1} \cdot \brac*{\lambda^I - \frac{\uu}{\vv} \lambda^{F,\wref{1}}}$ be a $w=\wref{1}$ admissible weight.  Then, $\wref{0} \cdot \lambda$ has the form $\mu = \mu^I - \frac{\uu}{\vv} \mu^{F,\wun}$ with
	\begin{equation}
			\mu^I = \sqbrac[\Big]{\lambda^I_2, \ \lambda^I_0, \ \lambda^I_1} \quad \text{and} \quad
			\mu^{F,\wun} = \sqbrac[\Big]{\lambda^{F,\wref{1}}_2 +1, \ \lambda^{F,\wref{1}}_0, \ \lambda^{F,\wref{1}}_1 -1}.
	\end{equation}
	It is easy to see that $\mu^I \in \pwlat{\uu-3}$ and $\mu^{F,\wun} \in \pwlat{\vv-1}$, so $\mu$ is a $w=\wun$ admissible weight.  Moreover, $\mu^{F,\wun}_0 \ge 1$ implies that $\mu$ is surviving.  Since $\wref{0} \cdot (\blank)$ is clearly self-inverse, we have the desired bijection between $w=\wun$ surviving weights and $w=\wref{1}$ admissible weights.  To show that it is $(j,\Delta)$-preserving, we show that the functions $j(\lambda)$ and $\Delta(\lambda)$ defined by \eqref{eq:bphwfromsl3} are invariant under $\lambda \mapsto \wref{0} \cdot \lambda$.  This is clear from $(\wref{0} \cdot \lambda)_1 = \kk+1-\lambda_2$ and $(\wref{0} \cdot \lambda)_2 = \kk+1-\lambda_1$.

	Finally, let $\lambda$ and $\mu$ be surviving weights and suppose that $\qhrfun{\slihw{\lambda}} \cong \qhrfun{\slihw{\mu}}$, so that $j(\lambda) = j(\mu)$ and $\Delta(\lambda) = \Delta(\mu)$.  We have just seen that $\lambda$ and $\wref{0} \cdot \lambda$ always give the same $j$ and $\Delta$.  But, if $\lambda$ is a $w=\wun$ surviving weight, then $\mu = \wref{0} \cdot \lambda$ is a $w=\wref{1}$ surviving weight.  Since the intersection of the sets of $w=\wun$ and $w=\wref{1}$ admissible weights is empty \cite[Prop.~2.1]{KacCla89}, we have $\lambda \ne \mu$.  As there are at most two weights corresponding to a given choice of $j$ and $\Delta$ (\cref{thm:qhr}), this shows that there are never two distinct $w=\wun$ surviving weights giving the same $j$ and $\Delta$.
\end{proof}
\noindent In what follows, a surviving weight shall be understood to mean a $w=\wun$ surviving weight unless otherwise indicated.  The set of ($w=\wun$) surviving level-$\kk$ weights will be denoted by $\survk$.  We shall also start dropping the label $w$ from $\lambda^{F,w}$, understanding that we mean $w=\wun$ unless otherwise indicated.

Let $\slmpik$ denote the maximal ideal of $\uslvoa{\kk}$, so that $\sslvoa{\kk} = \uslvoa{\kk} / \slmpik$.  If $\kk$ is an admissible level, then \cref{thm:affhwclass} says that $\slmpik \cdot \slihw{\lambda} = 0$ if and only if $\lambda$ is an admissible weight.  If, in addition, $\vv\ge2$, then
\begin{equation}
	\qhrfun{\sslvoa{\kk}} = \qhrfun{\slihw{\kk \fwt{0}}} \cong \ihw{0,0} = \sbpvoak,
\end{equation}
by \cref{thm:qhr}.  Moreover, the exactness of $\qhrfun{\blank}$ means that the maximal ideal $\bpmpik$ of $\ubpvoak$ is then isomorphic to $\qhrfun{\slmpik}$.  It follows that $\qhrfun{\slihw{\lambda}}$ is a $\sbpvoak$-module if and only if $\qhrfun{\slmpik} \cdot \qhrfun{\slihw{\lambda}} = 0$.

Recall that $\qhrfun{\blank}$ corresponds to tensoring with a ghost \svoa\ $\ghvoa$, graded by the fermionic ghost number, and taking the degree-$0$ cohomology with respect to a given differential (see \cref{sec:qhr} for the details).  Denote the cohomology class of a (degree-$0$) cocycle $a$ by $[a]$ (we trust that this notation will not be confused with the notation for Zhu algebra images in \cref{sec:ubpmods}).  Given (degree-$0$) cocycles $a$ and $v$ of the BRST complexes $\slmpik\otimes \ghvoa$ and $\slihw{\lambda}\otimes \ghvoa$, respectively, the action of $[a] \in \qhrfun{\slmpik}$ on $[v] \in \qhrfun{\slihw{\lambda}}$ is given by $[a] \cdot [v] \equiv [a](z) [v] = [a(z) v] \in \qhrfun{\slmpik \cdot \slihw{\lambda}}$.  For $\lambda$ admissible, we therefore obtain
\begin{equation}
	\qhrfun{\slmpik} \cdot \qhrfun{\slihw{\lambda}} \subseteq \qhrfun{\slmpik \cdot \slihw{\lambda}} = 0.
\end{equation}
This proves the following assertion.
\begin{proposition} \label{prop:sl=>bp}
	Let $\kk$ be admissible with $\vv\ge2$.  If $\slihw{\lambda}$ is an $\sslvoak$-module, then $\qhrfun{\slihw{\lambda}}$ is a $\sbpvoak$-module.
\end{proposition}
\noindent This also motivates the following assumption, which we shall understand to be in force for everything that follows.
\begin{assumption} \label{assume}
	In what follows, we shall restrict to fractional levels $\kk=-3+\frac{\uu}{\vv}$ with $\uu \ge 3$ and $\vv \ge 2$.  The restriction on $\uu$ means that $\kk$ is an admissible level for $\slthree$, whilst the restriction on $\vv$ guarantees that the minimal \qhr\ of $\sslvoak = \slminmoduv$ is $\sbpvoak = \bpminmoduv$ (for $\uu \ge 3$, we have $\qhrfun{\slminmod{\uu}{1}} = 0$ instead).
\end{assumption}

Of course, to obtain a classification of simple \hw\ $\sbpvoak$-modules from Arakawa's classification of simple \hw\ $\sslvoak$-modules (\cref{thm:affhwclass}), we need a converse of \cref{prop:sl=>bp}.  This is much more subtle.
\begin{theorem} \label{thm:surjection}
	Let $\kk$ be as in \cref{assume}.  Then, every simple \hw\ $\sbpvoak$-module is isomorphic to the minimal \qhr\ of some simple \hw\ $\sslvoak$-module.
\end{theorem}
\noindent Note that if $\lambda_0 \in \NN$, then $\qhrfun{\slihw{\lambda}} = 0$ is a $\sbpvoak$-module, irrespective of whether or not it is an $\sslvoak$-module.  It is therefore enough to show that if $\lambda_0 \notin \NN$ and $\slihw{\lambda}$ is not a $\sbpvoak$-module, then $\qhrfun{\slihw{\lambda}}$ is not a $\sbpvoak$-module.  Equivalently, we must show that $\lambda_0 \notin \NN$ and $\slmpik \cdot \slihw{\lambda} \ne 0$ implies that $\qhrfun{\slmpik} \cdot \qhrfun{\slihw{\lambda}} \ne 0$.  We defer the somewhat intricate proof of this assertion to \cref{app:proof}.

\subsection{Simple \hw\ $\bpminmoduv$-modules} \label{sec:hwmod}

From \cref{thm:affhwclass,lem:survivors,thm:surjection}, we conclude that the $\qhrfun{\slihw{\lambda}}$, with $\lambda \in \survk$, form a complete set of mutually nonisomorphic simple untwisted \hw\ modules for the \bp\ minimal model \voa\ $\bpminmoduv$ (assuming that the level is as in \cref{assume}).  The charge ($J_0$-eigenvalue) $j$ and conformal weight ($L_0$-eigenvalue) $\Delta$ of the \hwv\ of $\qhrfun{\slihw{\lambda}}$ was given in \eqref{eq:bphwfromsl3}: $\qhrfun{\slihw{\lambda}} \cong \ihw{j,\Delta}$.  This is then a classification of the simple untwisted \hw\ $\bpminmoduv$-modules.  Moreover, \cref{prop:twvsuntw} extends this to a classification of their twisted cousins.
\begin{theorem} \label{thm:hwclass}
	Let $\kk$ be as in \cref{assume}.  Then:
	\begin{itemize}
		\item Every simple untwisted \hw\ $\bpminmoduv$-module is isomorphic to one of the $\ihw{j,\Delta}$, where $j$ and $\Delta$ are determined from the Dynkin labels of a unique surviving weight $\lambda \in \survk$ by \eqref{eq:bphwfromsl3}.
		\item Every simple twisted \hw\ $\bpminmoduv$-module is isomorphic to one of the $\twihw{j,\Delta}$, where $j$ and $\Delta$ are determined from the Dynkin labels of a unique surviving weight $\lambda \in \survk$ by
		\begin{equation} \label{eq:bptwhwfromsl3}
			j = \frac{\lambda_1 - \lambda_2}{3} + \frac{2\kk+3}{6} \quad \text{and} \quad
			\Delta = \frac{(\lambda_1-\lambda_2)^2 - 3(\lambda_1+\lambda_2) \brac[\big]{2(\kk+1)-\lambda_1-\lambda_2}}{12(\kk+3)} + \frac{\lambda_1 - \lambda_2}{6} + \frac{2\kk+3}{24}.
		\end{equation}
	\end{itemize}
	Moreover, the $\ihw{j,\Delta}$ and $\twihw{j,\Delta}$ determined by the surviving weights are all mutually nonisomorphic.
\end{theorem}

In light of this classification, we let $\ihw{\lambda} = \ihw{j,\Delta}$ and $\twihw{\lambda} = \twihw{j,\Delta}$, where $j$ and $\Delta$ are given in terms of $\lambda \in \survk$ by \eqref{eq:bphwfromsl3} and \eqref{eq:bptwhwfromsl3}, respectively.  Note that this implies that %DR Verma notation moved into proof below as Vermas not yet defined.
\begin{equation} \label{eq:tw=sf}
	\twihw{\lambda} \cong \sfmod{1/2}{\ihw{\lambda}},
\end{equation}
by \cref{prop:twvsuntw}.  With this new notation, the vacuum module $\ihw{0,0}$ is identified as $\ihw{\lambda}$, where $\lambda = [k,0,0]$ has $\lambda^I = [\uu-3,0,0]$ and $\lambda^F = [\vv-1,0,0]$.

We record the following strengthening of \cref{thm:hwclass}, following \cite[Thm.~10.10]{AraRat15}, for later use.
\begin{theorem} \label{thm:hwrat}
	Let $\kk$ be as in \cref{assume}.  Then, every \hw\ $\bpminmoduv$-module, untwisted or twisted, is simple.
\end{theorem}
\begin{proof}
	We prove this for untwisted modules as the twisted case follows immediately from \eqref{eq:tw=sf} and the invertibility of spectral flow.  Since the simple quotient of any \hw\ $\bpminmoduv$-module $\ihw{}$ is isomorphic to some $\ihw{\lambda}$ with $\lambda \in \survk$, by \cref{thm:hwclass}, it is enough to show that $\ihw{}$ cannot have a composition factor isomorphic to $\ihw{\mu}$ for some $\mu \in \survk$ distinct from $\lambda$.  Indeed, it is enough to show that the Verma module $\vhw{\lambda} = \vhw{j,\Delta}$ of $\ubpvoak$ does not have such a composition factor.

	Recall that $\slvhw{\lambda}$ denotes the Verma module of $\uslvoak$ of highest weight $\lambda$ and let $\mult{\slvhw{\lambda}}{\slihw{\nu}}$ denote the multiplicity with which $\slihw{\nu}$ appears as a composition factor of $\slvhw{\lambda}$.  By \cref{thm:qhr}, \qhr\ takes $\slvhw{\lambda}$ to $\vhw{\lambda}$ and only $\slihw{\mu}$ and $\slihw{\wref{0} \cdot \mu}$ are sent to $\ihw{\mu}$.  As reduction is exact, we must have $\mult{\vhw{\lambda}}{\ihw{\mu}} = \mult{\slvhw{\lambda}}{\slihw{\mu}} + \mult{\slvhw{\lambda}}{\slihw{\wref{0} \cdot \mu}}$ (noting that $\mu$ and $\wref{0} \cdot \mu$ are distinct since $\mu \in \survk$).

	It follows that if $\vhw{\lambda}$ has $\ihw{\mu}$, $\mu \ne \lambda$, as a composition factor, then $\slvhw{\lambda}$ has either $\slihw{\mu}$ or $\slihw{\wref{0} \cdot \mu}$ as a composition factor.  But, $\lambda$, $\mu$ and $\wref{0} \cdot \mu$ are all admissible $\aslthree$-weights (corresponding to $w=\wun$, $\wun$ and $\wref{1}$, respectively, see \cref{lem:survivors}), hence they are dominant.  This is therefore impossible by the linkage principle for Verma $\aslthree$-modules.
\end{proof}
\noindent Because the \bgg\ category $\bpocat{\uu,\vv}$ of level-$\kk$ $\bpminmoduv$-modules admits contragredient duals, it follows from \cref{thm:hwrat} that every extension between $\ihw{\lambda}$ and $\ihw{\mu}$, with $\lambda \ne \mu$, splits.  It is likewise easy to see that a nonsplit self-extension of $\ihw{\lambda}$ requires a nonsemisimple action of $J_0$ or $L_0$ (which is forbidden in $\bpocat{\uu,\vv}$).  $\bpocat{\uu,\vv}$ is thus semisimple and, by \cref{thm:hwclass}, has finitely many isomorphism classes of simple objects.  We may therefore summarise this as follows: $\bpminmoduv$ is \emph{rational in category $\bpocat{\uu,\vv}$}.

In order to extend the \hw\ classification of \cref{thm:hwclass} to twisted \rhw\ $\bpminmoduv$-modules, we need to know when the top space $\top{(\twihw{j,\Delta})} = \twzhu{\twihw{j,\Delta}}$ is \infdim.  The condition for this is beautifully succinct when expressed in terms of surviving weights.

\begin{proposition} \label{prop:fdimtop}
	The top space of the simple twisted \hw\ $\bpminmoduv$-module $\twihw{\lambda}$ is \fdim\ if and only if $\lambda^F_1 = 0$.  When $\lambda^F_1 = 0$, the dimension of this top space is $\lambda^I_1 + 1$.
\end{proposition}
\begin{proof}
	By \cref{prop:hwzkmod}, $\top{(\twihw{j,\Delta})}$ is \fdim\ if and only if $h_{\kk}^n(j,\Delta) = 0$ for some $n \in \ZZ_{\ge1}$ and, if it is \fdim, then the dimension is the smallest such $n$.  Substituting \eqref{eq:bptwhwfromsl3} into the definition \eqref{eq:defh} of $h_{\kk}^n$ and simplifying, we find that
	\begin{equation}
		h_{\kk}^n(j,\Delta) = n \brac*{n-\lambda_1-1} \brac*{n+\lambda_2+1-\frac{\uu}{\vv}}.
	\end{equation}
	The only possible roots in $\ZZ_{\ge1}$ are thus $n=\lambda_1+1$ and $n=\frac{\uu}{\vv}-\lambda_2-1$.  As $\lambda = \lambda^I - \frac{\uu}{\vv} \lambda^F$, the former requires $\lambda_1 \in \ZZ$ so $\lambda^F_1 = 0$ and $n = \lambda^I_1+1 \in \ZZ_{\ge1}$.  On the other hand, the latter requires $n=-(\lambda^I_2+1) + \frac{\uu}{\vv}(\lambda^F_2+1)$ which is only an integer if $\lambda^F_2 = \vv-1$.  However, this contradicts $\lambda^F \in \pwlat{\vv-1}$ and $\lambda^F_0 \ge 1$ (\cref{lem:survivors}).
\end{proof}
\begin{corollary} \label{cor:howmanyhwms}
	Given $\kk$ as in \cref{assume}, there are (up to isomorphism):
	\begin{itemize}
		\item $\frac{1}{4} (\uu-1)(\uu-2)\vv(\vv-1)$ simple untwisted \hw\ $\bpminmoduv$-modules;
		\item $\frac{1}{2} (\uu-1)(\uu-2)(\vv-1)$ simple twisted \hw\ $\bpminmoduv$-modules with \fdim\ top spaces;
		\item $\frac{1}{4} (\uu-1)(\uu-2)(\vv-1)(\vv-2)$ simple twisted \hw\ $\bpminmoduv$-modules with \infdim\ top spaces;
	\end{itemize}
\end{corollary}
\noindent In particular, there are no simple twisted \hw\ $\bpminmoduv$-modules with \infdim\ top spaces when $\vv=2$.  This is in accord with the fact that the $\bpminmod{\uu}{2}$ with $\uu \ge 3$ are rational \cite{AraRat13}.

Recall that the conjugation automorphism $\conjsymb$ of $\bpminmoduv$, given in \eqref{eq:auts}, negates $J_0$ and preserves $L_0$.  At the level of their eigenvalues, this is effected in \eqref{eq:bphwfromsl3} by exchanging the Dynkin labels $\lambda_1$ and $\lambda_2$ of $\lambda$.  The result of this exchange is clearly still a surviving weight, by \cref{lem:survivors}.
\begin{proposition} \label{prop:hwconj}
	For each $\lambda \in \survk$, we have:
	\begin{itemize}
		\item $\conjmod{\ihw{[\lambda_0,\lambda_1,\lambda_2]}} \cong \ihw{[\lambda_0,\lambda_2,\lambda_1]}$.
		\item If $\lambda^F_1 = 0$, then $\conjmod{\twihw{\lambda}} \cong \twihw{\mu}$, where $\mu = [\lambda_2 - \frac{\uu}{\vv}, \lambda_1, \lambda_0 + \frac{\uu}{\vv}]$, hence $\mu^I = [\lambda^I_2,\lambda^I_1,\lambda^I_0]$ and $\mu^F = [\lambda^F_2+1,0,\lambda^F_0-1]$.  Otherwise, $\conjmod{\twihw{[\lambda_0,\lambda_1,\lambda_2]}}$ is not \hw\ (though it is \rhw).
	\end{itemize}
\end{proposition}
\begin{proof}
	The result of conjugating a simple untwisted \hw\ $\bpminmoduv$-module is clear from the above remarks, because the top spaces are one-dimensional.  For the twisted case, first note that the conjugate of $\twihw{\lambda}$ will be again \hw\ if its top space is \fdim\ (otherwise the top space of the conjugate module will be an \infdim\ \lw\ $\zk$-module).  By \cref{prop:fdimtop}, this requires $\lambda^F_1 = 0$, hence $\lambda_1 = \lambda^I_1$.  Assuming this, let $j$ and $\Delta$ denote the charge and conformal weight, respectively, of the \hwv\ of $\twihw{\lambda}$.  Then, the \hwv\ of $\conjmod{\twihw{\lambda}}$ has charge $\lambda_1-j$ and conformal weight $\Delta$.

	We therefore need to find $\mu \in \survk$ corresponding to these eigenvalues under \eqref{eq:bptwhwfromsl3}.  Solving for $\mu$, we find two solutions:
	\begin{equation}
		\begin{aligned}
			\mu_0 &= \lambda_2-\kk-3, & \mu_1 &= \lambda_1 & &\text{and} & \mu_2 &= \lambda_0+\kk+3, \\
			\text{or} \quad
			\mu_0 &= \kk+1-\lambda_2, & \mu_1 &= -\lambda_0-2 & &\text{and} & \mu_2 &= \kk+1-\lambda_1.
		\end{aligned}
	\end{equation}
	We know from the proof of \cref{lem:survivors} that only one of these is a $w=\wun$ surviving weight and the other is a $w=\wref{1}$ survivor obtained from the $w=\wun$ one by applying the shifted action of $\wref{0}$.  It is easy to check that the first solution is the $w=\wun$ survivor by writing it in the form
	\begin{equation}
		\mu_0 = \lambda^I_2 - \frac{\uu}{\vv} (\lambda^F_2+1), \quad
		\mu_1 = \lambda^I_1 \quad \text{and} \quad
		\mu_2 = \lambda^I_0 - \frac{\uu}{\vv} (\lambda^F_0-1).
	\end{equation}
	Indeed, $\lambda^F_0 \ge 1$ implies that $\mu^I = [\lambda^I_2,\lambda^I_1,\lambda^I_0] \in \pwlat{\uu-3}$, $\mu^F = [\lambda^F_2+1,0,\lambda^F_0-1] \in \pwlat{\vv-1}$ and $\mu^F_0 \ge 1$, hence that $\mu \in \survk$.
\end{proof}

It remains to determine when the spectral flow of a simple \hw\ $\bpminmoduv$-module is another such module.  By \cref{prop:twvsuntw}, it suffices to consider the untwisted case.  Again the key is the \fdim ity of the top space: $\sfmod{}{\ihw{\lambda}}$ will be \hw\ if and only if $\twihw{\lambda} = \sfmod{1/2}{\ihw{\lambda}}$ has a \fdim\ top space, that is if and only if $\lambda^F_1 = 0$ (\cref{prop:fdimtop}).  Indeed, if $\lambda^F_1 = 0$ and $v$ denotes the \hwv\ of $\ihw{\lambda}$, then that of $\sfmod{}{\ihw{\lambda}}$ is easily checked to be $(G^-_{1/2})^{\lambda^I_1} \sfmod{}{v}$.  We compute its charge and conformal weight, then determine the (unique $w=\wun$) surviving weight that gives these eigenvalues under \eqref{eq:bphwfromsl3}, as in the proof of \cref{prop:hwconj}.  We thereby obtain the following \lcnamecref{prop:sfhw}.
\begin{proposition} \label{prop:sfhw}
	If $\lambda \in \survk$ satisfies $\lambda^F_1 = 0$, then $\sfmod{}{\ihw{\lambda}} \cong \ihw{\mu}$, where $\mu = [\lambda_2 - \frac{\uu}{\vv}, \lambda_0 + \frac{\uu}{\vv}, \lambda_1] \in \survk$, hence $\mu^I = [\lambda^I_2,\lambda^I_0,\lambda^I_1]$ and $\mu^F = [\lambda^F_2+1,\lambda^F_0-1,0]$.  If $\lambda^F_1 \ne 0$, then $\sfmod{}{\ihw{\lambda}}$ is not \hw\ (nor \rhw).
\end{proposition}
\noindent Combining this with the dihedral relation \eqref{eq:dihedral} and \cref{prop:hwconj}, we obtain the following characterisation of the spectral flow orbit of a simple untwisted \hw\ $\bpminmoduv$-module $\ihw{\lambda}$.  We recall from \cref{prop:twvsuntw} that a twisted member $\sfmod{\ell+1/2}{\ihw{\lambda}}$, $\ell \in \ZZ$, of this orbit is \hw\ if and only if its untwisted predecessor $\sfmod{\ell}{\ihw{\lambda}}$ is.
\begin{theorem} \label{thm:sfhworbits}
	Take $\lambda \in \survk$ and define $\mu, \nu, \bar{\mu}, \bar{\nu} \in \survk$ by
	\begin{equation} \label{eq:auxsurv}
		\begin{aligned}
			\mu^I &= [\lambda^I_2,\lambda^I_0,\lambda^I_1], & \mu^F &= [\lambda^F_2+1,\lambda^F_0-1,0], \\
			\bar{\mu}^I &= [\lambda^I_1,\lambda^I_2,\lambda^I_0], & \bar{\mu}^F &= [\lambda^F_1+1,0,\lambda^F_0-1]
		\end{aligned}
		\qquad \text{and} \qquad
		\begin{aligned}
			\nu^I &= [\lambda^I_1,\lambda^I_2,\lambda^I_0], & \nu^F &= [1,\vv-2,0], \\
			\bar{\nu}^I &= [\lambda^I_2,\lambda^I_0,\lambda^I_1], & \bar{\nu}^F &= [1,0,\vv-2].
		\end{aligned}
	\end{equation}
	\begin{itemize}
		\item $\sfmod{}{\ihw{\lambda}}$ is \hw\ if and only if $\lambda^F_1=0$.  In this case, $\sfmod{}{\ihw{\lambda}} \cong \ihw{\mu}$.
		\item $\sfmod{-1}{\ihw{\lambda}}$ is \hw\ if and only if $\lambda^F_2=0$.  In this case, $\sfmod{-1}{\ihw{\lambda}} \cong \ihw{\bar{\mu}}$.
		\item $\sfmod{2}{\ihw{\lambda}}$ is \hw\ if and only if $\lambda^F=[1,0,\vv-2]$.  In this case, $\sfmod{2}{\ihw{\lambda}} \cong \ihw{\nu}$.
		\item $\sfmod{-2}{\ihw{\lambda}}$ is \hw\ if and only if $\lambda^F=[1,\vv-2,0]$.  In this case, $\sfmod{-2}{\ihw{\lambda}} \cong \ihw{\bar{\nu}}$.
		\item For $\abs{\ell} \in \ZZ_{\ge3}$, $\sfmod{\ell}{\ihw{\lambda}}$ is \hw\ if and only if $\vv=2$.  In this case, $\sfmod{\pm3}{\ihw{\lambda}} \cong \ihw{\lambda}$.
	\end{itemize}
\end{theorem}
\noindent Note that when $\vv=2$, every $\lambda \in \survk$ has $\lambda^F = [1,0,0]$.  The spectral flow orbits thus take the form
\begin{equation}
	\cdots \lma \ihw{\lambda} \lma \twihw{\lambda} \lma \ihw{\mu} \lma \twihw{\mu} \lma \ihw{\nu} \lma \twihw{\nu} \lma \ihw{\lambda} \lma \cdots,
\end{equation}
where $\mu$ and $\nu$ are as in \eqref{eq:auxsurv} (with $\mu^F = \nu^F = [1,0,0]$).  We picture the $\vv\ge3$ spectral flow orbits in \cref{fig:sforbits}.
\begin{figure}
	\makebox[\textwidth]{
		\begin{tikzpicture}[scale=0.7]
			\begin{scope}[shift={(-6.5,0)}]
				\node at (0,0.5) {$\cdots \lma$};
			\end{scope}
			\begin{scope}[shift={(-5.5,0)}]
				\draw (1,1.5) -- (0,0.5) -- (-0.33,-0.5);
				\fill[lightgray] (1,1.4) -- (0.1,0.47) -- (-0.23,-0.5) -- (1,-0.5);
			\end{scope}
			\begin{scope}[shift={(-4,0)}]
				\node at (0,0.5) {$\lma$};
			\end{scope}
			\begin{scope}[shift={(-3,0)}]
				\draw (1,1) -- (0,1) -- (-0.5,0);
				\fill[lightgray] (1,0.9) -- (0.05,0.9) -- (-0.4,0) -- (1,0);
			\end{scope}
			\begin{scope}[shift={(-1.5,0)}]
				\node at (0,0.5) {$\lma$};
			\end{scope}
			\begin{scope}[shift={(0,0)}]
				\draw (-1,0) -- (0,1) -- (1,0);
				\fill[lightgray] (-0.9,0) -- (0,0.9) -- (0.9,0);
				\node at (0,-0.5) {$\lambda^F_1, \lambda^F_2 \ne 0$};
			\end{scope}
			\begin{scope}[shift={(1.5,0)}]
				\node at (0,0.5) {$\lma$};
			\end{scope}
			\begin{scope}[shift={(3,0)}]
				\draw (-1,1) -- (0,1) -- (0.5,0);
				\fill[lightgray] (-1,0.9) -- (-0.05,0.9) -- (0.4,0) -- (-1,0);
			\end{scope}
			\begin{scope}[shift={(4,0)}]
				\node at (0,0.5) {$\lma$};
			\end{scope}
			\begin{scope}[shift={(5.5,0)}]
				\draw (-1,1.5) -- (0,0.5) -- (0.33,-0.5);
				\fill[lightgray] (-1,1.4) -- (-0.1,0.47) -- (0.23,-0.5) -- (-1,-0.5);
			\end{scope}
			\begin{scope}[shift={(6.5,0)}]
				\node at (0,0.5) {$\lma \cdots$};
			\end{scope}
			\begin{scope}[shift={(-2.75,-3)}]
				\begin{scope}[shift={(-6.5,0)}]
					\node at (0,0.5) {$\cdots \lma$};
				\end{scope}
				\begin{scope}[shift={(-5.5,0)}]
					\draw (1,1.5) -- (0,0.5) -- (-0.33,-0.5);
					\fill[lightgray] (1,1.4) -- (0.1,0.47) -- (-0.23,-0.5) -- (1,-0.5);
				\end{scope}
				\begin{scope}[shift={(-4,0)}]
					\node at (0,0.5) {$\lma$};
				\end{scope}
				\begin{scope}[shift={(-3,0)}]
					\draw (1,1) -- (0,1) -- (-0.5,0);
					\fill[lightgray] (1,0.9) -- (0.05,0.9) -- (-0.4,0) -- (1,0);
				\end{scope}
				\begin{scope}[shift={(-1.5,0)}]
					\node at (0,0.5) {$\lma$};
				\end{scope}
				\begin{scope}[shift={(0,0)}]
					\draw (-1,0) -- (0,1) -- (1,0);
					\fill[lightgray] (-0.9,0) -- (0,0.9) -- (0.9,0);
					\node at (0,-0.75) {$\begin{matrix} \lambda^F_1=0 \\ \lambda^F_0 \ne 1,\ \lambda^F_2 \ne 0 \end{matrix}$};
				\end{scope}
				\begin{scope}[shift={(1.5,0)}]
					\node at (0,0.5) {$\lma$};
				\end{scope}
				\begin{scope}[shift={(2.75,0)}]
					\draw (-0.75,0) -- (-0.25,1) -- (0.25,1) -- (0.75,0);
					\fill[lightgray] (-0.65,0) -- (-0.2,0.9) -- (0.2,0.9) -- (0.65,0);
				\end{scope}
				\begin{scope}[shift={(4,0)}]
					\node at (0,0.5) {$\lma$};
				\end{scope}
				\begin{scope}[shift={(5.5,0)}]
					\draw (-1,0) -- (0,1) -- (1,0);
					\fill[lightgray] (-0.9,0) -- (0,0.9) -- (0.9,0);
					\node at (0,-0.75) {$\begin{matrix} \lambda^F_2=0 \\ \lambda^F_0 \ne 1,\ \lambda^F_1 \ne 0 \end{matrix}$};
				\end{scope}
				\begin{scope}[shift={(7,0)}]
					\node at (0,0.5) {$\lma$};
				\end{scope}
				\begin{scope}[shift={(8.5,0)}]
					\draw (-1,1) -- (0,1) -- (0.5,0);
					\fill[lightgray] (-1,0.9) -- (-0.05,0.9) -- (0.4,0) -- (-1,0);
				\end{scope}
				\begin{scope}[shift={(9.5,0)}]
					\node at (0,0.5) {$\lma$};
				\end{scope}
				\begin{scope}[shift={(11,0)}]
					\draw (-1,1.5) -- (0,0.5) -- (0.33,-0.5);
					\fill[lightgray] (-1,1.4) -- (-0.1,0.47) -- (0.23,-0.5) -- (-1,-0.5);
				\end{scope}
				\begin{scope}[shift={(12,0)}]
					\node at (0,0.5) {$\lma \cdots$};
				\end{scope}
			\end{scope}
			\begin{scope}[shift={(-5.5,-6.5)}]
				\begin{scope}[shift={(-6.5,0)}]
					\node at (0,0.5) {$\cdots \lma$};
				\end{scope}
				\begin{scope}[shift={(-5.5,0)}]
					\draw (1,1.5) -- (0,0.5) -- (-0.33,-0.5);
					\fill[lightgray] (1,1.4) -- (0.1,0.47) -- (-0.23,-0.5) -- (1,-0.5);
				\end{scope}
				\begin{scope}[shift={(-4,0)}]
					\node at (0,0.5) {$\lma$};
				\end{scope}
				\begin{scope}[shift={(-3,0)}]
					\draw (1,1) -- (0,1) -- (-0.5,0);
					\fill[lightgray] (1,0.9) -- (0.05,0.9) -- (-0.4,0) -- (1,0);
				\end{scope}
				\begin{scope}[shift={(-1.5,0)}]
					\node at (0,0.5) {$\lma$};
				\end{scope}
				\begin{scope}[shift={(0,0)}]
					\draw (-1,0) -- (0,1) -- (1,0);
					\fill[lightgray] (-0.9,0) -- (0,0.9) -- (0.9,0);
					\node at (0,-0.5) {$\lambda^F = [1,0,\vv-2]$};
				\end{scope}
				\begin{scope}[shift={(1.5,0)}]
					\node at (0,0.5) {$\lma$};
				\end{scope}
				\begin{scope}[shift={(2.75,0)}]
					\draw (-0.75,0) -- (-0.25,1) -- (0.25,1) -- (0.75,0);
					\fill[lightgray] (-0.65,0) -- (-0.2,0.9) -- (0.2,0.9) -- (0.65,0);
				\end{scope}
				\begin{scope}[shift={(4,0)}]
					\node at (0,0.5) {$\lma$};
				\end{scope}
				\begin{scope}[shift={(5.5,0)}]
					\draw (-1,0) -- (0,1) -- (1,0);
					\fill[lightgray] (-0.9,0) -- (0,0.9) -- (0.9,0);
					\node at (0,-0.5) {$\lambda^F = [\vv-1,0,0]$};
				\end{scope}
				\begin{scope}[shift={(7,0)}]
					\node at (0,0.5) {$\lma$};
				\end{scope}
				\begin{scope}[shift={(8.25,0)}]
					\draw (-0.75,0) -- (-0.25,1) -- (0.25,1) -- (0.75,0);
					\fill[lightgray] (-0.65,0) -- (-0.2,0.9) -- (0.2,0.9) -- (0.65,0);
				\end{scope}
				\begin{scope}[shift={(9.5,0)}]
					\node at (0,0.5) {$\lma$};
				\end{scope}
				\begin{scope}[shift={(11,0)}]
					\draw (-1,0) -- (0,1) -- (1,0);
					\fill[lightgray] (-0.9,0) -- (0,0.9) -- (0.9,0);
					\node at (0,-0.5) {$\lambda^F = [1,\vv-2,0]$};
				\end{scope}
				\begin{scope}[shift={(12.5,0)}]
					\node at (0,0.5) {$\lma$};
				\end{scope}
				\begin{scope}[shift={(14,0)}]
					\draw (-1,1) -- (0,1) -- (0.5,0);
					\fill[lightgray] (-1,0.9) -- (-0.05,0.9) -- (0.4,0) -- (-1,0);
				\end{scope}
				\begin{scope}[shift={(15,0)}]
					\node at (0,0.5) {$\lma$};
				\end{scope}
				\begin{scope}[shift={(16.5,0)}]
					\draw (-1,1.5) -- (0,0.5) -- (0.33,-0.5);
					\fill[lightgray] (-1,1.4) -- (-0.1,0.47) -- (0.23,-0.5) -- (-1,-0.5);
				\end{scope}
				\begin{scope}[shift={(17.5,0)}]
					\node at (0,0.5) {$\lma \cdots$};
				\end{scope}
			\end{scope}
		\end{tikzpicture}
	}
\caption{A picture of the weights of the three types of spectral flow orbits through a simple \hw\ $\bpminmoduv$-module with $\vv \ge 3$.  The charge increases from left to right, whilst the conformal weight increases from top to bottom.  The given constraints on the Dynkin labels of $\lambda^F$ must be satisfied by the simple untwisted \hw\ $\bpminmoduv$-module $\ihw{\lambda}$ appearing at that point in the orbit.  Note that the unpictured modules in each infinite orbit, indicated by $\cdots$, are neither \hw\ nor \rhw: their conformal weights are unbounded below.} \label{fig:sforbits}
\end{figure}

\subsection{Simple \rhw{} $\bpminmoduv$-modules} \label{sec:rhwmod}

As we noted in \cref{thm:untwubpclass}, every simple untwisted \rhw\ $\bpminmoduv$-module is \hw.  The classification of simple untwisted \rhwms\ was therefore completed in \cref{thm:hwclass}.  It remains to classify the simple twisted \rhwms, specifically those whose top spaces are simple dense $\zk$-modules (those whose top spaces are simple \lw\ $\zk$-modules are conjugates of the simple twisted \hw\ $\bpminmoduv$-modules classified in \cref{thm:hwclass}).

A simple twisted \rhw\ $\ubpvoak$-module $\Mod{M}$ is a $\sbpvoak$-module if and only if its top space $\top{\Mod{M}} = \twzhu{\Mod{M}}$ is annihilated by $\twzhu{\bpmpik}$, where $\bpmpik$ denotes the maximal ideal of $\ubpvoak$.  An obvious consequence of \cref{thm:hwclass} is that $\twzhu{\bpmpik}$ annihilates $\twzhu{\twihw{\lambda}} \cong \zihw{j,\Delta}$, with $j$ and $\Delta$ determined by $\lambda$ as in \eqref{eq:bptwhwfromsl3}, if and only if $\lambda \in \survk$.  We extend this to the simple \rhwms\ $\twrhw{[j],\Delta,\omega}$ of \cref{thm:classtwrhw} using an argument similar to that of \cite[Prop.~4.2]{KawRel19}.

\begin{proposition} \label{thm:rhwclass}
	The irreducible semisimple coherent family $\zsscfam{\Delta,\omega}$ of $\zk$-modules is a $\twzhu{\bpminmoduv}$-module if and only if one of its \emph{\infdim} submodules is.
\end{proposition}
\begin{proof}
	Obviously, $\zsscfam{\Delta,\omega}$ being a $\twzhu{\bpminmoduv}$-module implies that every one of its submodules are too, in particular the \infdim\ ones.

	To prove the converse, we lean heavily on the general methodology developed in \cite{KawRel19} to classify \rhwms\ for affine \voas, though the argument here is easier because the relevant coherent families have one-dimensional weight spaces.  The first step is to consider the subalgebra $\ak = \twzhu{\bpmpik} \cap \ck$, where we recall that $\ck = \CC[J,L,\Omega]$ (\cref{lem:idcent}).  The relevance is that a simple weight $\twzhu{\ubpvoak}$-module $\Mod{M}$ is a $\twzhu{\bpminmoduv}$-module if and only if $\ak$ annihilates some nonzero element of $\Mod{M}$.  This fact is proved in exactly the same way that \cite[Lem.~4.1]{KawRel19} is (see also \cite{AdaSom94}) and so we omit the details.

	We next note that the action of $\ak$ preserves each of the one-dimensional weight spaces of the irreducible semisimple coherent family $\zsscfam{\Delta,\omega}$ and that this action is polynomial: for each $a \in \ak \subset \CC[J,L,\Omega]$, there is a polynomial $p_a$ in three variables such that $a$ acts on the weight space $\zsscfam{\Delta,\omega}(j,\Delta,\omega)$ as multiplication by $p_a(j,\Delta,\omega)$.  Since $\Delta$ and $\omega$ are fixed by the choice of coherent family, we may regard $p_a$ as a single-variable polynomial.

	If we now assume that one of the \infdim\ submodules of $\zsscfam{\Delta,\omega}$ is a $\twzhu{\bpminmoduv}$-module, then it is annihilated by $\twzhu{\bpmpik}$ and thus by $\ak$.  Thus, for every $a \in \ak$, we have $p_a(j,\Delta,\omega) = 0$ for infinitely many distinct values of $j$, whence $p_a(\blank,\Delta,\omega)$ must be the zero polynomial.  But, then $a$ annihilates all of $\zsscfam{\Delta,\omega}$, whence $\zsscfam{\Delta,\omega}$ is a $\twzhu{\bpminmoduv}$-module.
\end{proof}

Note that the top space of every (simple) $\twrhw{[j],\Delta,\omega}$ embeds into some irreducible semisimple coherent family and that every such family has an \infdim\ \emph{\hw} submodule $\zihw{j',\Delta}$, by \cref{prop:cfamuniv}.  From \cref{thm:hwclass}, we have classified all the simple \hw\ $\bpminmoduv$-modules in terms of surviving weights.  \cref{thm:rhwclass} thus determines the irreducible semisimple coherent families that are $\twzhu{\bpminmoduv}$-modules and in this way we find all the $\twrhw{[j],\Delta,\omega}$ that are simple $\bpminmoduv$-modules.  Algorithmically, this classification proceeds as follows.

Let $\infwtsk$ denote the set of ($w=\wun$) admissible $\aslthree$-weights $\lambda$ of level $\kk$ with $\lambda^F_0 \ne 0$, so that $\lambda \in \survk$ (\cref{lem:survivors}), and $\lambda^F_1 \ne 0$, so that $\twihw{\lambda}$ has an \infdim\ top space (\cref{prop:fdimtop}).  Then, $\infwtsk$ parametrises the isomorphism classes of the simple \hw\ $\bpminmoduv$-modules with \infdim\ top spaces.
\begin{itemize}
	\item For each $\lambda \in \infwtsk$, compute $j$ and $\Delta$ using \eqref{eq:bptwhwfromsl3}, then substitute into \eqref{eq:defomegapm} to compute $\omega$:
	\begin{equation} \label{eq:bprhwfromsl3}
		\omega = \omega^+_{j,\Delta}
		= -\frac{2}{27} \brac*{\lambda_1 - \lambda_2 + \kk + 3} \brac*{2 \lambda_1 + \lambda_2 - \kk} \brac*{\lambda_1 + 2 \lambda_2 - 2 \kk - 3}.
	\end{equation}
	This gives the eigenvalues of $J$, $L$ and $\Omega$ on the \hwv\ of $\top{(\twihw{\lambda})}$.
	\item Then, the $\twrhw{[j'],\Delta,\omega}$ are, for all $[j'] \in \CC/\ZZ$ satisfying $\omega^+_{i,\Delta} \ne \omega$ for every $i \in [j']$, simple \rhw\ $\bpminmoduv$-modules (by \cref{thm:classtwrhw,thm:rhwclass}) and all such modules are obtained, up to isomorphism, in this way.
\end{itemize}
As with the \hw\ $\bpminmoduv$-modules classified in \cref{sec:hwmod}, it is convenient to let $\twrhw{[j],\lambda} = \twrhw{[j],\Delta,\omega}$, where $\Delta$ and $\omega$ are given in terms of $\lambda$ by \eqref{eq:bptwhwfromsl3} and \eqref{eq:bprhwfromsl3}, respectively.

We may now summarise this classification as follows.
\begin{theorem} \label{thm:simplerhwclass}
	Let $\kk$ be as in \cref{assume} and let $j$ be such that $\twrhw{[j],\lambda}$ is simple.  Then, $\twrhw{[j],\lambda}$ is a (twisted) $\bpminmoduv$-module if and only if $\lambda \in \infwtsk$.
\end{theorem}
\noindent In fact, we shall see that a complete classification does not require considering every possible weight $\lambda \in \infwtsk$.  First however, we recall from \cref{cor:howmanyhwms} that there are no \hw\ $\bpminmoduv$-modules with \infdim\ top spaces, hence $\infwtsk = \varnothing$, when $\vv=2$.
\begin{corollary} \label{cor:norelaxed}
	Let $\kk$ be as in \cref{assume} with $\vv=2$.  Then, every simple (twisted) \rhw\ $\bpminmoduv$-module is \hw.
\end{corollary}
\noindent Again, this is consistent with the fact \cite{AraRat13} that $\bpminmod{\uu}{2}$ is rational for every $\uu \in 2\NN+3$.  It is therefore convenient to slightly refine \cref{assume} as follows.
\begin{assumption} \label{assume2}
	In what follows, we shall restrict to fractional levels $\kk = -3 + \frac{\uu}{\vv}$ with $\uu,\vv \ge 3$.
\end{assumption}
\noindent The levels of \cref{assume2} are also known as nondegenerate admissible levels in the literature.  We shall understand that \cref{assume2} is in force for the rest of this \lcnamecref{sec:sbpmods}.

Given an irreducible semisimple coherent family $\zsscfam{\Delta,\omega}$ of $\twzhu{\bpminmoduv}$-modules, we ask how many inequivalent \infdim\ \hw\ submodules it possesses.  By \cref{prop:zkdense}, the direct summands $\zrhw{[j],\Delta,\omega}$ are not simple for at least one, and at most three, $[j] \in \CC/\ZZ$ and each nonsimple summand has precisely one \infdim\ \hw\ submodule.  The answer to our question is therefore either one, two or three.  In fact, for $\kk$ as in \cref{assume2}, the answer is always three.
\begin{lemma} \label{lem:Z3action}
	If $\kk$ is as in \cref{assume2}, then each irreducible semisimple coherent family $\zsscfam{\Delta,\omega}$ of $\twzhu{\bpminmoduv}$-modules has precisely three \infdim\ \hw\ submodules.  The map $\infwtsk \to \CC^2$ given by $\lambda \mapsto (\Delta,\omega)$ is thus $3$-to-$1$.  Moreover, the highest weights $\lambda = \lambda^I - \frac{\uu}{\vv} \lambda^F$ of these three submodules are related by the following $\ZZ_3$-action:
	\begin{equation} \label{eq:Z3action}
		\begin{gathered}
			\cdots \longmapsto [\lambda_0,\lambda_1,\lambda_2] \longmapsto [\lambda_2 - \tfrac{\uu}{\vv},\lambda_0,\lambda_1 + \tfrac{\uu}{\vv}] \longmapsto [\lambda_1,\lambda_2 - \tfrac{\uu}{\vv},\lambda_0 + \tfrac{\uu}{\vv}] \longmapsto \cdots, \\
			\cdots \longmapsto [\lambda^I_0,\lambda^I_1,\lambda^I_2] \longmapsto [\lambda^I_2,\lambda^I_0,\lambda^I_1] \longmapsto [\lambda^I_1,\lambda^I_2,\lambda^I_0] \longmapsto \cdots, \\
			\cdots \longmapsto [\lambda^F_0,\lambda^F_1,\lambda^F_2] \longmapsto [\lambda^F_2+1,\lambda^F_0,\lambda^F_1-1] \longmapsto [\lambda^F_1,\lambda^F_2+1,\lambda^F_0-1] \longmapsto \cdots.
		\end{gathered}
	\end{equation}
\end{lemma}
\begin{proof}
	It is easy to see from \eqref{eq:Z3action} that if $\lambda \in \infwtsk$, then so do its images under the $\ZZ_3$-action.  The three \hwms\ corresponding to the $\ZZ_3$-orbit are thus $\bpminmoduv$-modules with \infdim\ top spaces if any is.  Moreover, substituting $\lambda_1 \mapsto \lambda_0 = \kk - \lambda_1 - \lambda_2$ and $\lambda_2 \mapsto \lambda_1 + \kk+3$ into \eqref{eq:bptwhwfromsl3} and \eqref{eq:bprhwfromsl3} shows that $\Delta$ and $\omega$ are invariant under this $\ZZ_3$-action.  The three \hwms\ therefore arise as submodules of the same irreducible semisimple coherent family.  These modules are mutually inequivalent because their highest weights can only coincide if $\lambda^I_0 = \lambda^I_1 = \lambda^I_2 = \frac{\uu-3}{3}$ and $\lambda^F_0 = \lambda^F_1 = \lambda^F_2 + 1 = \frac{\vv}{3}$.  But, this requires both $\uu$ and $\vv$ to be divisible by $3$.
\end{proof}

From \cref{cor:howmanyhwms}, we now have a precise count of the number of irreducible semisimple coherent families of $\twzhu{\bpminmoduv}$-modules.  Each direct summand of such a family is the top space of a simple twisted \rhw\ $\bpminmoduv$-module, by \cref{thm:twzhubij}.  With \cref{eq:auts,lem:idcent,thm:rhwclass}, we have the following \lcnamecref{quotientsofcoherents}.
\begin{theorem} \label{quotientsofcoherents}
	Let $\kk$ be as in \cref{assume2}.  Then:
	\begin{itemize}
		\item There are $\frac{1}{3} \abs*{\infwtsk} = \frac{1}{12} (\uu-1)(\uu-2)(\vv-1)(\vv-2)$ irreducible semisimple coherent families of $\twzhu{\bpminmoduv}$-modules $\zsscfam{\Delta,\omega}$, up to isomorphism.
		\item The families of twisted \rhw\ $\bpminmoduv$-modules $\twrhw{[j],\lambda} = \twrhw{[j],\Delta,\omega}$ are in $1$-to-$1$ correspondence with $\infwtsk / \ZZ_3$, where $\ZZ_3$ acts freely as in \eqref{eq:Z3action}.
		\item For each $\lambda \in \infwtsk$, the twisted \rhwm\ $\twrhw{[j],\lambda}$ is a simple $\bpminmoduv$-module for all cosets $[j] \in \CC/\ZZ$ except three, namely the three distinct cosets that contain a root $i$ of the polynomial $\omega^+_{i,\Delta} - \omega$.
		\item The conjugate of the simple twisted \rhw\ $\bpminmoduv$-module $\twrhw{[j],\Delta,\omega}$ is $\conjmod{\twrhw{[j],\Delta,\omega}} \cong \twrhw{[-j],\Delta,-\omega}$.
	\end{itemize}
\end{theorem}

Note that if $(\Delta,\omega)$ corresponds to a coherent family of $\twzhu{\bpminmoduv}$-modules, then the conjugation functor requires that so must $(\Delta,-\omega)$.  In fact, it is easy to check that $\Delta$ is invariant and $\omega$ is antiinvariant under the $\ZZ_2$-action $[\lambda_0,\lambda_1,\lambda_2] \leftrightarrow [\lambda_2 - \frac{\uu}{\vv},\lambda_1,\lambda_0 + \frac{\uu}{\vv}]$, that is
\begin{equation} \label{eq:Z2action}
	[\lambda^I_0,\lambda^I_1,\lambda^I_2] \longleftrightarrow [\lambda^I_2,\lambda^I_1,\lambda^I_0], \quad
	[\lambda^F_0,\lambda^F_1,\lambda^F_2] \longleftrightarrow [\lambda^F_2+1,\lambda^I_1,\lambda^I_0-1],
\end{equation}
which obviously preserves belonging to $\infwtsk$.  With \eqref{eq:Z3action}, this defines an action of $\wgrp = \grp{S}_3$ on $\infwtsk$.  The orbits clearly have length $6$ unless $\omega=0$, in which case \cref{lem:Z3action} forces them to have length $3$.  It is easy to check that this is consistent with the explicit factorisation of $\omega$ given in \eqref{eq:bprhwfromsl3}.

We remark that the spectral flow images $\sfmod{\ell}{\twrhw{[j],\lambda}}$, $\ell \ne 0$, of these simple twisted \rhw\ $\bpminmoduv$-modules are likewise simple $\bpminmoduv$-modules, but they are not \rhw\ because their conformal weights are not bounded below.

\subsection{Nonsimple \rhw{} $\bpminmoduv$-modules} \label{sec:indec}

In \cref{sec:cohfam}, we introduced three classes of irreducible coherent families of $\zk$-modules.  The first, the semisimple class, was the key ingredient in the classification arguments of the previous \lcnamecref{sec:rhwmod}.  Here, we will analyse the other two classes in order to demonstrate the existence of certain nonsemisimple twisted \rhw\ $\bpminmod{u}{v}$-modules, assuming that $\kk$ is as in \cref{assume2}.  We will also describe the structure of these nonsemisimple modules in terms of short exact sequences.

Consider therefore the irreducible nonsemisimple coherent family $\zpmcfam{\Delta,\omega}$ of $\zk$-modules on which $G^{\pm}$ acts injectively.  Recall that its simple direct summands are the $\zrhw{[j],\Delta,\omega}$, for all but (up to) three $[j] \in \CC/\ZZ$, and that its nonsimple direct summands are denoted by $\zpmrhw{[j],\Delta,\omega}$.  We begin by determining the structure of these nonsimple $\zk$-modules in the case relevant to studying $\bpminmoduv$-modules.

\begin{proposition} \label{rootsincoset}
	Let $\lambda \in \infwtsk$ and let $j$, $\Delta$ and $\omega$ be defined by \eqref{eq:bptwhwfromsl3} and \eqref{eq:bprhwfromsl3}.  Then, the nonsimple $\zk$-module $\zpmrhw{[j],\Delta,\omega}$ has exactly two composition factors, $\zihw{j,\Delta}$ and $\zconjmod{\zihw{-j-1,\Delta}}$, both of which are $\twzhu{\bpminmoduv}$-modules.  Moreover, we have the following nonsplit short exact sequences:
	\begin{equation} \label{es:zkmod}
		\dses{\zconjmod{\zihw{-j-1,\Delta}}}{}{\zprhw{[j],\Delta,\omega}}{}{\zihw{j,\Delta}}, \qquad
		\dses{\zihw{j,\Delta}}{}{\zmrhw{[j],\Delta,\omega}}{}{\zconjmod{\zihw{-j-1,\Delta}}}.
	\end{equation}
\end{proposition}
\begin{proof}
	We only consider $\zprhw{[j],\Delta,\omega}$ as the argument for $\zmrhw{[j],\Delta,\omega}$ is identical.  First, note that $\zihw{j,\Delta}$ is an \infdim\ $\twzhu{\bpminmoduv}$-module, by \cref{thm:hwclass}.  The irreducible semisimple coherent family $\zsscfam{\Delta,\omega}$ is therefore a $\twzhu{\bpminmoduv}$-module too, by \cref{thm:rhwclass}, hence so is the \lwm\ $\zconjmod{\zihw{-j-1,\Delta}} \subset \zssrhw{[j],\Delta,\omega}$.  As $\zssrhw{[j],\Delta,\omega}$ is the semisimplification of $\zprhw{[j],\Delta,\omega}$, they have the same composition factors.  To demonstrate that there are no more factors beyond the two already found, it suffices to show that $\zihw{-j-1,\Delta}$ is \infdim.

	Since the conjugate of $\zihw{-j-1,\Delta}$ is a $\twzhu{\bpminmoduv}$-module, $\zihw{-j-1,\Delta}$ must correspond to some $\mu \in \survk$, by \cref{thm:hwclass}.  Proceeding as in the proof of \cref{lem:survivors}, we find that the unique solution is $\mu = [\lambda_0, \lambda_2 - \frac{\uu}{\vv}, \lambda_1 + \frac{\uu}{\vv}]$, hence $\mu^I = [\lambda^I_0, \lambda^I_2, \lambda^I_1]$ and $\mu^F = [\lambda^F_0, \lambda^F_2+1, \lambda^F_1-1]$.  Because $\mu^F_1 = \lambda^F_2+1 \ne 0$, it follows that $\mu \in \infwtsk$ and so $\zihw{-j-1,\Delta}$ is \infdim, as desired.  This establishes the first exact sequence in \eqref{es:zkmod}.  It is clearly nonsplit because $G^+$ acts injectively on $\zprhw{[j],\Delta,\omega}$.
\end{proof}

At this point, it is not clear if the $\zpmrhw{[j],\Delta,\omega}$ corresponding to $\lambda \in \infwtsk$ are $\twzhu{\bpminmoduv}$-modules, even though their composition factors are.  We settle this using a simplified version of the argument of \cite[Thm.~5.3]{KawRel19}.
\begin{proposition} \label{prop:isamod}
	Let $\lambda \in \infwtsk$ and let $j$, $\Delta$ and $\omega$ be defined by \eqref{eq:bptwhwfromsl3} and \eqref{eq:bprhwfromsl3}.  Then, the nonsimple $\zk$-module $\zpmrhw{[j],\Delta,\omega}$ is a $\twzhu{\bpminmoduv}$-module.
\end{proposition}
\begin{proof}
	Again, we shall only detail the argument for $\zprhw{[j],\Delta,\omega}$.  Recall that $\bpmpik$ denotes the maximal ideal of $\ubpvoak$ and so $\twzhu{\bpmpik} \cdot \zihw{j,\Delta} = 0$, by virtue of $\zihw{j,\Delta}$ being a $\twzhu{\bpminmoduv}$-module.  From the first exact sequence in \eqref{es:zkmod}, we conclude that $\twzhu{\bpmpik} \cdot \zprhw{[j],\Delta,\omega} \subseteq \zconjmod{\zihw{-j-1,\Delta}}$.

	As $\zk$ is noetherian (this is an easy generalisation of \cite[Cor.~1.3]{SmiCla90}), so is its quotient $\twzhu{\ubpvoak}$ (\cref{prop:twzhu}).  The ideal $\twzhu{\bpmpik} \subset \twzhu{\ubpvoak}$ is therefore generated by a finite number of elements $a_1, \dots, a_n$ which we may, without loss of generality, choose to be eigenvectors of $J$.  Let $j_i$ denote the $J$-eigenvalue of $a_i$, $i=1,\dots,n$.

	Choose $j' \in [j]$ such that $j' \le j - \max \set{j_1,\dots,j_n}$.  Then, $a_i$ takes the $J$-eigenspace of $\zprhw{[j],\Delta,\omega}$ of eigenvalue $j'$ into the $J$-eigenspace of $\zconjmod{\zihw{-j-1,\Delta}}$ of eigenvalue $j'+a_i \le j$.  But, the eigenvalues of $J$ acting on $\zconjmod{\zihw{-j-1,\Delta}}$ are bounded below by $j+1$, hence $a_i$ annihilates the $J$-eigenspace of $\zprhw{[j],\Delta,\omega}$ of eigenvalue $j'$, for each $i$.  It follows that $\twzhu{\bpmpik}$ annihilates this eigenspace.  But, this eigenspace generates $\zprhw{[j],\Delta,\omega}$, hence $\twzhu{\bpmpik}$ (being an ideal) annihilates $\zprhw{[j],\Delta,\omega}$.
\end{proof}

By Zhu-induction (\cref{thm:twzhubij}), one may construct from each $\twzhu{\bpminmoduv}$-module $\zpmrhw{[j],\Delta,\omega}$ a twisted $\bpminmoduv$-module whose twisted Zhu image (its top space) is $\zpmrhw{[j],\Delta,\omega}$.  Consider the submodule of this induced module obtained by summing all the submodules whose intersection with the top space $\zpmrhw{[j],\Delta,\omega}$ is zero.  Quotienting by this submodule results in a twisted $\bpminmoduv$-module, which we shall denote by $\twpmrhw{[j],\lambda} = \twpmrhw{[j],\Delta,\omega}$, that has $\zpmrhw{[j],\Delta,\omega}$ as its top space and has the property that its nonzero submodules intersect this top space nontrivially.  In a sense, $\twpmrhw{[j],\Delta,\omega}$ is the smallest $\bpminmoduv$-module whose top space is $\zpmrhw{[j],\Delta,\omega}$.

The $\twpmrhw{[j],\lambda}$ are clearly nonsemisimple, because their top spaces are.  This proves the following result.
\begin{theorem} \label{thm:bpminmodsarelogcfts}
	When $\kk$ is as in \cref{assume2}, the simple \voa\ $\bpminmoduv$ admits nonsemisimple modules.  In physical language, the corresponding minimal model \cft\ is \emph{logarithmic}.
\end{theorem}
\noindent As we have mentioned before, the \bp\ minimal models corresponding to $\bpminmod{\uu}{2}$, with $\uu\in2\NN+3$, were shown to be rational in \cite{AraRat13}.

Our final task is then to determine the structure of these nonsemisimple $\bpminmoduv$-modules.  For this, it is convenient to introduce new modules $\twpmrv{[j],\lambda} = \twpmrv{[j],\Delta,\omega}$ that are obtained by treating $\zpmrhw{[j],\Delta,\omega}$ as a module over the twisted mode algebra $\twmodealg_0$ of \eqref{eq:pbw}, letting $\twmodealg_>$ act as $0$, and then inducing to a $\twmodealg$-module.  It follows that $\twpmrv{[j],\lambda}$ is a ``relaxed Verma'' $\ubpvoak$-module whose top space is $\zpmrhw{[j],\Delta,\omega}$.  In a sense, it is the largest $\ubpvoak$-module with this top space.

As such, we may consider the sum $\twpmnax{[j],\lambda}$ of all the submodules of $\twpmrv{[j],\lambda}$ whose intersection with the top space $\zpmrhw{[j],\Delta,\omega}$ is zero.  Because this top space is nonsemisimple, $\twpmnax{[j],\lambda}$ is a proper submodule of the maximal submodule $\twpmmax{[j],\lambda}$ of $\twpmrv{[j],\lambda}$.  Its utility lies in the fact that it provides an alternative construction of the $\bpminmoduv$-module $\twpmrhw{[j],\lambda}$:
\begin{equation} \label{eq:constructnonssrhwms}
	 \twpmrhw{[j],\lambda} \cong \twpmrv{[j],\lambda} \Big/ \twpmnax{[j],\lambda}.
\end{equation}
This exploits the fact that $\twpmrhw{[j],\lambda}$ is, in a sense, the smallest $\ubpvoak$-module with top space $\zpmrhw{[j],\Delta,\omega}$.

We now proceed in an analogous fashion to \cite[Sec.~4]{KawRel18}.
\begin{theorem} \label{thm:nonssstructure}
	Let $\kk$ be as in \cref{assume2} and let $\lambda \in \infwtsk$ define $j$, $\Delta$ and $\omega$ via \eqref{eq:bptwhwfromsl3} and \eqref{eq:bprhwfromsl3}.  We then have the following nonsplit short exact sequences of $\bpminmoduv$-modules:
	\begin{equation} \label{es:bpmod}
		\dses{\conjmod{\twihw{-j-1,\Delta}}}{}{\twprhw{[j],\Delta,\omega}}{}{\twihw{j,\Delta}}, \qquad
		\dses{\twihw{j,\Delta}}{}{\twmrhw{[j],\Delta,\omega}}{}{\conjmod{\twihw{-j-1,\Delta}}}.
	\end{equation}
\end{theorem}
\begin{proof}
	Once again, we only give the argument for $\twprhw{[j],\Delta,\omega}$.  First, note that the twisted Verma module $\twvhw{j,\Delta}$ is clearly isomorphic to the quotient $\twprv{[j],\Delta,\omega} \big/ \conjmod{\twvhw{-j-1,\Delta}}$, by \eqref{es:zkmod} and the exactness of induction.  Hence, $\twihw{j,\Delta}$ is also a quotient and \eqref{eq:constructnonssrhwms} gives
	\begin{equation} \label{eq:surjects}
		\frac{\twprhw{[j],\Delta,\omega}}{\twpmax{[j],\Delta,\omega} \big/ \twpnax{[j],\Delta,\omega}}
		\cong \frac{\twprv{[j],\Delta,\omega}}{\twpmax{[j],\Delta,\omega}} \cong \twihw{j,\Delta},
	\end{equation}
	since \rhwms\ have unique irreducible quotients.  Thus, $\twihw{j,\Delta}$ is a quotient of $\twprhw{[j],\Delta,\omega}$.

	Next, note that the (unique) maximal submodule of $\conjmod{\twvhw{-j-1,\Delta}}$ is $\conjmod{\twvhw{-j-1,\Delta}} \cap \twpnax{[j],\Delta,\omega}$, because the only submodule of $\conjmod{\twvhw{-j-1,\Delta}}$ intersecting its top space nontrivially is $\conjmod{\twvhw{-j-1,\Delta}}$ itself.  We therefore have
	\begin{equation} \label{eq:embeds}
		\conjmod{\twihw{-j-1,\Delta}}
		= \frac{\conjmod{\twvhw{-j-1,\Delta}}}{\conjmod{\twvhw{-j-1,\Delta}} \cap \twpnax{[j],\Delta,\omega}}
		\cong \frac{\conjmod{\twvhw{-j-1,\Delta}} + \twpnax{[j],\Delta,\omega}}{\twpnax{[j],\Delta,\omega}},
	\end{equation}
	which is clearly a submodule of $\twprv{[j],\Delta,\omega} \big/ \twpnax{[j],\Delta,\omega} \cong \twprhw{[j],\Delta,\omega}$.  Thus, $\conjmod{\twihw{-j-1,\Delta}}$ embeds into $\twprhw{[j],\Delta,\omega}$.

	To demonstrate exactness of the first sequence of \eqref{es:bpmod}, we note that
	\begin{equation} \label{eq:R/H=W/V+N}
		\frac{\twprhw{[j],\Delta,\omega}}{\conjmod{\twihw{-j-1,\Delta}}}
		\cong \frac{\twprv{[j],\Delta,\omega}}{\conjmod{\twvhw{-j-1,\Delta}} + \twpnax{[j],\Delta,\omega}}
		\cong \frac{\twvhw{j,\Delta}}{\brac[\big]{\conjmod{\twvhw{-j-1,\Delta}} + \twpnax{[j],\Delta,\omega}} \big/ \conjmod{\twihw{-j-1,\Delta}}}
	\end{equation}
	using \eqref{eq:constructnonssrhwms} and \eqref{eq:embeds}.  This shows that $\twprhw{[j],\Delta,\omega} \big/ \conjmod{\twihw{-j-1,\Delta}}$ is a twisted \hw\ $\bpminmoduv$-module.  By \cref{thm:hwrat}, it is simple and therefore isomorphic to $\twihw{j,\Delta}$, by \eqref{eq:surjects}.  This completes the proof.
\end{proof}

\section{Examples} \label{sec:ex}

We conclude by illustrating the above classification results with some specific examples of \bp\ minimal models.  The examples with $\vv=2$ extend the results of \cite{AraRat13} whilst the $(\uu,\vv) = (3,4)$ and $(4,3)$ examples extend those of \cite{AdaCla19}.

\subsection*{Example: $\bpminmod{3}{2}$}

For $\kk=-\frac{3}{2}$, the central charge of the minimal model is $\cc=0$.  Since $\lambda^I \in \pwlat{0} = \set{[0,0,0]}$ and $\lambda^F \in \pwlat{1}$ is constrained by $\lambda^F_0 \ge 0$ so that $\lambda^F = [1,0,0]$, we only have $\lambda = [0,0,0] - \frac{3}{2} [1,0,0] = [\kk,0,0]$.  There is therefore a unique simple untwisted \hwm\ $\ihw{-3 \fwt{0}/2} = \ihw{0,0}$ and a unique simple twisted \hwm\ $\twihw{-3 \fwt{0}/2} = \twihw{0,0}$ (up to isomorphism).  This is clearly the trivial minimal model.

\subsection*{Example: $\bpminmod{5}{2}$}

For $\kk=-\frac{1}{2}$, the central charge is instead $\cc=\frac{2}{5}$ and we have $\lambda^I \in \pwlat{2}$ and $\lambda^F = [1,0,0]$.  There are thus $\abs*{\pwlat{2}} = 6$ simple untwisted \hwms\ and so $6$ simple twisted \hwms, all with \fdim\ top spaces.  We illustrate these modules in \cref{fig:bp52}, arranging them according to $\lambda^I$ and listing the charges and conformal weights of their \hwvs.  We also indicate the effect of the conjugation and spectral flow automorphisms in this arrangement.
\begin{figure}
	\begin{tikzpicture}[>=stealth',scale=1.4]
		\draw[->] (0:0) -- (30:0.5) node[below right] {$\lambda^I_1$};
		\draw[->] (0:0) -- (90:0.5) node[left] {$\lambda^I_2$};
		\draw[->] (0:0) -- (-120:0.433) node[left] {$\lambda^I_0$};
		\begin{scope}[shift={(-3,0)}]
			\node at (0:0) {$\brac[\big]{0,0}$};
			\node at (30:1) {$\brac[\big]{\frac{1}{3},\frac{1}{30}}$};
			\node at (30:2) {$\brac[\big]{\frac{2}{3},\frac{1}{3}}$};
			\node at (90:1) {$\brac[\big]{-\frac{1}{3},\frac{1}{30}}$};
			\node at (90:2) {$\brac[\big]{-\frac{2}{3},\frac{1}{3}}$};
			\node at (60:1.732) {$\brac[\big]{0,\frac{1}{5}}$};
			\draw[dashed] (60:-0.55) -- (60:-0.15);
			\draw[dashed] (60:0.25) -- (60:1.6);
			\draw[dashed] (60:2) -- (60:2.4);
			\draw[<->] (52:2.5) -- node[auto,swap] {$\conjsymb$} (68:2.5);
		\end{scope}
		\begin{scope}[shift={(2,0)}]
			\node at (0:0) {$\brac[\big]{\frac{1}{3},\frac{1}{12}}_1$};
			\node at (30:1) {$\brac[\big]{\frac{2}{3},\frac{17}{60}}_2$};
			\node at (30:2) {$\brac[\big]{1,\frac{3}{4}}_3$};
			\node at (90:1) {$\brac[\big]{0,-\frac{1}{20}}_1$};
			\node at (90:2) {$\brac[\big]{-\frac{1}{3},\frac{1}{12}}_1$};
			\node at (60:1.732) {$\brac[\big]{\frac{1}{3},\frac{17}{60}}_2$};
			\draw[dashed] (-0.9,1) -- (-0.5,1);
			\draw[dashed] (0.5,1) -- (1.35,1);
			\draw[dashed] (2.1,1) -- (2.5,1);
			\draw[<->] (2.6,1.3) -- node[right] {$\conjsymb$} (2.6,0.7);
		\end{scope}
	\end{tikzpicture}
	\caption{The charges and conformal weights $(j,\Delta)$ of the untwisted (left) and twisted (right) simple \hw\ $\bpminmod{5}{2}$-modules, arranged by the Dynkin labels of the integral parts $\lambda^I$ of the corresponding surviving weights $\lambda$.  The subscript on the twisted labels gives the dimension of the top space.  Conjugation $\conjsymb$ is indicated by reflection about the dashed line and spectral flow $\sfsymb$ by $120^{\circ}$ anticlockwise rotation about each triangle's centre.} \label{fig:bp52}
\end{figure}

\subsection*{Example: $\bpminmod{9}{2}$}

We discuss one further minimal model with $\vv=2$, that with $\kk=\frac{3}{2}$ and $\cc=-\frac{22}{3}$.  This time, there are $\abs*{\pwlat{6}} = 28$ simple untwisted \hwms\ and, of course, each has a single twisted cousin.  As always when $\vv=2$, the top spaces are all \fdim\ and the fractional part $\lambda^F$ of the corresponding $\aslthree$-weights is $[1,0,0]$.

An interesting feature of this minimal model is that the (integer) spectral flows of the vacuum module $\ihw{0,0}$ correspond to $\lambda^I = [0,6,0]$ and $[0,0,6]$, hence $(j,\Delta) = (2,1)$ and $(-2,1)$.  Recalling that spectral flows of the vacuum module are always simple currents \cite{LiPhy97}, it follows that $\bpminmod{9}{2}$ admits an order-$3$ simple current extension $\VOA{A}$.  Moreover, if $E$ and $F$ denote the \hwvs\ of the simple current modules $\ihw{2,1}$ and $\ihw{-2,1}$, respectively, then it is easy to check that $E$, $F$ and $J$ define a (nonconformal) embedding of the $\sltwo$ minimal model $\sltwominmod{3}{1} = \saffvoa{1}{\sltwo}$ into $\VOA{A}$.

Defining $\overline{G}^{\,+} = G^-_{-1/2} E$ and $\overline{G}^{\,-} = G^+_{-1/2} F$, we see that $\VOA{A}$ has four linearly independent fields of conformal weight $\frac{3}{2}$ and that they decompose into two $\sltwo$-doublets $(G^-,\overline{G}^{\,+})$ and $(\overline{G}^{\,-},G^+)$.  $\VOA{A}$ may thus be regarded as some sort of bosonic analogue of the $N=4$ superconformal \svoa, see \cref{fig:bp92}.  However, a major difference is that the elements $E$, $J$, $F$, $G^{\pm}$ $\overline{G}^{\,\pm}$ and $L$ do not strongly generate $\VOA{A}$.  For example, the singular part of the \ope\ of $G^+(z)$ and $\overline{G}^{\,-}(w)$ is a simple pole whose coefficient is the $(j,\Delta) = (0,2)$ field corresponding to $T^- = (G^+_{-1/2})^2 F$.
\begin{figure}
	\begin{tikzpicture}[>=stealth',xscale=2,yscale=1.5]
		\draw[->] (-3,2) -- (-2.5,2) node[above right] {$j$};
		\draw[->] (-3,2) -- (-3,1.33) node[below left] {$\Delta$};
		\draw[white] (3,2) -- (2.5,2);
		\node at (0,2) {$\wun$};
		\node at (0,1) {$J$};
		\node at (1,0.5) {$G^+\ \overline{G}^{\,+}$};
		\node at (-1,0.5) {$\overline{G}^{\,-}\ G^-$};
		\node at (0,0) {$T^-\ \begin{matrix} L \ \pd J \\ \no{JJ} \end{matrix}\ T^+$};
		\node at (2,1) {$E$};
		\node at (-2,1) {$F$};
		\node at (2,0) {$\begin{matrix} \pd E \\ \no{JE} \end{matrix}$};
		\node at (-2,0) {$\begin{matrix} \pd F \\ \no{JF} \end{matrix}$};
		\node at (1,-0.5) {$\vdots$};
		\node at (-1,-0.5) {$\vdots$};
	\end{tikzpicture}
\caption{The states with conformal weight $\Delta \le 2$ of the ``$N=4$-like'' \voa\ $\VOA{A} = \ihw{-2,1} \oplus \ihw{0,0} \oplus \ihw{2,1}$ that extends $\bpminmod{9}{2}$.  Here, $T^+ = G^-_{-1/2} \overline{G}^{\,+}$ and $T^- = G^+_{-1/2} \overline{G}^{\,-}$.} \label{fig:bp92}
\end{figure}

It is nevertheless easy to explore the representation theory of $\VOA{A}$.  The set of $28$ (isomorphism classes of) simple untwisted \hw\ $\bpminmod{9}{2}$-modules decomposes into $10$ spectral flow orbits: $9$ of length $3$ and one fixed point.  It is easy to check from the charges and conformal weights that only four of these orbits define untwisted $\VOA{A}$-modules.  There are therefore precisely $4$ simple untwisted $\VOA{A}$-modules:
\begin{equation}
	\VOA{A} = \ihw{-2,1} \oplus \ihw{0,0} \oplus \ihw{2,1}, \quad \ihw{-1,1/6} \oplus \ihw{0,-1/3} \oplus \ihw{1,1/6}, \quad
	\ihw{-1,-1/6} \oplus \ihw{0,1/3} \oplus \ihw{1,-1/6} \quad \text{and} \quad \ihw{0,-2/9}.
\end{equation}
One can also classify the simple twisted $\VOA{A}$-modules, but now there are several more twisted sectors to consider.

\subsection*{Example: $\bpminmod{3}{4}$}

Consider next the \bp\ minimal model with $\kk=-\frac{9}{4}$ and $\cc=-\frac{23}{2}$.  This model arises as the $p=4$ member of a series $\mathcal{B}_p$ of interesting \voas\ constructed in \cite{CreCos13}.  As $\lambda^I \in \pwlat{0}$ and $\lambda^F \in \pwlat{3}$ satisfies $\lambda^F_0 \ge 1$, there are $\abs*{\pwlat{2}} = 6$ simple untwisted \hwms\ and $6$ simple twisted \hwms, $3$ of which have \fdim\ top spaces.  We illustrate these in \cref{fig:bp34} as we did for $\bpminmod{5}{2}$, but arranging the data according to $\lambda^F$ instead of $\lambda^I$.  One can check that this recovers the \hw\ classification of \cite{AdaCla19}.

In this illustration, the spectral flow functor $\sfsymb$ is again represented by a $120^{\circ}$ anticlockwise rotation, but does not preserve being \hw\ (because $\vv \ne 2$).  Indeed, the three spectral flow orbits through the simple \hw\ $\bpminmod{3}{4}$-modules are
\begin{equation}
	\begin{gathered}
		\cdots \lma \ihw{0,-1/2} \lma \twihw{-1/4,-9/16} \lma \cdots, \\
		\cdots \lma \ihw{1/4,-3/8} \lma \twihw{0,-5/16} \lma \ihw{-1/4,-3/8} \lma \twihw{-1/2,-9/16} \lma \cdots, \\
		\cdots \lma \ihw{1/2,-1/4} \lma \twihw{1/4,-1/16} \lma \ihw{0,0} \lma \twihw{-1/4,-1/16} \lma \ihw{-1/2,-1/4} \lma \twihw{-3/4,-9/16} \lma \cdots,
	\end{gathered}
\end{equation}
where the $\cdots$ indicate simple $\bpminmod{3}{4}$-modules that are not \hw.
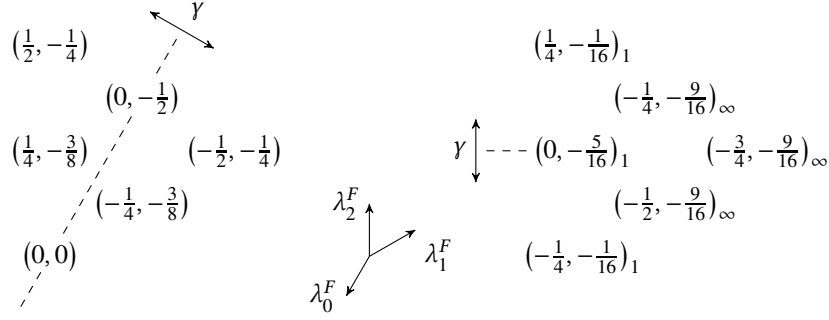
\begin{figure}
	\begin{tikzpicture}[>=stealth',scale=1.4]
		\draw[->] (0:0) -- (30:0.5) node[below right] {$\lambda^F_1$};
		\draw[->] (0:0) -- (90:0.5) node[left] {$\lambda^F_2$};
		\draw[->] (0:0) -- (-120:0.433) node[left] {$\lambda^F_0$};
		\begin{scope}[shift={(-3,0)}]
			\node at (0:0) {$\brac[\big]{0,0}$};
			\node at (30:1) {$\brac[\big]{-\frac{1}{4},-\frac{3}{8}}$};
			\node at (30:2) {$\brac[\big]{-\frac{1}{2},-\frac{1}{4}}$};
			\node at (90:1) {$\brac[\big]{\frac{1}{4},-\frac{3}{8}}$};
			\node at (90:2) {$\brac[\big]{\frac{1}{2},-\frac{1}{4}}$};
			\node at (60:1.732) {$\brac[\big]{0,-\frac{1}{2}}$};
			\draw[dashed] (60:-0.55) -- (60:-0.15);
			\draw[dashed] (60:0.25) -- (60:1.6);
			\draw[dashed] (60:2) -- (60:2.4);
			\draw[<->] (52:2.5) -- node[auto,swap] {$\conjsymb$} (68:2.5);
		\end{scope}
		\begin{scope}[shift={(2,0)}]
			\node at (0:0) {$\brac[\big]{-\frac{1}{4},-\frac{1}{16}}_1$};
			\node at (30:1) {$\brac[\big]{-\frac{1}{2},-\frac{9}{16}}_{\infty}$};
			\node at (30:2) {$\brac[\big]{-\frac{3}{4},-\frac{9}{16}}_{\infty}$};
			\node at (90:1) {$\brac[\big]{0,-\frac{5}{16}}_1$};
			\node at (90:2) {$\brac[\big]{\frac{1}{4},-\frac{1}{16}}_1$};
			\node at (60:1.732) {$\brac[\big]{-\frac{1}{4},-\frac{9}{16}}_{\infty}$};
			\draw[dashed] (-0.9,1) -- (-0.5,1);
			\draw[<->] (-1,1.3) -- node[left] {$\conjsymb$} (-1,0.7);
		\end{scope}
	\end{tikzpicture}
	\caption{The charges and conformal weights $(j,\Delta)$ of the untwisted (left) and twisted (right) simple \hw\ $\bpminmod{3}{4}$-modules, arranged by the Dynkin labels of the fractional parts $\lambda^F$ of the corresponding surviving weights $\lambda$.  The subscript on the twisted labels gives the dimension of the top space.  Conjugation $\conjsymb$ is indicated by reflection about the dashed line, restricted to the modules with \fdim\ top spaces (the conjugate of a \hwm\ with an \infdim\ top space is not \hw).} \label{fig:bp34}
\end{figure}

The three simple twisted \hwms\ with $\lambda^F_1 > 0$ have \infdim\ top spaces.  They also share the same conformal weight $\Delta = -\frac{9}{16}$ and $\omega$-parameter $\omega = \omega^+_{j,\Delta} = 0$, the latter computed as in \eqref{eq:bprhwfromsl3}.  It therefore follows that $\bpminmod{3}{4}$ admits one family of simple twisted \rhwms\ $\twrhw{[j],-9/16,0}$, $j \ne -\frac{1}{4}, -\frac{1}{2}, -\frac{3}{4} \pmod{1}$, as per \cref{quotientsofcoherents}.  As a consistency check, substituting $\kk=-\frac{9}{4}$ and $\Delta = -\frac{9}{16}$ into \eqref{eq:defomegapm} indeed gives
\begin{equation}
	\omega^+_{j,-9/16} - \omega = (2j+1)(j^2+j+\tfrac{3}{16}) - 0 = 2(j+\tfrac{1}{4})(j+\tfrac{1}{2})(j+\tfrac{3}{4}),
\end{equation}
as expected.

This family was first constructed in \cite[Thm.~7.2]{AdaCla19}, though four exceptional values of $j \pmod{1}$ were given there instead of three.  Here, we have also proven that there are no other families.  We also note that \cref{thm:nonssstructure} proves the existence of six nonsemisimple twisted \rhw\ $\bpminmod{3}{4}$-modules, each characterised by a nonsplit short exact sequence:
\begin{subequations}
	\begin{gather}
		\begin{gathered}
			\dses{\conjmod{\twihw{-3/4,-9/16}}}{}{\twprhw{[-1/4],-9/16,0}}{}{\twihw{-1/4,-9/16}}, \\
			\dses{\conjmod{\twihw{-1/2,-9/16}}}{}{\twprhw{[-1/2],-9/16,0}}{}{\twihw{-1/2,-9/16}}, \\
			\dses{\conjmod{\twihw{-1/4,-9/16}}}{}{\twprhw{[-3/4],-9/16,0}}{}{\twihw{-3/4,-9/16}},
		\end{gathered}
		\\
		\begin{gathered}
			\dses{\twihw{-1/4,-9/16}}{}{\twmrhw{[-1/4],-9/16,0}}{}{\conjmod{\twihw{-3/4,-9/16}}}, \\
			\dses{\twihw{-1/2,-9/16}}{}{\twmrhw{[-1/2],-9/16,0}}{}{\conjmod{\twihw{-1/2,-9/16}}}, \\
			\dses{\twihw{-3/4,-9/16}}{}{\twmrhw{[-3/4],-9/16,0}}{}{\conjmod{\twihw{-1/4,-9/16}}}.
		\end{gathered}
	\end{gather}
\end{subequations}

There are other nonsemisimple $\bpminmod{3}{4}$-modules.  In particular, there exist staggered (logarithmic) modules on which $J_0$ acts semisimply but $L_0$ has Jordan blocks of rank $2$.  This follows from the well known fact \cite{AdaLat09,NagTri11} that staggered modules exist for the triplet \voa\ $\VOA{W}(1,4)$ of central charge $-\frac{25}{2}$.  The connection is that the coset of $\bpminmod{3}{4} = \mathcal{B}_4$ by the Heisenberg subalgebra generated by $J$ is the singlet algebra $\VOA{I}(1,4)$ \cite{CreCos13} and that the latter has $\VOA{W}(1,4)$ as an (infinite-order) simple current extension \cite{RidMod13}.  We shall not study these staggered $\bpminmod{3}{4}$-modules here, but intend to investigate them more generally in a sequel.

\subsection*{Example: $\bpminmod{4}{3}$}

The minimal model with $\kk=-\frac{5}{3}$ and $\cc=-1$ was also studied in \cite{AdaCla19}.  This time, we have $\lambda^I \in \pwlat{1}$ and $\lambda^F \in \pwlat{2}$, hence there are $\abs*{\pwlat{1}} \abs*{\pwlat{1}} = 9$ simple untwisted \hwms.  Moreover, $6$ of the simple twisted \hwms\ have \fdim\ top spaces whilst the top spaces of the other $3$ are \infdim.  We arrange the \hw\ data in an $\slthree$-covariant fashion in \cref{fig:bp43}, making the scale for $\lambda^I$ significantly smaller than that for $\lambda^F$ to improve clarity.  It follows that there is again only one family of generically simple \rhw\ $\bpminmod{4}{3}$-modules.  This family must therefore be closed under conjugation and so $\omega = 0$.  This can of course be checked explicitly using \eqref{eq:bprhwfromsl3}.
\begin{figure}
	\begin{tikzpicture}[>=stealth',scale=1.4]
		\begin{scope}[shift={(-4,0)}]
			\node at (0:0) {$\brac[\big]{0,0}$};
			\node at (30:1) {$\brac[\big]{\frac{1}{3},\frac{1}{2}}$};
			\node at (90:1) {$\brac[\big]{-\frac{1}{3},\frac{1}{2}}$};
			\begin{scope}[shift={(30:3)}]
				\node at (0:0) {$\brac[\big]{-\frac{4}{9},\frac{1}{9}}$};
				\node at (30:1) {$\brac[\big]{-\frac{1}{9},-\frac{1}{18}}$};
				\node at (90:1) {$\brac[\big]{-\frac{7}{9},\frac{5}{18}}$};
			\end{scope}
			\begin{scope}[shift={(90:3)}]
				\node at (0:0) {$\brac[\big]{\frac{4}{9},\frac{1}{9}}$};
				\node at (30:1) {$\brac[\big]{\frac{7}{9},\frac{5}{18}}$};
				\node at (90:1) {$\brac[\big]{\frac{1}{9},-\frac{1}{18}}$};
			\end{scope}
		\end{scope}
		\begin{scope}[shift={(1.5,0)}]
			\node at (0:0) {$\brac[\big]{-\frac{1}{18},-\frac{1}{72}}_1$};
			\node at (30:1) {$\brac[\big]{\frac{5}{18},\frac{47}{72}}_2$};
			\node at (90:1) {$\brac[\big]{-\frac{7}{18},\frac{23}{72}}_1$};
			\begin{scope}[shift={(90:3)}]
				\node at (0:0) {$\brac[\big]{\frac{7}{18},\frac{23}{72}}_1$};
				\node at (30:1) {$\brac[\big]{\frac{13}{18},\frac{47}{72}}_2$};
				\node at (90:1) {$\brac[\big]{\frac{1}{18},-\frac{1}{72}}_1$};
			\end{scope}
			\begin{scope}[shift={(30:3)}]
				\node at (0:0) {$\brac[\big]{-\frac{1}{2},-\frac{1}{8}}_{\infty}$};
				\node at (30:1) {$\brac[\big]{-\frac{1}{6},-\frac{1}{8}}_{\infty}$};
				\node at (90:1) {$\brac[\big]{-\frac{5}{6},-\frac{1}{8}}_{\infty}$};
			\end{scope}
		\end{scope}
	\end{tikzpicture}
	\caption{The charges and conformal weights $(j,\Delta)$ of the untwisted (left) and twisted (right) simple \hw\ $\bpminmod{4}{3}$-modules, arranged by the Dynkin labels of the integral (small-scale) and fractional (large-scale) parts $\lambda^F$ of the corresponding surviving weights $\lambda$.  The subscript on the twisted labels gives the dimension of the top space.} \label{fig:bp43}
\end{figure}
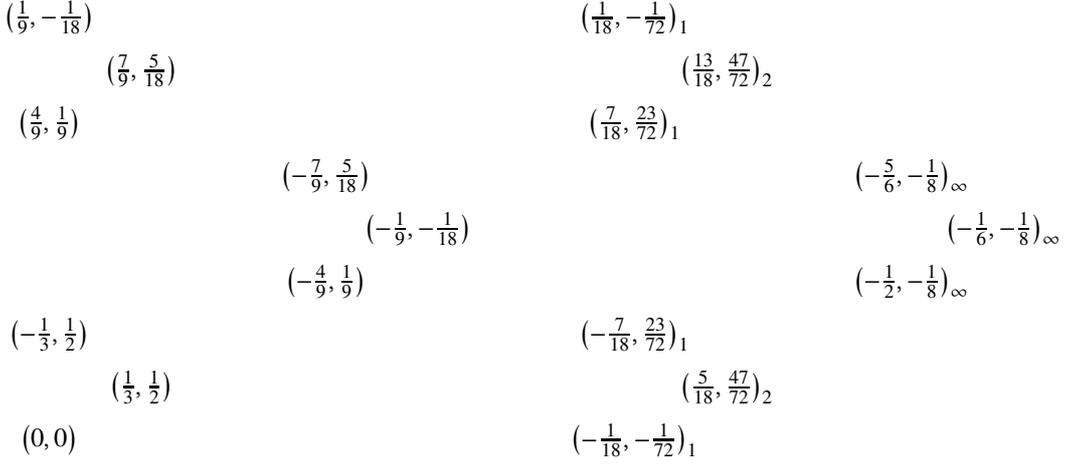

Along with the simple twisted \rhw\ $\bpminmod{4}{3}$-modules $\twrhw{[j],-1/8,0}$, $j \ne -\frac{1}{6}, -\frac{1}{2}, -\frac{5}{6} \pmod{1}$, we also deduce the existence of six nonsemisimple twisted \rhw\ $\bpminmod{4}{3}$-modules, characterised by the following nonsplit short exact sequences:
\begin{subequations}
	\begin{gather}
		\begin{gathered}
			\dses{\conjmod{\twihw{-5/6,-1/8}}}{}{\twprhw{[-1/6],-1/8,0}}{}{\twihw{-1/6,-1/8}}, \\
			\dses{\conjmod{\twihw{-1/2,-1/8}}}{}{\twprhw{[-1/2],-1/8,0}}{}{\twihw{-1/2,-1/8}}, \\
			\dses{\conjmod{\twihw{-1/6,-1/8}}}{}{\twprhw{[-5/6],-1/8,0}}{}{\twihw{-5/6,-1/8}},
		\end{gathered}
		\\
		\begin{gathered}
			\dses{\twihw{-1/6,-1/8}}{}{\twmrhw{[-1/6],-1/8,0}}{}{\conjmod{\twihw{-5/6,-1/8}}}, \\
			\dses{\twihw{-1/2,-1/8}}{}{\twmrhw{[-1/2],-1/8,0}}{}{\conjmod{\twihw{-1/2,-1/8}}}, \\
			\dses{\twihw{-5/6,-1/8}}{}{\twmrhw{[-5/6],-1/8,0}}{}{\conjmod{\twihw{-1/6,-1/8}}}.
		\end{gathered}
	\end{gather}
\end{subequations}
As with the case $(\uu,\vv) = (3,4)$ discussed above, there are other nonsemisimple $\bpminmod{4}{3}$-modules, in particular there are staggered (logarithmic) modules (as was already noted in \cite{AdaCla19}).  We review the argument briefly for completeness.

First, note \cite[Sec.~5.2]{AdaCla19} that the \bp\ minimal model \voa\ $\bpminmod{4}{3}$ embeds in the symplectic bosons \voa\ $\VOA{B}$ (also known as the bosonic ghost system, $\beta\gamma$ ghosts and the Weyl vertex algebra) with $\cc=-1$.  We recall that $\VOA{B}$ is strongly generated by $\beta$ and $\gamma$, both of conformal weight $\frac{1}{2}$, subject to the \opes
\begin{equation}
	\beta(z) \beta(w) \sim 0 \sim \gamma(z) \gamma(w) \quad \text{and} \quad \beta(z) \gamma(w) \sim \frac{-\wun}{z-w}.
\end{equation}
An embedding $\bpminmod{4}{3} \ira \VOA{B}$ is then given by
\begin{equation} \label{eq:bp43->B}
	J \longmapsto \frac{1}{3} \no{\beta \gamma}, \quad G^+ \longmapsto \frac{1}{3\sqrt{3}} \no{\beta \beta \beta}, \quad
	G^- \longmapsto -\frac{1}{3\sqrt{3}} \no{\beta \beta \beta}, \quad L \longmapsto \frac{1}{2} \brac*{\no{\pd \beta \gamma} - \no{\pd \gamma \beta}}.
\end{equation}
This suggests, and it is easy to check \cite[Prop.~5.9]{AdaCla19}, that $\bpminmod{4}{3}$ is (isomorphic to) the $\ZZ_3$-orbifold of $\VOA{B}$ corresponding to the automorphism $\ee^{2 \pi \ii J_0}$.  As $\VOA{B}$ is known \cite{RidFus10,RidBos14,AllBos20} to admit a family of staggered modules, each member related to the others by spectral flow, so does $\bpminmod{4}{3}$.  In fact, $\bpminmod{4}{3}$ admits three such families.

\subsection*{Example: $\bpminmod{5}{3}$}

We conclude with the \bp\ minimal model with $\kk=-\frac{4}{3}$ and $\cc=\frac{3}{5}$.  With $\lambda^I \in \pwlat{2}$ and $\lambda^F \in \pwlat{2}$, there are $\abs*{\pwlat{2}} \abs{\pwlat{1}} = 18$ simple untwisted \hwms\ and the twisted \hwms\ divide into $12$ with \fdim\ top spaces and $6$ with \infdim\ top spaces.  We illustrate the \hw\ data in \cref{fig:bp53}.  There are thus two families of generically simple twisted \rhwms, one with $\Delta = \frac{1}{8}$ and one with $\Delta = -\frac{3}{40}$.  As these conformal weights differ, each family must be closed under conjugation and so we have $\omega = 0$ for both (again).  We therefore have simple twisted \rhw\ $\bpminmod{5}{3}$-modules $\twrhw{[j],1/8,0}$, $j \ne -\frac{7}{6}, -\frac{1}{2}, \frac{1}{6} \pmod{1}$, and $\twrhw{[j],-3/40,0}$, $j \ne -\frac{5}{6}, -\frac{1}{2}, -\frac{1}{6} \pmod{1}$, along with the $12$ nonsemisimple versions guaranteed by \cref{thm:nonssstructure}.
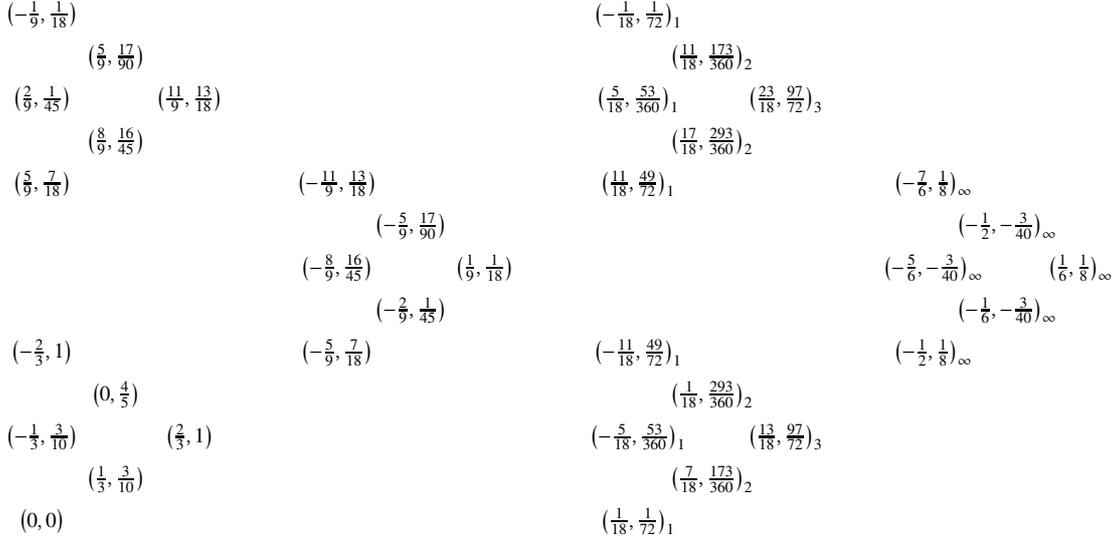
\begin{figure}
	\scalebox{0.8}{
		\begin{tikzpicture}[>=stealth',scale=1.4]
			\begin{scope}[shift={(-4,0)}]
				\node at (0:0) {$\brac[\big]{0,0}$};
				\node at (30:1) {$\brac[\big]{\frac{1}{3},\frac{3}{10}}$};
				\node at (30:2) {$\brac[\big]{\frac{2}{3},1}$};
				\node at (90:1) {$\brac[\big]{-\frac{1}{3},\frac{3}{10}}$};
				\node at (90:2) {$\brac[\big]{-\frac{2}{3},1}$};
				\node at (60:1.732) {$\brac[\big]{0,\frac{4}{5}}$};
				\begin{scope}[shift={(30:4)}]
					\node at (0:0) {$\brac[\big]{-\frac{5}{9},\frac{7}{18}}$};
					\node at (30:1) {$\brac[\big]{-\frac{2}{9},\frac{1}{45}}$};
					\node at (30:2) {$\brac[\big]{\frac{1}{9},\frac{1}{18}}$};
					\node at (90:1) {$\brac[\big]{-\frac{8}{9},\frac{16}{45}}$};
					\node at (90:2) {$\brac[\big]{-\frac{11}{9},\frac{13}{18}}$};
					\node at (60:1.732) {$\brac[\big]{-\frac{5}{9},\frac{17}{90}}$};
				\end{scope}
				\begin{scope}[shift={(90:4)}]
					\node at (0:0) {$\brac[\big]{\frac{5}{9},\frac{7}{18}}$};
					\node at (30:1) {$\brac[\big]{\frac{8}{9},\frac{16}{45}}$};
					\node at (30:2) {$\brac[\big]{\frac{11}{9},\frac{13}{18}}$};
					\node at (90:1) {$\brac[\big]{\frac{2}{9},\frac{1}{45}}$};
					\node at (90:2) {$\brac[\big]{-\frac{1}{9},\frac{1}{18}}$};
					\node at (60:1.732) {$\brac[\big]{\frac{5}{9},\frac{17}{90}}$};
				\end{scope}
			\end{scope}
			\begin{scope}[shift={(3,0)}]
				\node at (0:0) {$\brac[\big]{\frac{1}{18},\frac{1}{72}}_1$};
				\node at (30:1) {$\brac[\big]{\frac{7}{18},\frac{173}{360}}_2$};
				\node at (30:2) {$\brac[\big]{\frac{13}{18},\frac{97}{72}}_3$};
				\node at (90:1) {$\brac[\big]{-\frac{5}{18},\frac{53}{360}}_1$};
				\node at (90:2) {$\brac[\big]{-\frac{11}{18},\frac{49}{72}}_1$};
				\node at (60:1.732) {$\brac[\big]{\frac{1}{18},\frac{293}{360}}_2$};
				\begin{scope}[shift={(30:4)}]
					\node at (0:0) {$\brac[\big]{-\frac{1}{2},\frac{1}{8}}_{\infty}$};
					\node at (30:1) {$\brac[\big]{-\frac{1}{6},-\frac{3}{40}}_{\infty}$};
					\node at (30:2) {$\brac[\big]{\frac{1}{6},\frac{1}{8}}_{\infty}$};
					\node at (90:1) {$\brac[\big]{-\frac{5}{6},-\frac{3}{40}}_{\infty}$};
					\node at (90:2) {$\brac[\big]{-\frac{7}{6},\frac{1}{8}}_{\infty}$};
					\node at (60:1.732) {$\brac[\big]{-\frac{1}{2},-\frac{3}{40}}_{\infty}$};
				\end{scope}
				\begin{scope}[shift={(90:4)}]
					\node at (0:0) {$\brac[\big]{\frac{11}{18},\frac{49}{72}}_1$};
					\node at (30:1) {$\brac[\big]{\frac{17}{18},\frac{293}{360}}_2$};
					\node at (30:2) {$\brac[\big]{\frac{23}{18},\frac{97}{72}}_3$};
					\node at (90:1) {$\brac[\big]{\frac{5}{18},\frac{53}{360}}_1$};
					\node at (90:2) {$\brac[\big]{-\frac{1}{18},\frac{1}{72}}_1$};
					\node at (60:1.732) {$\brac[\big]{\frac{11}{18},\frac{173}{360}}_2$};
				\end{scope}
			\end{scope}
		\end{tikzpicture}
	}
	\caption{The charges and conformal weights $(j,\Delta)$ of the untwisted (left) and twisted (right) simple \hw\ $\bpminmod{5}{3}$-modules, arranged by the Dynkin labels of the integral (small-scale) and fractional (large-scale) parts $\lambda^F$ of the corresponding surviving weights $\lambda$.  The subscript on the twisted labels gives the dimension of the top space.} \label{fig:bp53}
\end{figure}

An interesting feature of this minimal model is the existence of modules $\ihw{\pm 2/3,1}$ corresponding to $\lambda^I = [0,2,0], [0,0,2]$ and $\lambda^F=[2,0,0]$.  These are not spectral flows of the vacuum module, but we nevertheless conjecture that they are simple currents generating an order-$3$ simple current extension $\VOA{C}$ of $\bpminmod{5}{3}$.  As with $\bpminmod{9}{2}$ (and assuming this conjecture), the \hwvs\ $E$ and $F$ of $\ihw{2/3,1}$ and $\ihw{-2/3,1}$, respectively, generate a copy of an $\sltwo$ minimal model, this time $\sltwominmod{5}{2} = \saffvoa{1/2}{\sltwo}$.  But unlike the situation for $\bpminmod{9}{2}$, the embedding $\sltwominmod{5}{2} \ira \VOA{C}$ is conformal.

We recall from \cite[Sec.~10]{CreMod13}, see also \cite{CreBra17}, that $\sltwominmod{5}{2}$ has a simple current whose top space is the four-dimensional simple $\sltwo$-module with conformal weight $\frac{3}{2}$.  We therefore conjecture that this order-$2$ simple current extension of $\sltwominmod{5}{2}$ is isomorphic to $\VOA{C}$, illustrating the low-conformal weight states of $\VOA{C}$ in \cref{fig:bp53'} for convenience (and noting that the $\sltwominmod{5}{2}$ Cartan element $H$ is identified with $3J$).  This extended \voa\ was conjectured to be the minimal \qhr\ of $\saffvoa{-3/2}{\alg{g}_2}$ in \cite[Sec.~10]{CreMod13}.  This was settled affirmatively in \cite[Thm.~6.8]{AdaCon17}.
\begin{figure}
	\begin{tikzpicture}[>=stealth',xscale=2.5,yscale=1.5]
		\draw[->] (-3,2) -- (-2.6,2) node[above right] {$j$};
		\draw[->] (-3,2) -- (-3,1.33) node[below left] {$\Delta$};
		\draw[white] (3,2) -- (2.6,2);
		\node at (0,2) {$\wun$};
		\node at (0,1) {$J$};
		\node at (1.5,0.5) {$G^+$};
		\node at (-1.5,0.5) {$G^-$};
		\node at (0.5,0.5) {$\overline{G}^{\,+}$};
		\node at (-0.5,0.5) {$\overline{G}^{\,-}$};
		\node at (0,0) {$\begin{matrix} L \ \pd J \\ \no{JJ} \end{matrix}$};
		\node at (1,1) {$E$};
		\node at (-1,1) {$F$};
		\node at (2,0) {$\no{EE}$};
		\node at (-2,0) {$\no{FF}$};
		\node at (1,0) {$\begin{matrix} \pd E \\ \no{JE} \end{matrix}$};
		\node at (-1,0) {$\begin{matrix} \pd F \\ \no{JF} \end{matrix}$};
		\node at (1,-0.5) {$\vdots$};
		\node at (-1,-0.5) {$\vdots$};
	\end{tikzpicture}
\caption{The states with conformal weight $\Delta \le 2$ of the extended algebra $\VOA{C} = \ihw{-2/3,1} \oplus \ihw{0,0} \oplus \ihw{2/3,1}$ of $\bpminmod{5}{3}$.  Here, $\overline{G}^{\,+} = G^+_{-1/2} F$ and $\overline{G}^{\,-} = G^-_{-1/2} E$, whilst $\no{EE}$ and $\no{FF}$ are proportional to $(G^+_{-1/2})^2 F$ and $(G^-_{-1/2})^2 E$, respectively.} \label{fig:bp53'}
\end{figure}

The conjectured embeddings $\bpminmod{5}{3} \ira \VOA{C} \hookleftarrow \sltwominmod{5}{2}$ may be tested through representation theory.  Indeed, $\sltwominmod{5}{2}$ has two simple \hwms\ with \fdim\ top spaces, in addition to the vacuum and simple current module.  Their direct sum may be identified with the simple $\VOA{C}$-module $\ihw{-1/3,3/10} \oplus \ihw{0,4/5} \oplus \ihw{1/3,3/10}$.  Likewise, there are four simple \hw\ $\sltwominmod{5}{2}$-modules with \infdim\ top spaces and they combine to give two simple $\VOA{C}$-modules $\ihw{-7/6,1/8} \oplus \ihw{-1/2,1/8} \oplus \ihw{1/6,1/8}$ and $\ihw{-5/6,-3/40} \oplus \ihw{-1/2,-3/40} \oplus \ihw{-1/6,-3/40}$.  The story is predictably similar for the \rhwms.

We finish by noting that $\bpminmod{5}{3}$ also admits staggered (logarithmic) modules because $\sltwominmod{5}{2}$ does \cite{AdaLat09,AdaRea17}, see also \cite{CreCos18}.  In fact, we expect that staggered $\bpminmoduv$-modules exist for all $\vv\ge3$ and hope to return to this in the future.

\appendix
\section{Proof of \cref{thm:surjection}} \label{app:proof}

In this \lcnamecref{app:proof}, we adopt the notation of \cref{sec:sl3} and assume throughout that $\lambda_0 \notin \NN$ so that $\qhrfun{\slihw{\lambda}} \ne 0$ (and that the level $\kk$ is as in \cref{assume}).  With these assumptions, the aim is to prove the following assertion:
\begin{equation} \label{eq:assertion}
	\slmpik \cdot \slihw{\lambda} \ne 0 \quad \Rightarrow \quad \qhrfun{\slmpik} \cdot \qhrfun{\slihw{\lambda}} \ne 0.
\end{equation}
$\qhrfun{\blank}$ is a cohomological functor which involves tensoring with a ghost \svoa\ whose vacuum element will be denoted by $\ghvac$.  With this, we shall prove \eqref{eq:assertion} by exhibiting elements $\chi \in \slmpik$ and $v \in \slihw{\lambda}$ for which $\chi\otimes \ghvac$ and $v\otimes \ghvac$ are (degree-$0$) closed elements of the appropriate BRST complexes and the (clearly closed) element $\chi_n v\otimes \ghvac$ is not exact, for some $n \in \ZZ$.  Using brackets to denote cohomology classes, $[\chi_n v\otimes \ghvac]$ then gives a nonzero element of $\qhrfun{\slmpik} \cdot \qhrfun{\slihw{\lambda}}$:
\begin{equation}
	[\chi \otimes \ghvac] \cdot [v \otimes \ghvac] \equiv [\chi \otimes \ghvac](z) [v \otimes \ghvac] = [\chi(z) v \otimes \ghvac] \ne 0.
\end{equation}
As noted at the end of \cref{sec:sl3}, this amounts to a proof of \cref{thm:surjection}.  To prove \eqref{eq:assertion} however, we need to delve a little deeper into the details of minimal \qhr\ for $\uslvoak$.

\subsection{Minimal \qhr} \label{sec:qhr}

Recall from \cite{KacQua03} that the minimal \qhr\ functor $\qhrfun{\blank}$ computes the cohomology of the tensor product of a given $\uslvoak$-module with certain ghost \svoas.  Specifically, we need a fermionic ghost system $\fgvoa{\alpha}$ for each positive root $\alpha \in \proots$ of $\slthree$ and one bosonic ghost system $\bgvoa$ corresponding to the two simple roots $\sroot{1}$ and $\sroot{2}$.  Denoting the fermionic ghosts by $b^{\alpha}$ and $c^{\alpha}$, $\alpha \in \proots$, and the bosonic ghosts by $\beta$ and $\gamma$, we take the defining \opes\ to be
\begin{equation}
	b^{\alpha}(z) c^{\alpha}(w) \sim \frac{\wun}{z-w} \quad \text{and} \quad \beta(z) \gamma(w) \sim \frac{\wun}{z-w},
\end{equation}
understanding that the remaining \opes\ between ghost generating fields are regular.  The tensor product of these ghost \svoas\ will be denoted by $\ghvoa =  \fgvoa{\sroot{1}} \otimes \fgvoa{\sroot{2}} \otimes \fgvoa{\hroot} \otimes \bgvoa$, for convenience.

We fix a basis of $\slthree$ for the computations to follow.  Let $E_{ij}$ denote the $3 \times 3$ matrix with $1$ in the $(i,j)$-th position and zeroes elsewhere.  Then, we set
\begin{equation}
	\rvec{\hroot} = E_{13}, \quad
	\begin{aligned}
		\rvec{\sroot{1}} &= E_{12}, & \cvec{\sroot{1}} &= E_{11}-E_{22}, & \nrvec{\sroot{1}} &= E_{21}, \\
		\rvec{\sroot{2}} &= E_{23}, & \cvec{\sroot{2}} &= E_{22}-E_{33}, & \nrvec{\sroot{2}} &= E_{32},
	\end{aligned}
	\quad \nrvec{\hroot} = E_{31}.
\end{equation}
Here, $\hroot = \sroot{1} + \sroot{2}$ is the highest root of $\slthree$ and we shall also set $\cvec{\hroot} = \cvec{\sroot{1}} + \cvec{\sroot{2}} = E_{11}-E_{33}$.

To define $\qhrfun{\Mod{M}}$ for a $\uslvoak$-module $\Mod{M}$, one first grades $\Mod{M} \otimes \ghvoa$ by the fermionic ghost number, that is by the total number of $c$-modes minus the total number of $b$-modes.  Equivalently, the ghost number is the eigenvalue of the zero mode of the field $\sum_{\alpha \in \proots} \no{b^{\alpha}(z) c^{\alpha}(z)}$.  Next, one introduces \cite{BerCon91,KacQua03} the following fermionic field of ghost number $1$:
\begin{equation} \label{eq:diff}
	d(z) = \brac[\big]{\rvec{\hroot}(z) + \wun} c^{\hroot}(z) + \brac[\big]{\rvec{\sroot{1}}(z) + \beta(z)} c^{\sroot{1}}(z) + \brac[\big]{\rvec{\sroot{2}}(z) + \gamma(z)} c^{\sroot{2}}(z) + \no{b^{\hroot}(z) c^{\sroot{2}}(z) c^{\sroot{1}}(z)}.
\end{equation}
A straightforward computation verifies that $d(z) d(w) \sim 0$.  We then form a differential complex by requiring that $d(z)$ is homogeneous of conformal weight $1$ and equipping $\Mod{M} \otimes \ghvoa$ with the differential $d = d_0$ (which obviously squares to $0$).

With \eqref{eq:diff}, this requirement on $d(z)$ requires that the conformal weight of $c^{\hroot}$ is also $1$, whilst that of $\rvec{\hroot}$ is $0$.  The latter may be achieved by adding $\frac{1}{2} \pd \cvec{\hroot}$ to the standard Sugawara \emt\ $T^{\text{Sug.}}$ of $\uslvoak$.  When this is done, homogeneity and \eqref{eq:diff} now fix the conformal weight $\widetilde{\Delta}$ of all the generating fields as in \cref{tab:genwts}.  The \emt\ of $\uslvoak \otimes \ghvoa$ is thus
\begin{equation}
	\begin{gathered}
		\widetilde{L} = T^{\text{Sug.}} + \frac{1}{2} \pd \cvec{\hroot} + \sum_{\alpha \in \proots} T^{\fgvoa{\alpha}} + T^{\bgvoa}, \\
		\text{where} \quad
		T^{\fgvoa{\sroot{i}}} = \frac{1}{2} \no{\pd b^{\sroot{i}} c^{\sroot{i}} + \pd c^{\sroot{i}} b^{\sroot{i}}}, \quad
		T^{\fgvoa{\hroot}} = \no{\pd b^{\hroot} c^{\hroot}} \quad \text{and} \quad
		T^{\bgvoa} = \frac{1}{2} \no{\pd \gamma \beta - \pd \beta \gamma}.
	\end{gathered}
\end{equation}
The central charge matches that of $\ubpvoak$, see \eqref{eq:bpc}:
\begin{equation}
	\frac{8\kk}{\kk+3} - 6\kk + 1 + 1 - 2 - 1 = -\frac{(2\kk+3)(3\kk+1)}{\kk+3}.
\end{equation}

\begin{table}
	\begin{tabular}{C||CCCCCCCC|CCCCCC|CC}
		& \rvec{\hroot} & \rvec{\sroot{1}} & \rvec{\sroot{2}} & \cvec{\sroot{1}} & \cvec{\sroot{2}} & \nrvec{\sroot{1}} & \nrvec{\sroot{2}} & \nrvec{\hroot} & b^{\sroot{1}} & c^{\sroot{1}} & b^{\sroot{2}} & c^{\sroot{2}} & b^{\hroot} & c^{\hroot} & \beta & \gamma \\
		\hline
		\# & 0 & 0 & 0 & 0 & 0 & 0 & 0 & 0 & -1 & 1 & -1 & 1 & -1 & 1 & 0 & 0 \\
		\widetilde{j} & 0 & 1 & -1 & 0 & 0 & -1 & 1 & 0 & 1 & -1 & -1 & 1 & 0 & 0 & 1 & -1 \\
		\widetilde{\Delta} & 0 & \frac{1}{2} & \frac{1}{2} & 1 & 1 & \frac{3}{2} & \frac{3}{2} & 2 & \frac{1}{2} & \frac{1}{2} & \frac{1}{2} & \frac{1}{2} & 0 & 1 & \frac{1}{2} & \frac{1}{2}
	\end{tabular}
	\bigskip
	\caption{The ghost numbers $\#$, charges $\widetilde{j}$ and conformal weights $\widetilde{\Delta}$ of the generating fields of the \svoa\ $\uslvoak \otimes \ghvoa$.} \label{tab:genwts}
\end{table}

As the notation suggests, $\widetilde{L}$ is closed and its image in cohomology (that is, in $\qhrfun{\uslvoak \otimes \ghvoa, d} = \qhrfun{\uslvoak} \cong \ubpvoak$) is $L$.  Note that the ``symmetric'' deformation of adding $\frac{1}{2} \pd \cvec{\hroot}$ to $T^{\text{Sug.}}$ ensures this result.  There are other deformations consistent with $d$ being a differential --- they correspond to adding a multiple of $\pd J$ to $L$.  Speaking of which, the element
\begin{equation}
	\widetilde{J} = \frac{1}{3} (\cvec{\sroot{1}} - \cvec{\sroot{2}}) + \no{b^{\sroot{1}} c^{\sroot{1}}} - \no{b^{\sroot{2}} c^{\sroot{2}}} - \no{\beta \gamma}
\end{equation}
is likewise closed and its image in cohomology is $J$ \cite{KacQua03}.  We give the charge ($\widetilde{J}_0$-eigenvalue) of the generating fields of $\uslvoak \otimes \ghvoa$ in \cref{tab:genwts} for completeness.  We also note that
\begin{equation} \label{eq:defGpm}
	\begin{aligned}
		\widetilde{G}^+ &= \nrvec{\sroot{2}} + \no{\cvec{\sroot{2}} \beta} - \no{b^{\sroot{1}} c^{\hroot}} - \no{b^{\sroot{1}} c^{\sroot{1}} \beta} + 2 \no{b^{\sroot{2}} c^{\sroot{2}} \beta} + \no{b^{\hroot} c^{\hroot} \beta} + \no{\beta \beta \gamma} + (\kk+1) \pd \beta \\
		\text{and} \quad
		\widetilde{G}^- &= \nrvec{\sroot{1}} - \no{\cvec{\sroot{1}} \gamma} + \no{b^{\sroot{2}} c^{\hroot}} - 2 \no{b^{\sroot{1}} c^{\sroot{1}} \gamma} + \no{b^{\sroot{2}} c^{\sroot{2}} \gamma} - \no{b^{\hroot} c^{\hroot} \gamma} + \no{\gamma \gamma \beta} - (\kk+1) \pd \gamma
	\end{aligned}
\end{equation}
are both closed.  Their images in cohomology are $G^+$ and $G^-$, respectively \cite{KacQua04}.

We remark that deforming the \emt\ of $\uslvoak$ means that we now have two distinct mode conventions for affine fields.  Our convention will be that mode indices with respect to the deformed conformal weight will be denoted with parentheses.  Thus, for an affine generator $a$ with deformed conformal weight $\widetilde{\Delta}$ as in \cref{tab:genwts}, we shall write
\begin{equation} \label{eq:modeconvention}
	a(z) = \sum_{n \in \ZZ} a_n z^{-n-1} = \sum_{n \in \ZZ - \widetilde{\Delta}} a_{(n)} z^{-n-\widetilde{\Delta}}.
\end{equation}
We shall not bother to so distinguish mode indices for ghost fields: their expansions will always be taken with respect to the conformal weights in \cref{tab:genwts}.

\subsection{The proof} \label{sec:proof}

We start with a well known fundamental result for the \hwv\ $v$ of $\slihw{\lambda}$, recalling that we are assuming throughout that $\lambda_0 \notin \NN$ and that $\kk$ satisfies \cref{assume}.  Let $\ghvac$ denote the vacuum vector of $\ghvoa$.  By \cite[Lem.~4.6.1 and Prop.~4.7.1]{AraRep05}, we have the following \lcnamecref{lem:evnotexact}.
\begin{lemma} \label{lem:evnotexact}
	For all $n \in \NN$, $(\rvec{\hroot}_{-1})^n v \otimes \ghvac$ is closed and inexact.  In particular, $[v \otimes \ghvac] \ne 0$.
\end{lemma}
%
%Then, it is easy to check from \eqref{eq:diff} that $d(v \otimes \ghvac) = 0$ and, more generally, that $d \brac*{(\rvec{\hroot}_{-1})^n v \otimes \ghvac} = 0$ for all $n \in \NN$.  On the other hand, we have $\comm{d}{b^{\hroot}_0} = \rvec{\hroot}_{(0)} + \wun = \rvec{\hroot}_{-1} + \wun$, hence
%\begin{equation}
%	d \brac[\big]{b^{\hroot}_0 (\rvec{\hroot}_{-1})^n v \otimes \ghvac} = (\rvec{\hroot}_{-1})^{n+1} v \otimes \ghvac + (\rvec{\hroot}_{-1})^n v \otimes \ghvac
%	\quad \Rightarrow \quad \sqbrac[\big]{(\rvec{\hroot}_{-1})^n v \otimes \ghvac} = (-1)^n \sqbrac[\big]{v \otimes \ghvac}
%\end{equation}
%in cohomology.  The image of the closed subspace $\vspn \set*{(\rvec{\hroot}_{-1})^n v \otimes \ghvac \st n \in \NN}$ in cohomology is therefore at most one-dimensional.  If $\lambda_0$ were a non-negative integer, then $(\rvec{\hroot}_{-1})^{\lambda_0+1} v = 0$ would force this image to be $0$.  However, we are assuming that $\lambda_0 \notin \NN$ and, in this case, \cite[Lemma~4.6.1, Prop.~4.7.1]{AraRep05} proves the contrary instead.

We next consider the maximal ideal $\slmpik$ of $\uslvoak$.
\begin{lemma} \label{lem:chi}
	$\slmpik$ is generated by a single \sv\ $\chi$ whose $\slthree$-weight and conformal weight are $(\uu-2) \hroot$ and $(\uu-2) \vv$, respectively.  Moreover, $\chi \otimes \ghvac$ is closed.
\end{lemma}
\begin{proof}
	This follows easily from \cite[Cor.~1]{KacMod88}, which says that the maximal submodule of a Verma module whose highest weight is admissible is generated by \svs\ of known weight.  In our case, the highest weight is $\kk \fwt{0}$ (which is admissible because $\kk$ is) and the only generating \sv\ that is nonzero in the quotient $\uslvoak$ of this Verma module has weight $\wref{} \cdot (\kk \fwt{0})$, where $\wref{}$ is the Weyl reflection corresponding to the root $-\hroot + \vv \imroot$.  Here, $\imroot$ denotes the standard imaginary root of $\aslthree$.  This \sv\ is $\chi$ and its $\slthree$- and conformal weights are now easily computed.  The fact that $\chi \otimes \ghvac$ is closed follows from $\chi$ being a \hwv.
\end{proof}
\noindent In fact, $\chi \otimes \ghvac$ is also inexact, though we will not need to \emph{a priori} establish this fact for what follows.

We remark that a nice conceptual proof of \cite[Cor.~1]{KacMod88} starts from the celebrated fact that the submodule structure of a Verma module only depends on the corresponding integral Weyl group \cite{FieCom06}.  This structure is therefore the same for all admissible \hw\ $\aslthree$-modules, irrespective of their level.  In particular, this structure matches that of a Verma module whose simple quotient is integrable, integrability being equivalent to admissibility for simple \hwms\ with $\vv=1$.  However, the fact that the maximal submodule is generated by \svs\ is well known in the integrable case, see \cite{KacInf90} or \cite{CarLie05}.

Suppose now that $\chi(z) v = 0$.  Because $\chi$ generates $\slmpik$, it follows that $\slmpik \cdot v = 0$.  Since $v$ generates $\slihw{\lambda}$, as a $\uslvoak$-module, and $\slmpik$ is a two-sided ideal of $\uslvoak$, we get $\slmpik \cdot \slihw{\lambda} = 0$.  Thus, the hypothesis of \eqref{eq:assertion}, that $\slihw{\lambda}$ is not an $\sslvoa{k}$-module, requires that $\chi_n v \ne 0$ for some $n \in \ZZ$.  As $\chi$ has $\slthree$-weight $(\uu-2) \hroot$, our knowledge of the weights of $\slihw{\lambda}$ lets us refine this requirement to $\chi_{-(\uu-2)-i} v \ne 0$ for some $i \in \NN$.  There is therefore a minimal $N \in \NN$ such that $\chi_{-(\uu-2)-N} v \ne 0$.

As $\slihw{\lambda}$ is simple, there therefore exists a \pbw\ monomial $U \in \envalg{\aslthree}$ such that
\begin{equation}
	U \chi_{-(\uu-2)-N} v = v
\end{equation}
(rescaling $\chi$ if necessary).  We choose an ordering for $U$ so that
\begin{equation}
	\nrvec{\alpha}_{n\le0} < \cvec{\alpha}_{n<0} < \rvec{\alpha}_{n<0} < \nrvec{\alpha}_{n>0} < \cvec{\alpha}_{n>0} < \rvec{\alpha}_{n\ge0}
\end{equation}
(obviously we may omit the $\cvec{\alpha}_0$ and $K$).  This means, for example, that the $\nrvec{\alpha}_n$ with $n\le0$ are ordered to the left while the $\rvec{\alpha}_n$ with $n\ge0$ are ordered to the right.  For $n\ge0$, we have $\rvec{\alpha}_0 \chi = 0$ and $\rvec{\alpha}_n v = 0$, hence
\begin{equation}
	\rvec{\alpha}_n \chi_{-(\uu-2)-N} v = (\rvec{\alpha}_0 \chi)_{-(\uu-2)-N+n} v = 0.
\end{equation}
We may therefore assume that $U$ contains no $\rvec{\alpha}_n$-modes with $n\ge0$.  Similarly,
\begin{equation}
	\cvec{\alpha}_n \chi_{-(\uu-2)-N} v = (\uu-2) \hroot(\cvec{\alpha}) \chi_{-(\uu-2)-(N-n)} v = 0
\end{equation}
for $n>0$, by the minimality of $N$.  Thus, we may assume that $U$ contains no $\cvec{\alpha}_n$-modes with $n>0$ either.  Finally, $v$ is not in the image of any $\nrvec{\alpha}_n$, with $n\le0$, $\cvec{\alpha}_n$, with $n<0$, or $\rvec{\alpha}_n$, with $n<0$.  All these modes may therefore also be excluded from $U$.  We conclude that $U$ may be taken to be a monomial in the modes $\nrvec{\alpha}_n$ with $n>0$.

Given a partition $\xi = [\xi_1 \ge \xi_2 \ge \cdots]$, let $\ell(\xi)$ denote its length and $\abs{\xi}$ denote its weight.  We write $\nrvec{\alpha}_{\xi} = \nrvec{\alpha}_{\xi_1} \nrvec{\alpha}_{\xi_2} \cdots$.  Then, there exist partitions $\xi$, $\pi$ and $\rho$ such that $U = \nrvec{\hroot}_{\xi} \nrvec{\sroot{2}}_{\pi} \nrvec{\sroot{1}}_{\rho}$ and so
\begin{equation} \label{eq:halfwaythere}
	\nrvec{\hroot}_{\xi} \nrvec{\sroot{2}}_{\pi} \nrvec{\sroot{1}}_{\rho} \chi_{-(\uu-2)-N} v = v.
\end{equation}
Moreover, considering $\slthree$- and conformal weights gives
\begin{equation} \label{eq:axipirho}
	\ell(\pi) = \ell(\rho), \quad \ell(\xi) + \ell(\pi) = \uu-2 \quad \text{and} \quad \abs{\xi} + \abs{\pi} + \abs{\rho} = \uu-2+N.
\end{equation}

\begin{lemma} \label{lem:usefulcomm}
	Let $F(z)$, $F \in \slthree$, be an affine field and let $U_0$ be a monomial in the negative root vectors $\nrvec{\alpha}_0$ of $\aslthree$.  Then, the modes of the field $(U_0 \chi)(w)$ satisfy
	\begin{equation} \label{eq:usefulcomm}
		\comm{F_m}{(U_0 \chi)_n} = (F_0 U_0 \chi)_{m+n}, \quad \text{for all}\ m,n \in \ZZ.
	\end{equation}
\end{lemma}
\begin{proof}
	Observe that $U_0 \chi$ is annihilated by the $F_m$ with $m>0$.  Consequently, the assertion follows easily from the \ope
	\begin{equation}
		F(z) (U_0 \chi)(w) \sim \frac{(F_0 U_0 \chi)(w)}{z-w}. \qedhere
	\end{equation}
\end{proof}

We apply \cref{lem:usefulcomm} to the \lhs\ of \eqref{eq:halfwaythere}, noting that the $\nrvec{}$-modes all annihilate $v$.  The result is
\begin{equation} \label{eq:complicated}
	\nrvec{\hroot}_{\xi} \nrvec{\sroot{2}}_{\pi} \nrvec{\sroot{1}}_{\rho} \chi_{-(\uu-2)-N} v
	= \brac*{(\nrvec{\hroot}_0)^{\ell(\xi)} (\nrvec{\sroot{2}}_0)^{\ell(\pi)} (\nrvec{\sroot{1}}_0)^{\ell(\rho)} \chi}_0 v,
\end{equation}
using \eqref{eq:axipirho}.  This looks complicated, but it allows us to determine the partitions $\xi$, $\pi$ and $\rho$.
\begin{lemma} \label{lem:mustbe0}
	If any of the parts of $\xi$, $\pi$ or $\rho$ are greater than $1$, then $\nrvec{\hroot}_{\xi} \nrvec{\sroot{2}}_{\pi} \nrvec{\sroot{1}}_{\rho} \chi_{-(\uu-2)-N} v = 0$.
\end{lemma}
\begin{proof}
	Suppose that $\xi$ has a part $\xi_i > 1$ (the argument is identical if $\pi$ or $\rho$ has a part greater than $1$).  Then, we can form a new partition $\xi'$ from $\xi$ by subtracting $1$ from $\xi_i$ and reordering parts if necessary.  Note that $\ell(\xi') = \ell(\xi)$ and $\abs{\xi'} = \abs{\xi} - 1$.  Then, \cref{lem:usefulcomm} and $N$ being minimal give
	\begin{align}
		0
		&= \nrvec{\hroot}_{\xi'} \nrvec{\sroot{2}}_{\pi} \nrvec{\sroot{1}}_{\rho} \chi_{-(\uu-2)-(N-1)} v
		= \brac*{(\nrvec{\hroot}_0)^{\ell(\xi')} (\nrvec{\sroot{2}}_0)^{\ell(\pi)} (\nrvec{\sroot{1}}_0)^{\ell(\rho)} \chi}_{-(\uu-2) + \abs{\xi'} + \abs{\pi} + \abs{\rho} - N+1} v \\
		&= \brac*{(\nrvec{\hroot}_0)^{\ell(\xi)} (\nrvec{\sroot{2}}_0)^{\ell(\pi)} (\nrvec{\sroot{1}}_0)^{\ell(\rho)} \chi}_0 v. \notag
	\end{align}
	But, this is the \rhs\ of \eqref{eq:complicated}.
\end{proof}
\noindent Combining \eqref{eq:halfwaythere}, which is manifestly nonzero, with \cref{lem:mustbe0} now forces all parts of $\xi$, $\pi$ and $\rho$ to be $1$.  As partition lengths and weights are now equal, the relations of \eqref{eq:axipirho} are easily solved to give $\abs{\xi} = \uu-2-N$ and $\abs{\pi} = \abs{\rho} = N$.  In particular, \eqref{eq:halfwaythere} now becomes
\begin{equation}
	(\nrvec{\hroot}_1)^{\uu-2-N} (\nrvec{\sroot{2}}_1)^N (\nrvec{\sroot{1}}_1)^N \chi_{-(\uu-2)-N} v = v.
\end{equation}
By rescaling $\chi$ again, if necessary, we arrive at following key result.
\begin{proposition} \label{prop:thekey}
	If $N$ is the minimal integer such that $\chi_{-(\uu-2)-N} v \ne 0$, then
	\begin{equation} \label{eq:halfwaythereagain}
		(\nrvec{\sroot{2}}_1)^N (\nrvec{\sroot{1}}_1)^N \chi_{-(\uu-2)-N} v = (\rvec{\hroot}_{-1})^{\uu-2-N} v.
	\end{equation}
\end{proposition}

The idea now is to use the fact that the \rhs\ of \eqref{eq:halfwaythereagain} is inexact when tensored with $\ghvac$ (\cref{lem:evnotexact}) to prove that the same is true for $\chi_{-(\uu-2)-N} v$.  For this, we need to replace the action of $\nrvec{\sroot{2}}_1$ and $\nrvec{\sroot{1}}_1$ with elements that preserve exactness, for example any closed elements.
\begin{lemma} \label{lem:G=f}
	For all $i,j \in \NN$, we have
	\begin{equation} \label{eq:G=f}
		\brac[\big]{\widetilde{G}^+_{(1/2)}}^i \brac[\big]{\widetilde{G}^-_{(1/2)}}^j \brac*{\chi_{-(\uu-2)-N} v \otimes \ghvac}
		= (\nrvec{\sroot{2}}_1)^i (\nrvec{\sroot{1}}_1)^j \chi_{-(\uu-2)-N} v \otimes \ghvac.
	\end{equation}
\end{lemma}
\begin{proof}
	We start with \eqref{eq:defGpm}, which gives
	\begin{equation}
		\widetilde{G}^-_{(1/2)}
		= \nrvec{\sroot{1}}_{(1/2)} - \sum_{m \in \ZZ} \cvec{\sroot{1}}_{(m)} \gamma_{-m+1/2} + \cdots
		= \nrvec{\sroot{1}}_1 - \sum_{m \in \ZZ} \cvec{\sroot{1}}_m \gamma_{-m+1/2} + \cdots,
	\end{equation}
	where the $\cdots$ stands for pure ghost terms.  As these ghost terms annihilate $\ghvac$, we have
	\begin{equation}
		\widetilde{G}^-_{(1/2)} \brac[\Big]{(\nrvec{\sroot{1}}_1)^j \chi_{-(\uu-2)-N} v \otimes \ghvac}
		= (\nrvec{\sroot{1}}_1)^{j+1} \chi_{-(\uu-2)-N} v \otimes \ghvac
		- \sum_{m=1}^{\infty} \cvec{\sroot{1}}_m (\nrvec{\sroot{1}}_1)^j \chi_{-(\uu-2)-N} v \otimes \gamma_{-m+1/2} \ghvac,
	\end{equation}
	for any $j \in \NN$.  Now, $m\ge1$ implies that $\cvec{\sroot{2}}_m v = 0$, hence that
	\begin{equation}
		\cvec{\sroot{1}}_m (\nrvec{\sroot{1}}_1)^j \chi_{-(\uu-2)-N} v
		= \comm{\cvec{\sroot{1}}_m}{(\nrvec{\sroot{1}}_1)^j} \chi_{-(\uu-2)-N} v
		+ (\nrvec{\sroot{1}}_1)^j \comm{\cvec{\sroot{1}}_m}{\chi_{-(\uu-2)-N}} v.
	\end{equation}
	The first commutator on the \rhs\ is a sum of terms, each obtained from $(\nrvec{\sroot{1}}_1)^j$ by replacing one of the $\nrvec{\sroot{1}}_1$ by $-2 \nrvec{\sroot{1}}_{m+1}$.  However, each of these terms is $0$ by \cref{lem:mustbe0}.  On the other hand, the second commutator is proportional to $\chi_{-(\uu-2)-(N-m)}$, so it annihilates $v$ by minimality of $N$.  We therefore obtain
	\begin{equation} \label{eq:iszero}
		\widetilde{G}^-_{(1/2)} \brac[\Big]{(\nrvec{\sroot{1}}_1)^j \chi_{-(\uu-2)-N} v \otimes \ghvac}
		= \nrvec{\sroot{1}}_1 (\nrvec{\sroot{1}}_1)^j \chi_{-(\uu-2)-N} v \otimes \ghvac,
	\end{equation}
	from which we conclude inductively that $\brac[\big]{\widetilde{G}^-_{(1/2)}}^j \brac*{\chi_{-(\uu-2)-N} v \otimes \ghvac} = (\nrvec{\sroot{1}}_1)^j \chi_{-(\uu-2)-N} v \otimes \ghvac$, for all $i\in\NN$.

	To deduce \eqref{eq:G=f}, we now repeat the argument by acting with $\widetilde{G}^+_{(1/2)}$ on $(\nrvec{\sroot{2}}_1)^i (\nrvec{\sroot{1}}_1)^j \chi_{-(\uu-2)-N} v \otimes \ghvac$.  There are no essential differences between this case and that described above, so we omit the details.
\end{proof}
\begin{corollary} \label{cor:done}
	$\chi_{-(\uu-2)-N} v \otimes \ghvac$ is closed and inexact.
\end{corollary}
\begin{proof}
	We have already seen that $\chi_{-(\uu-2)-N} v \otimes \ghvac$ is closed.  Suppose therefore that it is exact.  As $\comm{d}{\widetilde{G}^{\pm}_{(1/2)}} = 0$, since $\widetilde{G}^{\pm}$ is closed, it now follows from \cref{prop:thekey,lem:G=f} that
	\begin{equation}
		\brac[\big]{\widetilde{G}^+_{(1/2)}}^N \brac[\big]{\widetilde{G}^-_{(1/2)}}^N \brac[\big]{\chi_{-(\uu-2)-N} v \otimes \ghvac}
		= (\nrvec{\sroot{2}}_1)^N (\nrvec{\sroot{1}}_1)^N \chi_{-(\uu-2)-N} v \otimes \ghvac
		= (\rvec{\hroot}_{-1})^{\uu-2-N} v
	\end{equation}
	is also exact.  But, this contradicts \cref{lem:evnotexact}.
\end{proof}
\noindent This \lcnamecref{cor:done} completes the proof of \cref{thm:surjection}.

We conclude with a few remarks about this proof.  First, proving that \qhr\ indeed realises all simple \hwms\ is obviously desirable and has been studied in several settings.  However, Arakawa's general results \cite{AraRep05,AraRat15} in this direction for universal minimal and regular W-algebras do not immediately imply the desired results for their simple quotients.  Indeed, the cases where this completeness result for simple W-algebras is known seem to be cases in which the simple quotient is rational and $C_2$-cofinite, see for example \cite{AraRat13,AraRat15,AraRat19}.  Our proof, applying as it does to the nonrational and non-$C_2$-cofinite simple \bp\ algebras, is therefore quite novel and seems to be very different from the rational proofs in the literature.

Second, this proof relies on certain key facts that might be regarded as special to the \bp\ algebras.  In particular, we use the explicit realisation \eqref{eq:defGpm} of the charged generators of $\ubpvoak$.  However, the pure ghost terms played no role in the proof, so it may be possible to generalise this part of the argument to other minimal, or perhaps even subregular, W-algebras.  On the other hand, the proof also exploits the fact that the maximal ideal of $\uslvoak$ is generated by a single \sv, which does not normally hold when generalising to nonadmissible levels.  It is therefore not clear that this proof can be adapted for the nonadmissible case, but it would of course be interesting to try.

Alternatively, it may be that one can prove more general completeness results of this type by further developing the inverse \qhr\ methods introduced in \cite{SemInv94,AdaRea17} and extended to the \bp\ algebras in \cite{AdaRea20}.  These methods have the advantage of building up the representation theory iteratively from that of the so-called exceptional W-algebras \cite{AraRat19}, in particular from the regular ones.  This may then lead to uniform methods for all W-algebras, at least when the level is admissible and (sufficiently) nondegenerate.  We hope to have the opportunity to report on this promising direction in the future.

\flushleft
%\bibliography{bp}

\begin{thebibliography}{10}

\bibitem{AraRat13}
T~Arakawa.
\newblock Rationality of {Bershadsky}-{Polyakov} vertex algebras.
\newblock {\em Comm. Math. Phys.}, 323:627--633, 2013.
\newblock \pp{1005.0185}{math.QA}.

\bibitem{AdaCla19}
D~Adamovi\'{c} and A~Kontrec.
\newblock Classification of irreducible modules for {Bershadsky}--{Polyakov}
  algebra at certain levels.
\newblock \pp{1910.13781}{math.QA}.

\bibitem{AdaBer20}
D~Adamovi\'{c} and A~Kontrec.
\newblock {Bershadsky}--{Polyakov} vertex algebras at positive integer levels and duality.
\newblock \pp{2011.10021}{math.QA}.

\bibitem{PolGau90}
A~Polyakov.
\newblock Gauge transformations and diffeomorphisms.
\newblock {\em Int. J. Mod. Phys.}, A5:833--842, 1990.

\bibitem{BerCon91}
M~Bershadsky.
\newblock Conformal field theories via {Hamiltonian} reduction.
\newblock {\em Comm. Math. Phys.}, 139:71--82, 1991.

\bibitem{KacQua03}
V~Kac, S~Roan, and M~Wakimoto.
\newblock Quantum reduction for affine superalgebras.
\newblock {\em Comm. Math. Phys.}, 241:307--342, 2003.
\newblock \opp{0302015}{math-ph}.

\bibitem{AraAss15}
T~Arakawa.
\newblock Associated varieties of modules over {Kac}--{Moody} algebras and
  {$C_2$}-cofiniteness of {W}-algebras.
\newblock {\em Int. Math. Res. Not.}, 2015:11605--11666, 2015.
\newblock \pp{1004.1554}{math.QA}.

\bibitem{KacQua04}
V~Kac and M~Wakimoto.
\newblock Quantum reduction and representation theory of superconformal
  algebras.
\newblock {\em Adv. Math.}, 185:400--458, 2004.
\newblock \opp{0304011}{math-ph}.

\bibitem{AraRep05}
T~Arakawa.
\newblock Representation theory of superconformal algebras and the
  {Kac}--{Roan}--{Wakimoto} conjecture.
\newblock {\em Duke Math. J.}, 130:435--478, 2005.
\newblock \opp{0405015}{math-ph}.

\bibitem{AdaVer95}
D~Adamovi\'{c} and A~Milas.
\newblock Vertex operator algebras associated to modular invariant
  representations of {$A_1^{\left(1\right)}$}.
\newblock {\em Math. Res. Lett.}, 2:563--575, 1995.
\newblock \opp{9509025}{q-alg}.

\bibitem{FeiEqu98}
B~Feigin, A~Semikhatov, and I~Yu Tipunin.
\newblock Equivalence between chain categories of representations of affine $sl
  \left( 2 \right)$ and {$N = 2$} superconformal algebras.
\newblock {\em J. Math. Phys.}, 39:3865--3905, 1998.
\newblock \opp{9701043}{hep-th}.

\bibitem{RidRel15}
D~Ridout and S~Wood.
\newblock Relaxed singular vectors, {Jack} symmetric functions and fractional
  level $\widehat{\mathfrak{sl}} \left( 2 \right)$ models.
\newblock {\em Nucl. Phys.}, B894:621--664, 2015.
\newblock \pp{1501.07318}{hep-th}.

\bibitem{GabFus01}
M~Gaberdiel.
\newblock Fusion rules and logarithmic representations of a {WZW} model at
  fractional level.
\newblock {\em Nucl. Phys.}, B618:407--436, 2001.
\newblock \opp{0105046}{hep-th}.

\bibitem{RidSL210}
D~Ridout.
\newblock $\widehat{\mathfrak{sl}} \left( 2 \right)_{-1/2}$ and the triplet
  model.
\newblock {\em Nucl. Phys.}, B835:314--342, 2010.
\newblock \pp{1001.3960}{hep-th}.

\bibitem{CreMod12}
T~Creutzig and D~Ridout.
\newblock Modular data and {Verlinde} formulae for fractional level {WZW}
  models {I}.
\newblock {\em Nucl. Phys.}, B865:83--114, 2012.
\newblock \pp{1205.6513}{hep-th}.

\bibitem{CreMod13}
T~Creutzig and D~Ridout.
\newblock Modular data and {Verlinde} formulae for fractional level {WZW}
  models {II}.
\newblock {\em Nucl. Phys.}, B875:423--458, 2013.
\newblock \pp{1306.4388}{hep-th}.

\bibitem{AugMod17}
J~Auger, T~Creutzig, and D~Ridout.
\newblock Modularity of logarithmic parafermion vertex algebras.
\newblock {\em Lett. Math. Phys.}, 108:2543--2587, 2018.
\newblock \pp{1704.05168}{math.QA}.

\bibitem{AdaRea17}
D~Adamovi\'{c}.
\newblock Realizations of simple affine vertex algebras and their modules: the
  cases $\widehat{sl(2)}$ and $\widehat{osp(1,2)}$.
\newblock {\em Comm. Math. Phys.}, 366:1025--1067, 2019.
\newblock \pp{1711.11342}{math.QA}.

\bibitem{KawRel18}
K~Kawasetsu and D~Ridout.
\newblock Relaxed highest-weight modules {I}: rank $1$ cases.
\newblock {\em Comm. Math. Phys.}, 368:627--663, 2019.
\newblock \pp{1803.01989}{math.RT}.

\bibitem{AraWei16}
T~Arakawa, V~Futorny, and L~Ramirez.
\newblock Weight representations of admissible affine vertex algebras.
\newblock {\em Comm. Math. Phys.}, 353:1151--1178, 2017.
\newblock \pp{1605.07580}{math.RT}.

\bibitem{RidAdm17}
D~Ridout, J~Snadden, and S~Wood.
\newblock An admissible level $\widehat{\mathfrak{osp}} \left( 1 \middle\vert 2
  \right)$-model: modular transformations and the {Verlinde} formula.
\newblock {\em Lett. Math. Phys.}, 108:2363--2423, 2018.
\newblock \pp{1705.04006}{hep-th}.

\bibitem{WooAdm18}
S~Wood.
\newblock Admissible level $\mathfrak{osp} \left( 1 \middle\vert 2 \right)$
  minimal models and their relaxed highest weight modules.
\newblock {\em Transform. Groups}, 25:887--943, 2020.
\newblock \pp{1804.01200}{math.QA}.

\bibitem{CreCos18}
T~Creutzig, S~Kanade, T~Liu, and D~Ridout.
\newblock Cosets, characters and fusion for admissible-level
  $\mathfrak{osp}(1\vert2)$ minimal models.
\newblock {\em Nucl. Phys.}, B938:22--55, 2018.
\newblock \pp{1806.09146}{hep-th}.

\bibitem{KawRel19}
K~Kawasetsu and D~Ridout.
\newblock Relaxed highest-weight modules {II}: classifications for affine
  vertex algebras.
\newblock \pp{1906.02935}{math.RT}.

\bibitem{KawRel20}
K~Kawasetsu.
\newblock Relaxed highest-weight modules {III}: character formulae.
\newblock \pp{2003.10148}{math.RT}.

\bibitem{FutSim20}
V~Futorny, O~Morales, and L~Ramirez.
\newblock Simple modules for affine vertex algebras in the minimal nilpotent
  orbit.
\newblock \pp{2002.05568}{math.RT}.

\bibitem{FutPos20}
V~Futorny and L~K\v{r}i\v{z}ka.
\newblock Positive energy representations of affine vertex algebras.
\newblock \pp{2002.05586}{math.RT}.

\bibitem{AdaRea16}
D~Adamovi\'{c}.
\newblock A realization of certain modules for the {$N=4$} superconformal
  algebra and the affine {Lie} algebra {$A_2^{(1)}$}.
\newblock {\em Transform. Groups}, 21:299--327, 2016.
\newblock \pp{1407.1527}{math.QA}.

\bibitem{RidBos14}
D~Ridout and S~Wood.
\newblock Bosonic ghosts at $c=2$ as a logarithmic {CFT}.
\newblock {\em Lett. Math. Phys.}, 105:279--307, 2015.
\newblock \pp{1408.4185}{hep-th}.

\bibitem{FR20}
Z~Fehily and D~Ridout.
\newblock In preparation.

\bibitem{AdaRea20}
D~Adamovi\'{c}, K~Kawasetsu, and D~Ridout.
\newblock A realisation of the {Bershadsky}--{Polyakov} algebras and their
  relaxed modules.
\newblock \pp{2007.00396}{math.QA}.

\bibitem{ZamInf85}
A~Zamolodchikov.
\newblock Infinite additional symmetries in two-dimensional conformal quantum
  field theory.
\newblock {\em Theoret. and Math. Phys.}, 65:1205--1213, 1985.

\bibitem{SemInv94}
A~Semikhatov.
\newblock Inverting the {Hamiltonian} reduction in string theory.
\newblock In {\em 28th International Symposium on Particle Theory,
  Wendisch-Rietz, Germany}, pages 156--167, 1994.
\newblock \opp{9410109}{hep-th}.

\bibitem{MatCla00}
O~Mathieu.
\newblock Classification of irreducible weight modules.
\newblock {\em Ann. Inst. Fourier (Grenoble)}, 50:537--592, 2000.

\bibitem{CreLog13}
T~Creutzig and D~Ridout.
\newblock Logarithmic conformal field theory: beyond an introduction.
\newblock {\em J. Phys.}, A46:494006, 2013.
\newblock \pp{1303.0847}{hep-th}.

\bibitem{RidVer14}
D~Ridout and S~Wood.
\newblock The {Verlinde} formula in logarithmic {CFT}.
\newblock {\em J. Phys. Conf. Ser.}, 597:012065, 2015.
\newblock \pp{1409.0670}{hep-th}.

\bibitem{SmiCla90}
S~Smith.
\newblock A class of algebras similar to the enveloping algebra of $sl(2)$.
\newblock {\em Trans. Amer. Math. Soc.}, 322:285--314, 1990.

\bibitem{KacMod88}
V~Kac and M~Wakimoto.
\newblock Modular invariant representations of infinite-dimensional {Lie}
  algebras and superalgebras.
\newblock {\em Proc. Nat. Acad. Sci. USA}, 85:4956--4960, 1988.

\bibitem{AraRat16}
T~Arakawa.
\newblock Rationality of admissible affine vertex algebras in the category
  $\mathcal{O}$.
\newblock {\em Duke Math. J.}, 165:67--93, 2016.
\newblock \pp{1207.4857}{math.QA}.

\bibitem{AdaCon17}
D~Adamovi\'{c}, V~Kac, P~M\"{o}seneder Frajria, P~Papi, and O~Per\v{s}e.
\newblock Conformal embeddings of affine vertex algebras in minimal
  {W}-algebras {II}: decompositions.
\newblock {\em Jpn. J. Math.}, 12:261--315, 2017.
\newblock \pp{1604.00893}{math.RT}.

\bibitem{GorSim07}
M~Gorelik and V~Kac.
\newblock On simplicity of vacuum modules.
\newblock {\em Adv. Math.}, 211:621--677, 2007.
\newblock \opp{0606002}{math-ph}.

\bibitem{ZhuMod96}
Y~Zhu.
\newblock Modular invariance of characters of vertex operator algebras.
\newblock {\em J. Amer. Math. Soc.}, 9:237--302, 1996.

\bibitem{FeiAnn92}
B~Feigin, T~Nakanishi, and H~Ooguri.
\newblock The annihilating ideals of minimal models.
\newblock {\em Int. J. Mod. Phys.}, A7:217--238, 1992.

\bibitem{LiRep94}
H~Li.
\newblock {\em Representation theory and tensor product theory for vertex
  operator algebras}.
\newblock PhD thesis, Rutgers University, 1994.
\newblock \opp{9406211}{hep-th}.

\bibitem{KacVer94}
V~Kac and W~Wang.
\newblock Vertex operator superalgebras and their representations.
\newblock In {\em Mathematical aspects of conformal and topological field
  theories and quantum groups}, volume 175 of {\em Contemporary Mathematics},
  pages 161--191, Providence, 1994. American Mathematical Society.
\newblock \opp{9312065}{hep-th}.

\bibitem{DonTwi98}
C~Dong, H~Li, and G~Mason.
\newblock Twisted representations of vertex operator algebras.
\newblock {\em Math. Ann.}, 310:571--600, 1998.
\newblock \opp{9509005}{q-alg}.

\bibitem{BloSVir16}
O~Blondeau-Fournier, P~Mathieu, D~Ridout, and S~Wood.
\newblock Superconformal minimal models and admissible {Jack} polynomials.
\newblock {\em Adv. Math.}, 314:71--123, 2017.
\newblock \pp{1606.04187}{hep-th}.

\bibitem{DeSFin06}
A~De Sole and V~Kac.
\newblock Finite vs affine {W}-algebras.
\newblock {\em Jpn. J. Math.}, 1:137--261, 2006.
\newblock \opp{0511055}{math-ph}.

\bibitem{TjiFin92}
T~Tjin.
\newblock Finite {W}-algebras.
\newblock {\em Phys. Lett.}, B292:60--66, 1992.
\newblock \opp{9203077}{hep-th}.

\bibitem{MazLec10}
V~Mazorchuk.
\newblock {\em Lectures on {$\mathfrak{sl}_2 \left( \mathbb{C}
  \right)$}-Modules}.
\newblock Imperial College Press, London, 2010.

\bibitem{FreVer92}
I~Frenkel and Y~Zhu.
\newblock Vertex operator algebras associated to representations of affine and
  {Virasoro} algebras.
\newblock {\em Duke. Math. J.}, 66:123--168, 1992.

\bibitem{DiFCon97}
P~Di Francesco, P~Mathieu, and D~S\'{e}n\'{e}chal.
\newblock {\em Conformal Field Theory}.
\newblock Graduate Texts in Contemporary Physics. Springer-Verlag, New York,
  1997.

\bibitem{KacCla89}
V~Kac and M~Wakimoto.
\newblock Classification of modular invariant representations of affine
  algebras.
\newblock In {\em Infinite-Dimensional {Lie} Algebras and Groups
  (Luminy-Marseille, 1988)}, volume~7 of {\em Advanced Series in Mathematical
  Physics}, pages 138--177. World Scientific, New Jersey, 1989.

\bibitem{AraRat15}
T~Arakawa.
\newblock Rationality of {W}-algebras: Principal nilpotent cases.
\newblock {\em Ann. Math.}, 182:565--604, 2015.
\newblock \pp{1211.7124}{math.QA}.

\bibitem{AdaSom94}
D~Adamovi\'{c}.
\newblock Some rational vertex algebras.
\newblock {\em Glas. Mat. Ser. III}, 29:25--40, 1994.
\newblock \opp{9502015}{q-alg}.

\bibitem{LiPhy97}
H~Li.
\newblock The physics superselection principle in vertex operator algebra
  theory.
\newblock {\em J. Algebra}, 196:436--457, 1997.

\bibitem{CreCos13}
T~Creutzig, D~Ridout, and S~Wood.
\newblock Coset constructions of logarithmic $\left( 1,p \right)$-models.
\newblock {\em Lett. Math. Phys.}, 104:553--583, 2014.
\newblock \pp{1305.2665}{math.QA}.

\bibitem{AdaLat09}
D~Adamovi\'{c} and A~Milas.
\newblock Lattice construction of logarithmic modules for certain vertex
  algebras.
\newblock {\em Selecta Math. New Ser.}, 15:535--561, 2009.
\newblock \pp{0902.3417}{math.QA}.

\bibitem{NagTri11}
K~Nagatomo and A~Tsuchiya.
\newblock The triplet vertex operator algebra {$W \left( p \right)$} and the
  restricted quantum group {$\overline{U}_q \left( sl_2 \right)$} at $q =
  e^{\frac{\pi i}{p}}$.
\newblock {\em Adv. Stud. Pure Math.}, 61:1--49, 2011.
\newblock \pp{0902.4607}{math.QA}.

\bibitem{RidMod13}
D~Ridout and S~Wood.
\newblock Modular transformations and {Verlinde} formulae for logarithmic
  $\left( p_+, p_- \right)$-models.
\newblock {\em Nucl. Phys.}, B880:175--202, 2014.
\newblock \pp{1310.6479}{hep-th}.

\bibitem{RidFus10}
D~Ridout.
\newblock Fusion in fractional level $\widehat{\mathfrak{sl}} \left( 2
  \right)$-theories with $k=-\tfrac{1}{2}$.
\newblock {\em Nucl. Phys.}, B848:216--250, 2011.
\newblock \pp{1012.2905}{hep-th}.

\bibitem{AllBos20}
R~Allen and S~Wood.
\newblock Bosonic ghostbusting --- the bosonic ghost vertex algebra admits a
  logarithmic module category with rigid fusion.
\newblock \pp{2001.05986}{math.QA}.

\bibitem{CreBra17}
T~Creutzig, Y-Z Huang, and J~Yang.
\newblock Braided tensor categories of admissible modules for affine {Lie}
  algebras.
\newblock {\em Comm. Math. Phys.}, 362:827--854, 2018.
\newblock \pp{1709.01865}{math.QA}.

\bibitem{FieCom06}
P~Fiebig.
\newblock The combinatorics of category {$\mathscr{O}$} over symmetrizable
  {Kac}--{Moody} algebras.
\newblock {\em Transform. Groups}, 11:29--49, 2006.
\newblock \opp{0305378}{math.RT}.

\bibitem{KacInf90}
V~Kac.
\newblock {\em Infinite-Dimensional {Lie} Algebras}.
\newblock Cambridge University Press, Cambridge, 1990.

\bibitem{CarLie05}
R~Carter.
\newblock {\em {Lie} algebras of finite and affine type}, volume~96 of {\em
  Cambridge Studies in Advanced Mathematics}.
\newblock Cambridge University Press, Cambridge, 2005.

\bibitem{AraRat19}
T~Arakawa and J~van Ekeren.
\newblock Rationality and fusion rules of exceptional {W}-algebras.
\newblock \pp{1905.11473}{math.RT}.

\end{thebibliography}
%\bibliographystyle{unsrt}
\providecommand{\opp}[2]{\textsf{arXiv:\mbox{#2}/#1}}
\providecommand{\pp}[2]{\textsf{arXiv:#1 [\mbox{#2}]}}

\end{document}